\documentclass[a4paper]{amsart}
\usepackage[utf8]{inputenc}
\usepackage[T1]{fontenc}

\usepackage{amsthm, amsmath, amsfonts, mathrsfs, mathabx, amssymb}
\usepackage{stmaryrd}

\usepackage{mathpazo}
\usepackage{relsize}

\usepackage{times}
\usepackage{cancel}

\usepackage[pagebackref,colorlinks=true,pdfpagemode=none,urlcolor=blue,
linkcolor=blue,citecolor=blue]{hyperref}

\usepackage{scalerel} 

\usepackage{xcolor}
\usepackage{accents}

\usepackage{bbm}

\definecolor{labelkey}{gray}{.8}
\definecolor{refkey}{gray}{.8}

\definecolor{darkred}{rgb}{0.9,0.1,0.1}
\definecolor{darkgreen}{rgb}{0,0.5,0}

\newcommand{\tnorm}[1]{{\left\vert\kern-0.25ex\left\vert\kern-0.25ex\left\vert #1 
    \right\vert\kern-0.25ex\right\vert\kern-0.25ex\right\vert}}

\newtheorem{theorem}{Theorem}[section]
\newtheorem{lemma}[theorem]{Lemma}
\newtheorem{definition}[theorem]{Definition}
\newtheorem{corollary}[theorem]{Corollary}
\newtheorem{proposition}[theorem]{Proposition}

\newtheorem{assumption}[theorem]{Assumption}

\theoremstyle{remark}
\newtheorem{remark}[theorem]{Remark}

\renewenvironment{proof}[1][Proof]{ {\itshape \noindent {#1.}} }{$\Box$
\medskip}

\numberwithin{equation}{section}

\newcommand{\cal}{\mathcal}

\newcommand{\mb}[1]{{\mathbf #1}}
\newcommand{\bb}[1]{{\mathbb #1}}

\newcommand{\hsum}{\underset{k \in  \widehat{\T}_n}{\hat{ \sum}}}
\newcommand{\R}{\mathbb{R}}
\newcommand{\hT}{\widehat{\mathbb{T}}}
\newcommand{\bbT}{\mathbb{T}}
\newcommand{\Om}{\Omega}

\newcommand{\bbE}{\mathbb{E}}
\newcommand{\bbP}{\mathbb{P}}

\newcommand{\bbZ}{\mathbb{Z}}

\newcommand{\thN}{\tilde{\hat{N}}}

\newcommand{\E}{\mathbb{E}}

\newcommand{\om}{\omega}
\newcommand{\ga}{\gamma}

\newcommand{\dd}{{\rm d}}

\newcommand{\C}{\mathbb{C}}

\newcommand{\cL}{\mathcal{L}}

\newcommand{\bbR}{\mathbb{R}}
\newcommand{\eps}{\varepsilon}
\newcommand{\bD}{\mathbf{D}}
\renewcommand{\epsilon}{\varepsilon}

\newcommand{\cE}{\mathcal{E}}
\newcommand{\vphi}{\varphi}

\newcommand{\la}{\lambda}

\newcommand{\T}{\mathbb{T}}

\newcommand{\bT}{\mathbb{T}}

\newcommand{\cD}{\mathcal{D}}

\newcommand{\mL}{{\mb L}}
\newcommand{\lnm}{\llbracket}
\newcommand{\rnm}{\rrbracket}

\renewcommand{\hat}{\widehat}
\renewcommand{\bar}{\overline}
\newcommand{\mlg}{\mathlarger}

\newcommand{\lsc}{\mlg{\ll}\,}
\newcommand{\rsc}{\mlg{\gg}_\eta}

\renewcommand{\tilde}{\widetilde}

\begin{document}

\title[Hydrodynamic limit for oscillator chain]
{The hydrodynamic limit for the energy transport in a
  stochastically perturbed harmonic chain of oscillators}

\author{
  Tomasz Komorowski}\address{ Institute of Mathematics, Polish Academy
  Of Sciences, Warsaw, Poland.}
    \email{tkomorowski@impan.pl}
  
\author{  Stefano Olla}\address{Universit\'e Paris-Dauphine, PSL Research University,
    CNRS, CEREMADE, 
    \emph{and}  Institut Universitaire de France \emph{and} GSSI, L'Aquila} 
  \email{olla@ceremade.dauphine.fr}
  
  \author{ Marielle Simon}
  \address{Aix-Marseille Université, CNRS, Institut de Mathématiques de Marseille (I2M), 3 place Victor Hugo, 13331 Marseille Cedex 3
    \emph{and} Institut Universitaire de France \emph{and} GSSI, L'Aquila}
    \email{marielle.simon@univ-amu.fr}

     \thanks{T.K acknowledges the support of
the NCN grant 2024/53/B/ST1/00286. M.S. is partially supported by the ANR grant CONVIVIALITY (ANR-23-CE40-0003) of the French National Research Agency
(ANR)}

\begin{abstract}
We consider a pinned harmonic chain perturbed by random velocity flips, which conserve the total energy. Previous results have established the diffusive hydrodynamic limit for the energy profile at the level of expectations. In this article, under suitable fourth-moment bounds on the initial data, we prove for the first time a law of large numbers for the empirical energy distribution. Our approach is based on the Wigner distribution and shows that its random fluctuations vanish in the macroscopic limit. Consequently, the empirical energy profile converges in probability to the deterministic solution of the corresponding heat equation.
  \end{abstract}

  \maketitle

\section{Introduction}

Over the last four decades, remarkable progress has been made in deriving macroscopic evolution equations for conserved quantities directly from microscopic dynamics through space-time scaling limits, in particular for stochastic interacting systems. Among the most powerful tools are entropy methods, introduced in the seminal work of Guo, Papanicolaou, and Varadhan \cite{gpv}; see also \cite{kl} for a systematic exposition.

A serious limitation of entropy methods is that they require sufficient integrability of the microscopic currents of the conserved quantities. More precisely, in order to perform the necessary cutoff arguments, one typically needs higher moments of the currents to be controlled in terms of the entropy. This becomes problematic for chains of oscillators whose Hamiltonian dynamics is perturbed by a conservative stochastic dynamics preserving the energy. In such systems, the energy current typically grows, at large energies, at the same order as the local energy itself, and the entropy bound does not provide the additional moment estimates required by the usual argument.

At the Euler, or hyperbolic, scale this difficulty can be circumvented \cite{even-olla} by using Yau's relative entropy method \cite{yau}. At the diffusive scale, however, where one expects the heat equation to describe the macroscopic transport of energy, the corresponding problem remains open in considerable generality. For typical anharmonic chains, the energy current is \emph{non-gradient}, and at present only equilibrium energy fluctuations can be treated \cite{os}, through suitable adaptations of Varadhan's non-gradient method; see \cite{var-ng,kl}.

For harmonic chains perturbed by an energy-conserving noise, more
explicit computations are available. In this setting, the heat
equation can be derived for the macroscopic evolution of the energy in
the diffusive limit at the level of the \emph{expected} empirical energy
profile \cite{KOS18,klo,klos}. One particularly simple
energy-conserving stochastic perturbation consists in independently
flipping the signs of the particle velocities at random times. More
recently, analogous results have been obtained for certain
deterministic chaotic perturbations \cite{clo26}.

We mention here a   result of \cite{bernardin}, where the convergence in probability
  (the hydrodynamic limit) {has been proved} for  an unpinned
  harmonic chain with a stochastic exchange of kinetic energy between neighboring atoms.
  The hydrodynamic limit is then shown for a particular value of
the strength of the inter-particle noise $\gamma=\gamma_*>0$. For this
  value   of the parameter (in fact $\gamma_*=1$ in \cite{bernardin}),
which is very particular to the system considered, the problem of
large energies mentioned in the foregoing, can be
avoided. What is quite remarkable, even a slight modification of the system by
perturbing the  coupling parameter $\gamma$ between the Hamiltonian and the
noise term, destroys this feature and the argument used in
\cite{bernardin} is no longer applicable.

The purpose of the present article is to strengthen these results by
proving a hydrodynamic limit, formulated as a \emph{law of large numbers}, for the empirical energy
distribution of such harmonic chains. In particular, while the
previous results concern the convergence of the expectation of the
empirical energy profile, or the hydrodynamic limit for a particular
model of a stochasticaly perturbed harmonic chain, here we prove
convergence in probability (in fact in the $\mL^1$ sense) of the random empirical profile itself for
a quite general class of chains.

Our analysis is based on the Wigner distribution approach developed in
\cite{KOS18}. Assuming a bound on the fourth moment of the Fourier
transform of the initial configuration, see \eqref{eq:3z}, we prove that
the second moment of the Laplace transform of the fluctuating part of
the random Wigner distribution vanishes in the macroscopic limit; see
Theorem \ref{thm:intermediate}. It follows that the empirical energy
distribution converges in the $\mL^1$ sense to a deterministic limit, as
stated in Theorem \ref{thm021103-26}. This upgrades the previously known convergence at the level of expectations to a law of large numbers.

In the present article we restrict ourselves to the pinned harmonic
chain, corresponding to $\omega_0\neq0$ in \eqref{eq:dynamics0}. We
expect that the argument can be extended to the unpinned case, where
the volume is also conserved, as in \cite{KOS18}, as well as to other exchanges of velocities
that conserve only energy either random \cite{bo} or chaotic \cite{clo26}.
Another natural direction is to consider stochastic exchanges of velocities that conserve
also total momentum. In that case the macroscopic energy transport is superdiffusive and is governed by a fractional heat equation of order 3/4 \cite{JKO15}.

\subsection*{Organization of the paper}

In Sections \ref{sec2.1}--\ref{sec2.3} we formulate rigorously the
  model. The main result is stated in Section \ref{sec2.4}, see
  Theorem \ref{thm021103-26}. In Section \ref{sec:strategy} we
  reformulate the question of the convergence of the energy functional
  into the   convergence of the respective Wigner functions, see
  Theorem \ref{thm:intermediate}. In fact,   the latter,  is a consequence
  of the convergence of the  Laplace transforms of the Fourier-Wigner
functions, defined in Section \ref{sec3.1}, see Theorem \ref{thm:laplace}. Section \ref{sec:aux} is
devoted to proving some estimates of the Fourier-Wigner functions that
shall be used in the proof of Theorem \ref{thm:laplace},  given
in Sections \ref{sec:proofLaplace} and \ref{sec:proofrest}. Some
additional facts, dealing with the examples of the initial data, convergence of the Wigner
functions and their Laplace transforms, are proved in Sections
\ref{AppA} and \ref{secB} of the
Appendix. For the sake of  
completeness of the argument, in Section \ref{sec5},  we give the proof of  
  the convergence of the averages of   the Wigner functions. This result
plays a role in the proof of Theorem \ref{thm:laplace}. It
relies on the technique developed in \cite{KOS18}, that was used to
show the diffusive limit of the averages of the Wigner functions in
the case of the unpinned chain $(\om_0=0)$.


\section{Preliminaries and formulation of the main result}

\label{sec2}

\subsection{Random flip model}

\label{sec2.1}

Denote by $\T_n$ the set of integers $\bbZ$ equipped with the
identification of elements modulo an integer element $n\ge 1$.
We consider $n$ particles whose positions and  momenta
$\{\big(q_x(t),p_x(t)\big)\}_{x\in\T_n}$ evolve according to the stochastic
differential equations
\begin{equation} \left\{ \begin{aligned}
    \dd   q_x(t) &=   p_x(t)\; \dd   t ,\\
    \dd   p_x(t) &=  \big(\Delta q_x(t)-\om_0^2q_x(t)\big)\; \dd   t - 2p_x(t^-)
    \; \dd     N_x(\gamma  t),\qquad x\in \T_n.
  \end{aligned}\right.\label{eq:dynamics0}\end{equation}
Here $\om_0 > 0$ is the \textit{pinning intensity}, $\{  N_x(t)\; ; \; t \geq 0\}_{x\in\T_n}$ are $n$ independent Poisson processes of
intensity $1$, and the constant $\gamma>0$ is the \emph{noise intensity}. We suppose that
they are defined over a certain probability space $(\Xi,{\cal
  F},\bbP)$. Besides, given a function $f:\T_n\to\mathbb C$ the right and left gradients and laplacian discrete operators are defined as
$$
\nabla  f_x=f_{x+1}-f_{x},\quad  \nabla^\star f_x=f_x-f_{x-1},\quad
\Delta f_x:=\nabla^\star\nabla  f_x=f_{x+1}+f_{x-1}-2f_{x}.
$$
We denote by
$\Om_n=(\bbR\times \bbR)^{\T_n}$ the configuration space of the
particle ensemble.


\subsection{Fourier transform and wave function} 

Let us denote by $\widehat{f}$ the Fourier transform of a finite sequence $\{f_x\}_{x\in\T_n}$ of numbers in $\C$,
 defined as follows: 
\begin{equation}\label{eq:fourier}
\widehat{f}(k)=\sum_{x\in\T_n} f_x\; e^{-2i\pi x k},\qquad k\in \widehat{\T}_n:=\big\{0,\tfrac1n,...,\tfrac{n-1}n\big\}.
\end{equation}
Here
$\widehat{\T}_n:=\big\{0,\tfrac1n,...,\tfrac{n-1}n\big\}$ is the dual
to the discrete torus $\T_n$. It has a group structure with respect to
the addition modulo $1$.

Reciprocally, for any $f: \widehat{\T}_n \to \C$, we denote by $\big\{\widecheck{f}_x\big\}_{x\in\T_n}$ its inverse Fourier transform given by
\begin{equation}
\label{eq:inversefouri}
\begin{split}
 \widecheck{f}_x = \underset{k \in  \widehat{\T}_n}{\widehat{\sum}} e^{2i\pi xk}
f(k),\qquad x \in \T_n, \quad\mbox{where}\quad
    \underset{k \in  \widehat{\T}_n}{\widehat\sum}  :=\frac1n\sum_{k \in  \widehat{\T}_n} .
\end{split}
\end{equation}
The Parseval identity reads 
\begin{equation}\label{eq:parseval}
\|\hat f\|_{\mb L^2(\hat \T_n)}^2:= \underset{k \in  \widehat{\T}_n}{\hat\sum}  \big|\widehat{f}(k)\big|^2=\sum_{x\in\T_n} \big|f_x\big|^2 = \|f\|^2_{\mb L^2(\T_n)}.
\end{equation}
 If $\{f_x\}_{x\in\T_n}$ and $\{g_x\}_{x\in\T_n}$ are two  sequences indexed by the discrete torus, their convolution is given by
\[ (f\ast g)_x:= \sum_{y \in \T_n} f_y\; g_{x-y},\qquad x \in \T_n.
\]
Denote by $\bT$ the unit torus in dimension $1$, understood here as
$[0,1]_{\equiv}$, with $\equiv$  the relation of identification of the
endpoints. {Suppose that $\mathbf{X}$ is some set. Given a function $\varphi: \bT\times
\mathbf{X}\to\mathbb C$, such that $|\varphi(\cdot,\xi)|$  is Lebesgue
integrable, for each $\xi\in \mathbf{X}$,} we define its Fourier
transform w.r.t.~the first variable $u\in\T$ by
\begin{equation}
  \label{FC}
  {\cal F}(\varphi)(\eta,\xi):=\int_{\bT}\varphi(u,\xi)e^{-2\pi iu\eta}\dd
  u,\qquad \eta\in\bbZ, \xi \in \mathbf{X}.
  \end{equation}
\begin{definition}[Dispersion relation and wave function] \label{def:dispersion}
We  define the \emph{dispersion relation}
\begin{equation}\label{eq:omega}
        \om(k)=\big(\om_0^2+4\sin^2(\pi k)\big)^{1/2},\qquad k\in\bT.\end{equation}
Then, the \emph{wave function}  is defined as
\begin{equation}
	\label{011307f}
	\hat\psi(t,k):= \om(k)\hat q(t,k)
	+i  \hat p(t,k),\qquad k\in\hat\bbT_n,
\end{equation}
where $\{\hat q(t,k)\}_{k\in\hat\T_n}$ are $\{\hat p(t,k)\}_{k\in\hat\T_n}$ are the Fourier coefficients of $\{p_x(t)\}_{x\in\T_n}$ and $\{q_x(t)\}_{x\in\T_n}$, respectively.
\end{definition}

\begin{remark}
	The inverse  Fourier transform of the wave function is given by
	\begin{equation}
		\label{011307}
		\psi_x(t):= \left(\tilde{\om} *q\left( t\right)\right)_x 
		+i  p_x\left( t\right),\qquad x\in\bbT_n,
	\end{equation}
	where $\{\tilde \om_x\}_{x\in\T_n}$ are the Fourier inverse coefficients of the dispersion
	relation \eqref{eq:omega}. 
\end{remark}


\subsection{Energy functional and the
  distribution of the initial data}

\label{sec2.3}

Given a configuration $({\bf q},{\bf  p})=\{(q_x,p_x)\}_{x\in\T_n}\in\Om_n$ we  define the \textit{microscopic
energy density} (energy per atom) functional
\begin{equation}
  \label{ex}
  e_x:=\frac12 \big(p_x^2+\om_0^2q_x^2+\big(\nabla^\star
  q_x\big)^2\big),\qquad x \in  \T_n.
\end{equation}
Let $T(\cdot)\in C(\bT)$ be a continuous, positive function and
  \begin{equation}
    \label{Ts}
   0< T_*:=\min_{u\in\bT}T(u) \le  T^*:=\max_{u\in\bT}T(u).
 \end{equation}
 {We assume that the initial probability distribution $\mu_n$ on
$\Omega_n$ satisfies the following:}
\begin{assumption}
  \label{T}
 For any $\varphi\in C(\bT)$, we have
 \begin{equation}
   \label{eq:1}
    \lim_{n\to+\infty}\int_{\Om_n}\Big|\frac{1}{n}\sum_{x\in\bT_n}e_x
      \varphi\Big(\frac{x}{n}\Big)-\int_{\bbT} \varphi(u) T(u) \dd u\Big|\dd \mu_{n}=0.
    \end{equation}
   { Furthermore, we assume that
      there exists $C>0$ such that }
    \begin{align}
      \label{eq:3z}
       \frac1{n^2}\hsum \int_{\Om_n}\big|\hat\psi(k)\big|^4 \;
      \dd\mu_n \le C,\qquad n=1,2,\ldots.
    \end{align}
\end{assumption}


\begin{remark}
  An example of sequence of probability distributions that satisfy Assumption \ref{T}
is given by the local Gibbs measures
\begin{equation}
  \label{eq:localG}
  \mu_{n,T(\cdot)} (\dd{\bf q}, \dd{\bf p})=
  \frac{1}{Z_{n,T(\cdot)}}\exp\bigg\{-\sum_{x\in\bT_n} T^{-1}\Big(\frac{x}{n}\Big)e_x\bigg\}
  \dd{\bf q}\dd{\bf p},
\end{equation}
where $Z_{n,T(\cdot)}$ is the normalizing factor, ${\bf q}=(q_x)_{x\in\T_n}, {\bf p}=(p_x)_{x\in\T_n}$.
This is proved in detail in Appendix \ref{app:lln}.
\end{remark}
        
        \begin{remark}
          \label{rmk2.4}
          The initial distribution $\mu_n$ can be also deterministic, i.e.~concentrated on
          initial (sequence of) configurations that satisfy Assumption \ref{T}.
          Any randomness of the initial configuration does not play any role in the proof.
          An example of such sequence of configurations is shown in Appendix \ref{xxx}.
        \end{remark}


\subsection{Statements of the main results}
\label{sec2.4}

We consider the evolution of positions and momenta in the diffusive time scale, namely
$q_x^{(n)}(t)=q(n^2t)$, $p_x^{(n)}(t)=p(n^2t)$, $x\in\T_n$, which satisfy, for any $x \in \T_n$,
\begin{equation} \left\{ \begin{aligned}
    \dd   q_x^{(n)}(t) &=  n^2 p_x^{(n)}(t)\; \dd   t ,\\
    \dd   p_x^{(n)}(t) &=  n^2\big(\Delta q_x^{(n)}(t)-\om_0^2q_x^{(n)}(t)\big)\; \dd   t - 2p_x^{(n)}(t^-)
    \; \dd     N_x^{(n)}(\gamma  t),
  \end{aligned}\right.\label{eq:dynamicsn}\end{equation}
where $N_x^{(n)}(\gamma  t) :=N_x(n^2\gamma  t)$.
In the following we assume that $\{\big(q_x^{(n)}(0), p_x^{(n)}(0)\big)\}_{x\in\T_n}$ are distributed according to $\mu_n$, which satisfies Assumption \ref{T} and we denote 
\begin{equation}
  \label{ex1}
  e_x^{(n)}(t):=\frac12 \big(p_x^2(n^2t)+\om_0^2q_x^2(n^2t)+(\nabla^\star
  q_x(n^2t))^2\big),\qquad x\in\bbT_n.
\end{equation}
The process $\big(q_x^{(n)}(t),p_x^{(n)}(t)\big)_{x\in\T_n}$, $t\ge0$  is defined over the product probability space $(\Om_n\times
\Xi,{\cal B}(\Om_n)\otimes {\cal F},\bbP_n)$, where $\bbP_n:=\mu_{n}\otimes\bbP$.
The \textit{mean energy} functional is then defined as 
\begin{equation}
  \label{bex1}
\widebar  e_x^{(n)}(t):=  \bbE_{n}\big[e_x^{(n)}(t)\big],
\end{equation}
with $\bbE_n$ the expectation w.r.t.~the measure $\bbP_n$. One can easily see that the total energy is conserved almost surely along time evolution, namely
\begin{equation}
  \label{ex2}
 \mathcal{E}_{\rm tot}^{(n)}(t):=\frac1n\sum_{x\in\bbT_n} e_x^{(n)}(t)\equiv \frac1n \sum_{x\in\bbT_n}
 e_x^{(n)}(0) = \mathcal{E}_{\rm tot}^{(n)}(0),\qquad t\ge0.
\end{equation}
Note that, from Assumption \ref{T} and thanks to the conservation law: for any $t\ge 0$,
	\begin{equation}\label{eq:energybounds} \sup_{n\ge 1} \E_n\big[ \mathcal{E}_{\rm tot}^{(n)}(t) \big] < \infty, \qquad \sup_{n\ge 1} \E_n\big[(\mathcal{E}_{\rm tot}^{(n)}(t))^2 \big] < \infty.\end{equation}Indeed, by the Jensen inequality and the Plancherel identity \eqref{eq:parseval}, we have: for any $t\ge 0$,
	\begin{align*}
		\frac1{n^2}
		\int_{\Om_n}\hsum\big|\hat\psi(k)\big|^4\dd\mu_n &\ge
		\int_{\Om_n}  \Big(\frac1{n}\hsum\big|\hat\psi(k)\big|^2\Big)^2\dd\mu_n
           =\int_{\Om_n}\Big(\frac1{n}\sum_{x\in\bbT_n}\big|\psi_x\big|^2\Big)^2\dd\mu_n\\
       &   { =\int_{\Om_n}\Big(\frac2{n}\sum_{x\in\bbT_n} e_x(0)\Big)^2\dd\mu_n
           = \mathbb E_{n}\left[\Big(\frac2{n}\sum_{x\in\bbT_n} e_x(t)\Big)^2\right].}
		\end{align*}
\medskip

In the following we denote
by $C_c(\mathbf{X})$ (resp.~$C_c^\infty(\mathbf{X})$) the space of continuous (resp. smooth) and compactly
supported  functions on a given metric space $(\mathbf{X},{\rm d})$. By $\|\cdot\|_\infty$ we denote the supremum
norm on the space.
Our first result deals with the convergence of the mean energy profile. 

\begin{theorem}
  \label{thm022303-26}
 For any  function $\Phi\in C_c\big(\T\times[0,+\infty)\big)$ 
  \begin{equation}
      \label{022303-26}
      \lim_{n\to+\infty} \frac{1}{n}\sum_{x\in\bT_n}\int_0^{+\infty}\bar
      e_x^{(n)}(t)\Phi\Big(\frac{x}{n},t\Big)\dd t=\int_0^{+\infty}\dd
      t\int_{\bT}T(u,t)\Phi(u,t)\dd u,
    \end{equation}
    where $T(u,t)$ is the unique solution of the following initial
    value problem
    \begin{equation}
      \label{032303-26}
      \begin{split}
        \partial_t T(u,t)&=\bD\partial_{uu}T(u,t),\qquad (u,t)\in\T\times[0,+\infty),\\
T(u,0)&=T(u),\qquad u\in\bbT.
      \end{split}
    \end{equation}
    Here
    \begin{equation}
      \label{042303-26}
        \bD=\frac{1}{4\pi^2\ga}\int_{\bbT}[\om'(k)]^2\dd k\end{equation}
        and $\om$ is the dispersion relation given in \eqref{eq:omega}.
    \end{theorem}
   The proof of this result uses the technique developed in
    \cite{KOS18}, that deals with the unpinned case ($\om_0=0$). For
    the sake of completeness, we present the
    argument adapted to the pinned case, considered here,  in Section \ref{sec5}.

   Our next result deals with the fluctuations of the random field \eqref{ex1}.
\begin{theorem}
  \label{thm021103-26}
  For any   function $\Phi\in C_c\big(\T\times[0,+\infty)\big)$  we have
  \begin{equation}
      \label{011103-26}
      \lim_{n\to+\infty}\bbE_n\bigg[\Big|\frac{1}{n}\sum_{x\in\bT_n}\int_0^{+\infty}\big[e_x^{(n)}(t)-\bar
      e_x^{(n)}(t)\big]\Phi\Big(\frac{x}{n},t\Big)\dd t \Big|\bigg]=0.
    \end{equation}
    \end{theorem}
 The main steps of the proof of Theorem \ref{thm021103-26} are given in Section \ref{sec:strategy} and the proof is concluded in Section \ref{sec:proofLaplace}.

As a conclusion from Theorems \ref{thm022303-26} and \ref{thm021103-26} we obtain the following.
\begin{corollary}
  \label{thm021103-26b}
  For any   function $\Phi\in C_c\big(\T\times[0,+\infty)\big)$  we have
  \begin{equation}
      \label{011103-26bis}
      \lim_{n\to+\infty}\bbE_n\bigg[\Big|
      \frac{1}{n}\sum_{x\in\bT_n}\int_0^{+\infty}e_x^{(n)}(t)\Phi\Big(\frac{x}{n},t\Big)\dd
      t -\int_0^{+\infty}\dd t\int_{\bT}T(u,t)\Phi(u,t)\dd u\Big|\bigg]=0,
    \end{equation}
    where $T(u,t)$ is as in the statement of Theorem \ref{thm022303-26}.
    \end{corollary}

\section{Strategy of the proof}
\label{sec:strategy}

\subsection{Link between convergence of energy and Wigner functions}
\label{sec3.1}

Recall the wave function which has been defined in
\eqref{011307f}. In consequence of the Parseval identity, we have, for any $t\ge 0$,
\begin{align}
  \label{psi-e1}
  \hsum |\hat\psi(t,k)|^2=\sum_{x\in\bT_n}|\psi_x(t)|^2=2 \sum_{x\in\bT_n}e_x(t).
  \end{align}
Next, we introduce the following important tool: 
 \begin{definition} \label{def:wigner} The scaled \emph{Fourier-Wigner
     function} $\mathbf{W}_n:[0,+\infty)\times \bbZ \times \hat\T_n
   \to\mathbb C^4$ is given by  
\begin{equation}
	\label{012401-26}
	\mathbf{W}_n(t,\eta,k):=\begin{bmatrix}
		W_{n,+}(t,\eta,k)  \\
		W_{n,-}(t,\eta,k)  \\
		Y_{n,+}(t,\eta,k)    \\
		Y_{n,-}(t,\eta,k)       
	\end{bmatrix}, \qquad (t,\eta,k)\in [0,+\infty)\times\bbZ \times \widehat{\bbT}_n,
\end{equation}
where
\begin{equation}
	\label{022210-25}
	\begin{split}
		W_{n,+}(t,\eta,k)& :=\frac{1}{2n}  
		\widehat\psi\Big(n^2t,k+\frac{\eta}{n}\Big)
		\widehat\psi^\star(n^2t,k)  ,\\
		 Y_{n,+}(t,\eta,k)& :=\frac{1}{2n}  
		\widehat\psi\Big(n^2t,k+\frac{\eta}{n}\Big)
		\widehat\psi (n^2t,-k)  ,\\ 
		 Y_{n,-}(t,\eta,k) &:=\frac{1}{2n}  
		\widehat\psi^\star\Big(n^2t,-k-\frac{\eta}{n}\Big)
		\widehat\psi ^\star(n^2t,k)
		, \\
		 W_{n,-}(t,\eta,k) &:=\frac{1}{2n}  
		\widehat\psi^\star\Big(n^2t,-k-\frac{\eta}{n}\Big)
		\widehat\psi(n^2t,-k)    .
	\end{split}
\end{equation}
\end{definition}

{For the brevity sake, sometimes we simply denote $
W_{n,+}(t)\equiv W_{n,+}(t,\cdot,\cdot)$ and
we may omit writing $(\eta,k)$ (or sometimes only $k$) in the notation for similar quantities.}

Similarly to \eqref{bex1}, we define 
$ \bar W_{n,+}(t):=\E_n\big[W_{n,+}(t)\big]$.
We also define the fluctuating fields, both for the energy density
and Wigner function $W_{n,+}(t)$ by letting  
\begin{align}
  \label{teW}
	\tilde e_x^{(n)}(t) & := e_x^{(n)} (t)- \bar e_x^{(n)}(t),\notag\\
	\tilde W_{n,+}(t) & :=   W_{n,+}(t) -  \bar W_{n,+}(t).
\end{align}
For
any  $\Phi \in C_c^\infty\big(\T\times[0,+\infty)\big)  $, let us  define the functionals
\begin{align*}
	\tilde{\cal E}_n[\Phi]&:=\frac1n \sum_{x\in\T_n} \int_0^{+\infty} \tilde e_x^{(n)}(t) \Phi\Big(\frac x n,t\Big)\dd t, 
	\\
	\tilde{\cal W}_n[\Phi]&:= \sum_{\eta \in \bbZ} \hsum \int_0^{+\infty} \tilde W_{n,+}(t,\eta,k) {\cal F}(\Phi)^\star(\eta,t) \dd t.
\end{align*}
{Since $\Phi\in
C_c^\infty\big(\T\times[0,+\infty))$, for   any $p>0$, 
\[\sum_{\eta \in \bbZ} \int_0^{+\infty}(1+\eta^2)^p
\big|{\cal F}(\Phi)\big|(\eta,t)\dd t< +\infty\] hence $\tilde{\cal W}_n$ is well defined.}
It is well-known that the Fourier-Wigner functionals $\tilde{{\cal W}}_n$
  is closely related to the energy field. Indeed, thanks
  to Corollary \ref{cor032603-26} and energy conservation \eqref{ex2},  there exists a constant $C>0$ such that
 \begin{multline*}
 	 \E_n \bigg[\Big| \int_0^{+\infty}  \Big\{  \sum_{x\in\bT_n } \hsum
 W_{n,+}(t,x,k) \Phi\Big(\frac{x}{n},t\Big)-\frac{1}{n}\sum_{x\in\bT_n
                  } e_x(t) \Phi\Big(\frac{x}{n},t\Big)\Big\} \dd t\Big|\bigg]\\
     \le {\E_n\big[\mathcal{E}_{\rm tot}^{(n)}(t)\big]} \times 
  \frac{C}{n}\int_0^{+\infty}\Big(\|\partial_u\Phi(t) \|_\infty+\sum_{\eta\in\bbZ}\big|\eta{\cal
    F}\big(\Phi(t)\big)(\eta)\big|\Big)\dd t.
 \end{multline*}
In particular, from this and   \eqref{eq:1}, we conclude that  for any $\Phi \in C_c^{\infty}
\big(\T\times[0,+\infty)\big)$ 
\begin{equation}
	\label{eq:wigner_energy}
\lim_{n\to+\infty}	\E_n \Big[ \big| \tilde{\cal W}_n[\Phi] - \tilde\cE_n[\Phi] \big| \Big] = 0.
\end{equation}
Using the density argument and the energy bound \eqref{eq:energybounds}, we conclude that \eqref{eq:wigner_energy} holds
for any $\Phi \in C_c 
\big(\T\times[0,+\infty)\big)$.

From \eqref{eq:wigner_energy}, we see that in order to prove Theorem \ref{thm021103-26}, it is enough to show that $\E_n[| \tilde{\cal W}_n[\Phi]|]$ vanishes as $n\to +\infty$. The proof of Theorem \ref{thm021103-26} becomes a direct consequence of \eqref{eq:wigner_energy} and the following:

\begin{theorem}[Limit of the Wigner functionals] \label{thm:intermediate}
	Let $\varphi \in C_c([0,+\infty))$. For any $\eta \in \bbZ$, 
	\begin{equation} \label{eq:mainconv}\lim_{n\to +\infty} \E_n \bigg[ \Big| \hsum \int_0^{+\infty} \vphi(t) \tilde W_{n,+}(t,\eta,k) \dd t \Big| \bigg] = 0.
	\end{equation}
      \end{theorem}
{Indeed, from Theorem \ref{thm:intermediate} we conclude that Theorem
\ref{thm021103-26} holds for any  $\Phi$ that is
a finite linear combination of
functions of the form $\vphi(t;\eta)e^{2\pi i \eta u}$, $\eta\in\bbZ$, where
$\vphi(\cdot;\eta) \in C_c([0,+\infty))$. By the Stone-Weierstrass Theorem, such
functions form a subset, that is   dense in the supremum norm in 
$C_c 
\big(\T\times[0,+\infty)\big)$. This and the energy conservation
property \eqref{ex2}, in turn imply Theorem
\ref{thm021103-26}.}

      \begin{remark}
	In this section the other Wigner functions $W_{n,-},
        Y_{n,\pm}$ do not appear in the statements of the results,
        however they will play an important role in our subsequent argument.
	\end{remark}



The convergence result stated in Theorem \ref{thm:intermediate} above will be proved using \emph{Laplace transforms}. Namely, let us define, for any $\lambda \in\C$ such that $\mathrm{Re}\,\lambda>0$, 
\begin{equation}\hat w_{n,+}(\lambda,\eta,k):=\int_0^{+\infty} e^{-\lambda t} \tilde W_{n,+}(t,\eta,k)\dd t, \qquad (\eta,k) \in \bbZ\times \hat\T_n. \label{eq:lapw}\end{equation}
It will be a consequence of the forthcoming Theorem \ref{thm011002-26} below that this Laplace transform is well defined $\bbP_n$ a.s.

We claim that Theorem \ref{thm:intermediate} is a consequence of the following result: 

\begin{theorem}\label{thm:laplace}
For any $\lambda\in\C$ such that $\mathrm{Re}\,\lambda >0$ and any $\eta \in\bb Z$,
\[\lim_{n\to+\infty} \E_n \Big[ \big\| \hat w_{n,+}(\lambda,\eta,\cdot)\big\|_{{\mb L}^2(\hat \T_n)}\Big] = 0,\]
{where $\|\cdot\|_{\mL^2(\hT_n)}$ has been defined in \eqref{eq:parseval}.}
      \end{theorem}
   The proof of the fact that Theorem \ref{thm:laplace}
        implies Theorem \ref{thm:intermediate}  is  standard and it is
        presented  in Appendix \ref{AppC}.

	To summarize, the conclusion of  Theorem \ref{thm021103-26} is a
        consequence of \eqref{eq:wigner_energy} and   Theorem
        \ref{thm:laplace}. The latter is shown,
          in Sections \ref{sec:proofLaplace} and \ref{sec:proofrest}.

\section{Auxiliaries}
    \label{sec:aux}

\subsection{Dynamics of the scaled Fourier-Wigner functions}

Recall the scaled Fourier-Wigner functions from Definition \ref{def:wigner} and the dispersion relation $\omega$ given in Definition \ref{def:dispersion}. Let us introduce some useful operators: for any $(\eta,k) \in \bbZ \times \hat{\T}_n$, and $f$ defined on $\hat\T_n$,
\begin{equation}\label{eq:not} \begin{split}
   &
\delta_{n}\om(\eta,k):=n\Big(\om\big(k+\frac{\eta}{n}\big)-\om(k)\Big), \qquad 
 \bar\om_n(\eta,k)
        :=\frac12\Big(\om(k)+\om\big(k+\frac{\eta}{n}\big)\Big),\vphantom{\Bigg(}\\
  &
    {\cal R}_nf:=\underset{\ell \in \hat\T_n}{\hat\sum}\; f(\ell), \;\qquad
         {\cal L}_nf(k):= {\cal R}_nf- f(k)=\underset{\ell \in \hat\T_n}{\hat{\sum}}\big[f(\ell)-f(k)\big],
\end{split} \end{equation}
and finally the Dirichlet form
\begin{equation}
	\label{eq:cD} \cD_n(f):= 2 \big\langle -{\cal L}_n f, f \big\rangle_{{\mb L}^2(\hat\T_n)} = \underset{k,\ell\in \hat\T_n}{\hat\sum} \big|f(k)-f(\ell)\big|^2.
\end{equation}
Since $-{\cal L}_n$ is the projection on an orthogonal complement to
constants in ${\mb L}^2(\hT_n)$, we conclude ${\cal L}_n^2f= -{\cal L}_nf$ and 
\begin{equation}
  \label{030204-26}
  \cD_n(f)=\|{\cal L}_nf\|^2_{{\mb L}^2(\hT_n)}\le \|
   f\|_{{\mb L}^2(\hT_n)}^2 
  \end{equation}
  and moreover
  \begin{equation}
  	\label{eq:estimdir} \cD_n(f+g) \le 2\cD_n(f)  + 2\cD_n(g).
  \end{equation}
Finally, define also the Fourier transform of the Poisson noise and the respective
martingales obtained by compensation: for any $t\ge 0$ and $k \in\hat\T_n$,
\begin{align*}
  &
    \hat{ N} (t,k)=\sum_{x\in\bbT_n}e^{-2\pi i
      kx}  N_x(t)\qquad \mbox{and}\qquad 
    \tilde{\hat{ N}} (t,k)=\sum_{x\in\bbT_n}e^{-2\pi i
      kx} \tilde{N}_x(t),
\end{align*}
where 
$\tilde{N}_x(t):=N_x(t)- t$. Furthermore, let us introduce the following notations: for $(t,\eta,k) \in \R_+ \times \bbZ\times \hat\T_n$,
\begin{equation}
  \label{020204-26}
  \begin{split} 
 U_{n,+}(t,\eta,k)&:=\frac1{2}\big(  Y_{n,+}(t,\eta,k)  + 
     Y_{n,-}(t,\eta,k)\big),\\
  U_{n,-}(t,\eta,k)&:=\frac1{2i}\big(   Y_{n,+}(t,\eta,k)  - 
     Y_{n,-}(t,\eta,k)\big),\\
  W_{n,0}(t,\eta,k)&:=\frac1{2}\big(   W_{n,+}(t,\eta,k)  +
    W_{n,-}(t,\eta,k)\big), \\
  \delta W_{n}(t,\eta,k)&:=\frac1{2}\big(   W_{n,+}(t,\eta,k)  -
     W_{n,-}(t,\eta,k)\big). \end{split}
     \end{equation}
From the  dynamics described in \eqref{eq:dynamicsn},  tedious but straightforward computations show that the scaled Fourier-Wigner
functions satisfy the following closed system of stochastic
differential equations: for any $(t,\eta,k) \in \R_+ \times \bbZ\times \hat\T_n$,
\begin{align}
  \dd W_{n,+}(t,\eta,k)      
   = & \Big[-in \delta_{n}\om(\eta,k)    W_{n,+}(t,\eta,k)    +2\ga n^2 {\cal     L}_n\big(  W_{n,+}-  U_{n,+}\big) (t,\eta,k)\Big] \dd t\notag\\
  &
    +  \underset{\ell'  \in \hat\T_n}{\hat{\sum}}  \Big[
      \Big( {\cal     L}_n\big(  W_{n,+} -  U_{n,+}\big)  -i U_{n,-}  \Big)(t-,\eta-n\ell',k)     \notag
 \\
 & \quad + \Big( {\cal
    L}_n\big(  W_{n,+} -  U_{n,+} \big)  +i U_{n,-}  \Big)  (t-,\eta-n\ell',k+\ell')
          \Big]\dd\thN  (\gamma n^2 t,\ell'), \label{012403-26} \end{align}
          and 
          \begin{align}
      \dd W_{n,-}(t,\eta,k)    = & \Big[in \delta_{n}\om(\eta,k)   W_{n,-}(t,\eta,k)    
          +2\ga n^2{\cal L}_n\big(   W_{n,-}-  U_{n,+}
   \big)(t,\eta,k)   \Big] \dd t  \notag
 \\
  &
    +\underset{\ell'  \in \hat\T_n}{\hat{\sum}} \Big[
    \Big({\cal     L}_n\big(  W_{n,-} -  U_{n,+}\big) +i U_{n,-}  \Big) (t-,\eta-n\ell',k) 
  \notag \\
  &
   \quad  +  \Big({\cal     L}_n\big(  W_{n,-} -  U_{n,+}\big)  -i U_{n,-}  \Big)  (t-,\eta-n\ell',k+\ell') 
               \Big] \dd\thN  (\gamma n^2 t,\ell') ,  \label{012403-26a}\end{align}
               and 
               \begin{align} 
 \dd   U_{n,+}(t,\eta,k)     = & \Big[2n^2 \bar\om_n(\eta,k)
       U_{n,-}(t,\eta,k)  
         +2\ga n^2 {\cal     L}_n\big(  U_{n,+} -  W_{n,0}\big) 
    (t,\eta,k)     \Big]\dd t \notag
 \\
  &
    + \underset{\ell'  \in \hat\T_n}{\hat{\sum}}\Big[ {\cal     L}_n\big(  U_{n,+} -  W_{n,0}\big) 
(t- ,\eta-n\ell',k)  \notag \\ &  
      \quad + {\cal     L}_n\big(  U_{n,+} -  W_{n,0}\big)  (t-,\eta-n\ell',k+\ell')
      \Big]
    \dd\thN  (\gamma n^2 t,\ell'), \label{012403-26b}
\end{align}
and
\begin{align}
\dd  U_{n,-}(t,\eta,k)     =  & \Big[-2n^2 \bar\om_n(\eta,k)
       U_{n,+}(t,\eta,k)  -2\ga n^2   U_{n,-}(t,\eta,k)  \Big]\dd
                t \notag\\
 &
   +i\underset{\ell'  \in \hat\T_n}{\hat{\sum}}\Big[
\big(  \delta W_{n}    +iU_{n,-} 
      \big)(t- ,\eta-n\ell',k) \notag
    \\ & \quad -\big(  \delta W_{n}    -iU_{n,-} 
      \big))(t-,\eta-n\ell',k+\ell')
         \Big]\dd\thN  (\gamma n^2 t,\ell').  \label{012403-26c}
\end{align}

\subsection{Preliminary estimates on the initial Fourier-Wigner functions}

Let us fix $\eta \in \bbZ$, and introduce the following scalar product $\lsc \cdot \rsc$ between two random fields $F,G$ defined on $\bbZ\times \hT_n\times \Om_n\times\Xi$ as follows:
	\[
	\lsc F, G \rsc \; : = \E_n \Big[\big\langle F(\eta,\cdot),G(\eta,\cdot)\big\rangle_{\mL^2(\hT_n)} \Big] =\E_n \bigg[\hsum F(\eta,k) G^\star(\eta,k)\bigg],
	\]	
	and its associated norm: 
	\begin{equation} \big\lnm F \big\rnm_\eta \; := \sqrt{\lsc F,F \rsc} = \Big(\E_n \Big[ \big\|F(\eta,\cdot)\big\|^2_{\mL^2(\hT_n)}\Big]\Big)^{1/2}. \label{norms}\end{equation}	
 In the present section we give some preliminary bounds involving $\lnm W_{n,\pm}\rnm_\eta$, $\lnm U_{n,\pm}\rnm_\eta$  which will be useful in the rest of the paper.

\begin{proposition}[Initial bound]
  \label{prop030204-26} From Assumption \eqref{eq:3z}, we have
\begin{equation}
  \label{WUpm}
  {\mathfrak W}_2:=
  \sup_{n\ge1} \sup_{\eta \in \bbZ} \bigg\{ \lnm W_{n,+}(0) \rnm_\eta+ \lnm W_{n,-}(0) \rnm_\eta + 
 \lnm
 Y_{n,+}(0) \rnm_\eta+  \lnm
 Y_{n,-}(0) \rnm_\eta \bigg\}<+\infty.
\end{equation}

\end{proposition}

\proof Recall that the initial condition is distributed according to $\mu_n$ which satisfies Assumption \ref{T}.
By the Cauchy-Schwarz inequality, for any $\eta \in\bbZ$, we have
\begin{align}
  \label{012008-26}
   \big\lnm W_{n,+}(0)\big\rnm_\eta^2& = \hsum\int_{\Om_n}|W_{n,+}(0,\eta,k)|^2 \; \dd\mu_{n}\notag
    \\
    & \le  { \Big(\frac{1}{2n}\Big)^2
    \bigg(\hsum\int_{\Om_n}\Big|\widehat\psi\Big(k+\frac{\eta}{n}\Big)\Big|^4
    \dd\mu_{n}\bigg)^{1/2}
    \bigg(\hsum\int_{\Om_n} \big| \widehat\psi^\star(k)\big|^4 \dd\mu_{n}\bigg)^{1/2}}
  \\
 & =  \Big(\frac{1}{2n}\Big)^2
    \bigg(\hsum\int_{\Om_n}\big|\widehat\psi(k)\big|^4
    \dd\mu_{n}\bigg)^{1/2}
   \bigg(\hsum\int_{\Om_n} \big| \widehat\psi^\star(k)\big|^4 \dd\mu_{n}\bigg)^{1/2}\notag\\
  &= \Big(\frac{1}{2n}\Big)^2
    \bigg(\hsum\int_{\Om_n}\big|\widehat\psi(k)\big|^4
    \dd\mu_{n}\bigg) \le C,\notag
\end{align}
 {from Assumption \ref{T}, see \eqref{eq:3z}.}
This implies $\sup_n \sup_\eta \lnm W_{n,+}(0)\rnm_\eta < +\infty$. The argument for the
 remaining Fourier-Wigner functions follows in the same fashion.
\qed


  
  Our main result of the present section is the following estimate.
    \begin{theorem}\label{thm011002-26} There
      exists $C >0$ such that, for any $t\ge 0$,  
        $\eta \in\bbZ$ and  $n =1,2,\ldots$,
                     \begin{equation}
        \label{022008-26}
        	\int_0^t\Big[	 {  \big\lnm W_{n,+} (s)\big\rnm^{
                    2}_\eta +\big\lnm W_{n,-} (s)\big\rnm^{  2}_\eta+
                  \big\lnm Y_{n,+} (s)\big\rnm^{  2}_\eta + \big\lnm
                  Y_{n,-} (s)\big\rnm^{  2}_\eta  }\Big]\dd s
		\le  C(t+1).
                	\end{equation}
\end{theorem}

The proof of this result is done in   Section \ref{sec4.3}.

\subsection{Proof of Theorem \ref{thm011002-26}}

\label{sec4.3}

\subsubsection{Time evolution of the second moments of Fourier-Wigner functions}

Using It\^o formula we easily obtain from
\eqref{012403-26}--\eqref{012403-26c} the following equations on the
evolution of the $\mL^2$--norms (recall also notation \eqref{020204-26}):
\begin{align}
 \label{010204-26}
  & \frac{\dd}{\dd t}\bbE_n\big| W_{n,+}(t,\eta,k)     \big|^2  \\ & =2{\rm Re}\bigg(
    \bbE_n \Big[W_{n,+}^\star\Big(-in \delta_{n}\om(\eta,k) W_{n,+}
    +2\ga n^2 {\cal L}_n\big(  W_{n,+}-  U_{n,+}\big) \Big)
    (t,\eta,k)\Big]\bigg) \notag \\
  &\quad 
    + \gamma n^2 \underset{\ell' \in\hat\T_n}{\hat{\sum}}  \bbE_n\bigg[\Big|
      \big(U_{n,+} -i U_{n,-}  -W_{n,+} \big)(t,\eta-n\ell',k)   \notag \\ & 
   \qquad \qquad + \big(U_{n,+} +i U_{n,-}  -W_{n,+} \big)  (t,\eta-n\ell',k+\ell')
    +2{\cal R}_n
            \big(W_{n,+}-U_{n,+}\big) (t,\eta-n\ell' )
          \Big|^2\bigg], \notag 
\end{align}
\begin{align}
  \label{010204-26a} 
&
    \frac{\dd}{\dd t}\bbE_n| W_{n,-}(t,\eta,k)     |^2  \\
 &=2{\rm Re}\bigg(\bbE_n \Big[W_{n,-}^\star \Big(in \delta_{n}\om(\eta,k)   W_{n,-}    
          +2\ga n^2{\cal L}_n\big(   W_{n,-}-  U_{n,+}
   \big)  \Big) (t,\eta,k)\Big]\bigg)  \notag
 \\
  &
   \quad + \gamma n^2 \underset{\ell' \in\hat\T_n}{\hat{\sum}} \bbE_n \bigg[\Big|
    \big(U_{n,+} +i U_{n,-}  -W_{n,-} \big) (t,\eta-n\ell',k) 
        \notag  \\ & \qquad \qquad  +  \big(U_{n,+} -i U_{n,-}  -W_{n,-} \big)  (t,\eta-n\ell',k+\ell')      +2{\cal R}_n
            \big(W_{n,-} -U_{n,+}\big) (t,\eta-n\ell' )
                \Big|^2\bigg] ,  \notag
\end{align}
\begin{equation}
  \label{010204-26b}
   \begin{split} &
     \frac{\dd}{\dd t}\bbE_n| U_{n,+}(t,\eta,k)     |^2 \\
 &=2{\rm Re}\bigg(\bbE_n \Big[U_{n,+}^\star \Big(2n^2 \bar\om_n(\eta,k)
       U_{n,-}
         +2\ga n^2{\cal L}_n\big(   U_{n,+} 
     - W_{n,0}    \big)
       \Big)(t,\eta,k)\Big]\bigg) 
 \\
  &\quad 
    + \gamma n^2 \underset{\ell'\in\hat\T_n}{\hat{\sum}} \bbE_n\bigg[\Big|
\big(  W_{n,0}     -U_{n,+} 
      \big)(t ,\eta-n\ell',k) 
      +\big(  W_{n,0}  -U_{n,+} 
      \big)(t,\eta-n\ell',k+\ell')  
     \\
          & \qquad \qquad 
    -2{\cal R}_n\big(W_{n,+} -U_{n,+} \big)
            \big(t,\eta-n\ell'   \big)    
   \Big|^2 \bigg]  
\end{split} \end{equation}
and
\begin{equation}
  \label{010204-26c} \begin{split}
&
  \frac{\dd}{\dd t}\bbE_n| U_{n,-}(t,\eta,k)     |^2  \\ & =2{\rm Re}\bigg(\bbE_n \Big[ U_{n,-}^\star\Big(-2n^2 \bar\om_n(\eta,k)
       U_{n,+}  -2\ga n^2     U_{n,-}  \Big)(t,\eta,k) \Big]\bigg)   \\
 &
    \quad + \gamma n^2\underset{\ell'\in\hat\T_n}{\hat{\sum}} \bbE_n\bigg[\Big|
\frac 12\big( \delta W_{n}   +2iU_{n,-} 
      \big)(t ,\eta-n\ell',k)  
  - \frac 12\big( \delta W_{n}   -2iU_{n,-} 
      \big)(t,\eta-n\ell',k+\ell')
         \Big|^2\bigg].  
\end{split} \end{equation}
Multiplying \eqref{010204-26} and \eqref{010204-26a} by $\frac12$ sideways and afterwards
       adding  all \eqref{010204-26}--\eqref{010204-26c} sideways, and
       performing the summation $\hsum$   we
       obtain the following identity (recall
       \eqref{eq:cD} for the notation of ${\cal D}_n$ and \eqref{norms} for the notation of $\lnm \cdot \rnm_\eta$): for any $t\ge 0$ and $\eta \in \bbZ$,
\begin{equation}
 \label{011002-26} \begin{split}
  &
\frac{\dd}{\dd t} \Big( \frac12 \big\lnm W_{n,+}(t )
\big\rnm^{  2}_\eta + \frac12       \big\lnm  W_{n,-}(t)
\big\rnm^{  2}_\eta  + \big\lnm U_{n,+}(t)      \big\rnm^{
  2}_\eta + \big\lnm  U_{n,-}(t)      \big\rnm^{  2}_\eta \Big)   \\
  & 
  =-2n^2\ga\bbE_n \Big[   {\cal
    D}_n\big(  W_{n,0} -
    U_{n,+} \big)(t,\eta)  \Big]
   -\ga n^2 { \big\lnm \delta W_{n}(t)\big\rnm_\eta^2  }
    -4\ga n^2   {\big\lnm U_{n,-}(t)\big\rnm_\eta^2 }
    \\
   &
   \quad + 2 \gamma n^2 \underset{k,\ell'\in\hat\T_n}{\hat{\sum}}   \bbE_n \bigg[\Big|{\cal L}_n  \big(U_{n,+}   - W_{n,0}
     \big)(t,\eta-n\ell',k) + {\cal L}_n  \big(U_{n,+}   - W_{n,0} \big) (t,\eta-n\ell',k+\ell') \Big|^2 \bigg]
   \\
   &
\quad  +\gamma n^2 \underset{k,\ell'\in\hat\T_n}{\hat{\sum}}  \bbE_n \bigg[\Big|\delta W_{n}  (t,\eta-n\ell',k) 
                \Big|^2 \bigg]
    +4 \gamma n^2 \underset{k,\ell'\in\hat\T_n}{\hat{\sum}}  \bbE_n \bigg[ \Big|U_{n,-} 
     (t,\eta-n\ell',k )
         \Big|^2 \bigg]. 
\end{split}\end{equation}
Recall that $U_{n,+}$ and $U_{n,-}$ are linear combinations of $Y_{n,+}$ and $Y_{n,-}$, see \eqref{020204-26}. Therefore, given  estimate  \eqref{022008-26} we are seeking, the remainder of this section will be devoted to estimating separately each term appearing in the right hand side above, after integration in time, in three different steps.

\subsubsection{Step 1} Here we prove the following: 

\begin{proposition}
	\label{cor011002-26a}
	There exists a constant $C>0$ such that, for any $t\ge 0$, $\eta \in \bbZ$, $n\ge 1$,
	\begin{equation}
		\label{021002-26a} \int_0^t\bbE_n \Big[   {\cal
			D}_n\big(   W_{n,0}  -
		U_{n,+}\big) (s,\eta) \Big]\dd s
	+  { \int_0^t \Big(\big\lnm
		\delta W_{n}(s)\big\rnm_\eta^2  + \big\lnm
		U_{n,-}(s)\big\rnm_\eta^2 \Big)\dd s} \le C(t+1).
	\end{equation}
\end{proposition}

\begin{proof}
%
%
%
From the dynamics \eqref{011002-26} we conclude that
\begin{equation}
 \label{050204-26} \begin{split}
   &
2\ga\int_0^t\bbE_n \Big[   {\cal
    D}_n\big(   W_{n,0 }     -
    U_{n,+} \big) (s,\eta) \Big]\dd s
     +\ga \int_0^t \big\lnm
    \delta W_{n}(s)\big\rnm_\eta^2\; \dd s 
  +4\ga   \int_0^t \big\lnm
    U_{n,-}(s)\big\rnm_\eta^2\;\dd s
   \\
  &   
   \quad  +  \frac1{n^2}\Big(\frac12 \big\lnm W_{n,+}(t )
    \big\rnm^{  2}_\eta + \frac12       \big\lnm  W_{n,-}(t)
    \big\rnm^{  2}_\eta  + \big\lnm U_{n,+}(t)      \big\rnm^{
    	2}_\eta + \big\lnm  U_{n,-}(t)      \big\rnm^{  2}_\eta\Big)  \vphantom{\int_0^1}\\    
  &
    \le \frac1{n^2}\Big(\frac12 \big\lnm W_{n,+}(0 )
    \big\rnm^{  2}_\eta + \frac12       \big\lnm  W_{n,-}(0)
    \big\rnm^{  2}_\eta  + \big\lnm U_{n,+}(0)      \big\rnm^{
    	2}_\eta + \big\lnm  U_{n,-}(0)      \big\rnm^{  2}_\eta\Big) \vphantom{\int_0^1} \\ & \quad 
       + \frac{8\gamma}{n}   {\sum_{\eta'\in\bbT_n}} \; \hsum  \int_0^t\bbE_n
     \bigg[\Big|{\cal L}_n  \big(U_{n,+}   - W_{n,0}
     \big)(s,\eta',k)\Big|^2 \bigg] \dd s  
   \\
   &
 \quad +\frac{\gamma}{n} {\sum_{\eta'\in\bbT_n}} \; \hsum \int_0^t \bbE_n
    \bigg[ \Big|\delta W_{n}  (s,\eta',k) 
                \Big|^2 \bigg]\dd s
    + \frac{4\gamma}{n}{\sum_{\eta'\in\bbT_n}} \; \hsum \int_0^t\bbE_n \bigg[\Big|U_{n,-} 
     (s,\eta',k )
     \Big|^2 \bigg]\dd s.
   \end{split}\end{equation}
 In the left hand side we recover the quantity that we want to estimate (recall \eqref{021002-26a}) plus a term that we can neglect (the one with $\frac1{n^2}$) since it is positive. We now estimate the right hand side. 

First of all, recall from \eqref{030204-26} that $\|\cL_n
(f+g)\|^2_{\mL^2(\hT_n)} \le 2\|f\|^2_{\mL^2(\hT_n)}+
2\|g\|^2_{\mL^2(\hT_n)}$, and by definition
\begin{align*}
\|W_{n,0}
(t,\eta)\|_{\mL^2(\hT_n)}  &\le
\frac12\big(\|W_{n,+}
(t,\eta)\|_{\mL^2(\hT_n)}  +\|W_{n,-}
                 (t,\eta)\|_{\mL^2(\hT_n)}\big)\\
  &
    =\|W_{n,+}
(t,\eta)\|_{\mL^2(\hT_n)} .
\end{align*}
Therefore,
\begin{multline*} \hsum  \bbE_n
 \bigg[\Big|{\cal L}_n  \big(U_{n,+}   - W_{n,0}
\big)(s,\eta',k)\Big|^2 \bigg]   \\ \le 2 \hsum  \bbE_n
\Big[\big|U_{n,+} 
(s,\eta',k)\big|^2 \Big]  +  2 \hsum \bbE_n
\Big[\big|W_{n,+}
(s,\eta',k)\big|^2 \Big]  .
 \end{multline*} Taking the summation over $\eta' \in \T_n$, and dividing by $n$, we get
 
\begin{equation}
 \label{eq:14}
 \begin{split}
   \frac{2}{n}   \sum_{\eta'\in\T_n}\hsum\bbE_n \bigg[\Big| W_{n,+} (s,\eta',k)\Big|^2\bigg]
& =  \frac{1}{2 n^3}
 \sum_{\eta\in\bbT_n}\hsum \bbE_n\left[
 \Big|\hat\psi \Big(n^2t, k+\frac{\eta}{n} \Big)\Big|^2 |\hat\psi ^\star (n^2t,k )|^2\right]
 \\
 &
       =  \frac{1}{2 n^2}\bbE_n\bigg[\Big(\hsum |\hat\psi (n^2t ,k )|^2\Big)^2 \bigg]
    \\&   =  \frac{1}{2 n^2}{\E_n \bigg[}\Big( \sum_{x\in\bbT_n} e_x^{(n)} ( t  )\Big)^2 \bigg]
       = {\int_{\Omega_n}} \big({\cal E}_{\rm tot}^{(n)}(0)\big)^2\; {\dd \mu_n}.
 \end{split}
\end{equation}
The last equality holds by conservation of energy, see \eqref{ex2}.
  The right hand side of \eqref{eq:14} is bounded thanks to estimate \eqref{eq:energybounds}.
       A similar computation holds for $U_{n,+}$.
  We finally get that there exists $C>0$ such that,  for any $t\ge 0$, $n\ge 1$,
\begin{equation}
	\label{eq:b1}
 \frac{8\ga}{n} \sum_{\eta'\in\T_n} \; \hsum  \int_0^t \bbE_n
\bigg[\Big|{\cal L}_n  \big(U_{n,+}   - W_{n,0}
\big)(s,\eta',k)\Big|^2 \bigg]	\dd s \le C(t+1).
\end{equation}
In a similar fashion we can estimate analogous expressions involving $
\delta W_{n}$ and $
U_{n,-}$ in the right hand side of \eqref{050204-26}. Combining this with the estimates above we easily conclude \eqref{021002-26a}. \end{proof}

\subsubsection{Step 2}

Next we show the following.

\begin{proposition}
 \label{cor021002-26}
There exists $C>0$ such that, for any $t\ge 0$, $\eta \in\bbZ$, $n\ge 1$,
 \begin{align}
   \label{021002-26}
  & { \int_0^t \big\lnm 
    U_{n,+}(s)\big\rnm_\eta^2 \; \dd s } \le C(t+1).
\end{align}
\end{proposition}
\proof
Using It\^o formula, together with \eqref{012403-26a} and
\eqref{012403-26b},
we obtain
\begin{align*}
 &
   \frac{\dd}{\dd t}  \bbE_n\Big[U_{n,-}(t,\eta,k)  U_{n,+}^\star
   (t,\eta,k)  \Big]  = -2n^2 \bar\om_n(\eta,k)
       \bbE_n \Big[\big|U_{n,+}(t,\eta,k)\big|^2\Big]  \\ & -2\ga n^2 \bbE_n \Big[
                U_{n,-}(t,\eta,k)   
   U_{n,+}^\star(t,\eta,k) \Big] 
   + 2n^2 \bar\om_n(\eta,k)
       \bbE_n\Big[\big|U_{n,-}(t,\eta,k)  \big|^2\Big] \\ & 
         +2\ga n^2  \bbE_n\Big[U_{n,-}(t,\eta,k) \Big({\cal L}_n\big(   U_{n,+}
            -  W_{n,0}    \big)
    (t,\eta,k)     \Big)^\star\Big]
\\
 &
   -i\ga n^2 \underset{\ell'\in\hat\T_n}{\hat{\sum}}\bbE_n \bigg[ \bigg\{ 
\tfrac 12\big(  \delta W_{n}    +2iU_{n,-} 
      \big)(t- ,\eta-n\ell',k) 
     - \tfrac 12\big( \delta W_{n}    -2iU_{n,-} 
      \big)(t-,\eta-n\ell',k+\ell')
    \bigg\}\\
 &
 \qquad \qquad \quad \times \bigg\{ 
{\cal L}_n\big(  W_{n,0}     -U_{n,+} 
      \big)(t- ,\eta-n\ell',k) 
      +\ {\cal L}_n\big(  W_{n,0}    -U_{n,+} 
      \big)(t-,\eta-n\ell',k+\ell')
    \bigg\}^\star\bigg].
\end{align*}
Integrating in time both sides, summing over $k$, using \eqref{030204-26} and the fact that
$\bar\om_n(\eta,k)\ge \om_0  >0$, we conclude that
\begin{align}
 \label{010304-26}
& 2\om_0  \hsum\int_0^t
       \bbE_n \Big[\big|U_{n,+}(s,\eta,k)\big|^2\Big]\dd s
 \le  \frac{1}{n^2}\sum_{s=0,t} \hsum \Big|\bbE_n \Big[U_{n,-}(s,\eta,k)  U_{n,+}^\star
                (s,\eta,k)  \Big]\Big| 
 \\
 &+ 2\ga \hsum\int_0^t\Big| \bbE_n \Big[
                U_{n,-}(s,\eta,k)   
   U_{n,+}^\star(s,\eta,k) \Big]\Big|\dd s  
   + 2\|\om\|_\infty
       \hsum\int_0^t
       \bbE_n \Big[\big|U_{n,-}(s,\eta,k)\big|^2\Big]\dd s\notag
 \\
 &
   +2\ga   \hsum \int_0^t \Big|\bbE_n\Big[U_{n,-}(s,\eta,k) \Big({\cal L}_n\big(   U_{n,+}
            -  W_{n,0}    \big)
    (s,\eta,k)     \Big)^\star\Big]\Big|\dd s\notag \\
 &
   +\ga  \underset{k,\ell'\in\hat\T_n}{\hat{\sum}}\int_0^t\bigg\{\bbE_n \Big[\big|
\tfrac 12  \delta W_{n}     (s    ,\eta-n\ell',k) \big|^2\Big]+4 \bbE_n \Big[\big| U_{n,-} 
   (s    ,\eta-n\ell',k)\big|^2\Big]
  \notag\\
  &
\qquad \qquad \qquad  +\bbE_n \Big[\big| \tfrac 12  \delta W_{n}    (s,\eta-n\ell',k+\ell')\big|^2\Big]+4 \bbE_n \Big[\big|  U_{n,-}(s,\eta-n\ell',k+\ell')\big|^2\Big] \bigg\}^{1/2}
 \notag  \\
 &
 \qquad \qquad \times \bigg\{\bbE_n \Big[\big|
 {\cal L}_n\big(  W_{n,0}    -U_{n,+} 
      \big)(s ,\eta-n\ell',k) \big|^2\Big] 
   \notag \\ & \qquad \qquad \qquad +\bbE_n \Big[\big| 2{\cal L}_n\big(  W_{n,0}     -U_{n,+} 
      \big)(t-,\eta-n\ell',k+\ell')
        \big|^2\Big]\bigg\}^{1/2} \dd s.\notag
\end{align}
Invoking Young's inequality $2ab\le ca^2+b^2/c$, valid for all
$a,b,c>0$, we can write
\begin{align*}
2\ga \hsum\int_0^t\Big| \bbE_n \Big[
                U_{n,-}(s,\eta,k)   
   U_{n,+}^\star(s,\eta,k) \Big]\Big|\dd s  \le & \; \om_0  \hsum\int_0^t
  \bbE_n \Big[ \big|U_{n,+}(s,\eta,k)\big|^2\Big]\dd s\\
  &   +\frac{\ga}{\om_0}  \hsum\int_0^t
  \bbE_n \Big[\big|U_{n,-}(s,\eta,k)\big|^2\Big]\dd s.
\end{align*}
Hence, {recalling the definition of $\cD_n$ in \eqref{030204-26}},  the right hand side of \eqref{010304-26} can be
estimated by  (using $\hat\sum_\eta \le \sup_\eta$)
   \begin{multline*}
\frac{1}{n^2}\sum_{s=0,t} \hsum\Big|\bbE_n \Big[U_{n,-}(s,\eta,k)  U_{n,+}^\star
                (s,\eta,k)  \Big]\Big| +\om_0  \int_0^t
                {\big\lnm U_{n,+}(s)\big\rnm_\eta^2}\;\dd s\\
    + C{\sup_{\eta'}\int_0^t
      \bigg\{  \big\lnm U_{n,-}(s)\big\rnm_{\eta'}^2\; \dd s 
        +{\cal D}_n \big( W_{n,0} - U_{n,+}\big)(s,\eta') + 
      \big\lnm \delta W_n(s)\big\rnm_{\eta'}^2\bigg\}\dd s}
 \end{multline*}
and the conclusion of the proposition follows from Proposition \ref{cor011002-26a} proved in Step 1.
\qed

\subsubsection{Step 3} Finally we prove:

\begin{proposition}
 \label{cor041002-26}
There exists $C>0$ such that, for any $t\ge 0$, $\eta\in\bbZ$, $n\ge 1$,
 \begin{align}
   \label{021002-26b}
  &    \int_0^t {\big\lnm
    W_{n,0} (s)\big\rnm_\eta^{2}}\; \dd s \le C(t+1).
\end{align}
\end{proposition}
\proof
Recalling the notation  ${\cal
 R}_nf=\hat\sum_{\ell \in \hat\T_n}f(\ell)$ we can write (cf.~\eqref{ex2})
\begin{align*}
 &
   \big|{\cal R}_n\big(W_{n,\pm}\big)(s,\eta)\big|\le
   \frac{1}{2n}\underset{\ell\in\hat\T_n}{\hat\sum}\Big| \widehat\psi\Big(n^2s,\ell+\frac{\eta}{n}\Big)\Big|
   \big|\widehat\psi^\star(n^2s,\ell)\big| 
   \le \frac{1}{2n}\underset{\ell\in\hat\T_n}{\hat\sum}
   \big|\widehat\psi (n^2s,\ell)\big|^2={\cal E}_{\rm tot}^{(n)}(0).
\end{align*}
Similarly
$   |{\cal R}_n\big(U_{n,\pm}\big)(s,\eta)|\le {\cal E}^{(n)}_{\rm
 tot}(0).$
Writing 
\begin{align*}
 &W_{n,0} (s,\eta,k)=\big({\cal R}_n -{\cal
   L}_n\big)W_{n,0}(s,\eta,k)\quad \mbox{and}\quad U_{n,+} (s,\eta,k)=\big({\cal R}_n -{\cal
   L}_n\big)U_{n,+}(s,\eta,k)
\end{align*}
we conclude
\begin{align*}
 &\hsum \int_0^t\bbE_n\Big[\big|
     (W_{n,0}-U_{n+})(s,\eta,k) \big|^2\Big]\dd s \le 2\int_0^t\bbE\Big[\big|{\cal R}_n
   (W_{n,0}-U_{n+}) (s)\big)(\eta)\big|^2\Big]\dd
   s\\
 &\quad
   + 2 \hsum\int_0^t\bbE_n\Big[\big|{\cal
   L}_n (W_{n,0}-U_{n+}) (s,\eta,k)\big|^2\Big]\dd s \le C(t+1),
\end{align*}
by virtue of Proposition \ref{cor011002-26a}. This, together with Proposition
\ref{cor021002-26},  implies \eqref{021002-26b}.
\qed

Gathering together estimates \eqref{021002-26a}, \eqref{021002-26} and
\eqref{021002-26b}, we conclude estimate \eqref{022008-26}  {(recall that $W_{n,0}=\frac12(W_{n,+}+W_{n,-})$)},
finishing in this way the proof of Theorem \ref{thm011002-26}.\qed

\section{Proof of Theorem \ref{thm:laplace}}

 \label{sec:proofLaplace}
 
 Theorem \ref{thm:laplace} involves the Laplace transform
 $\hat w_{n,+}$ of   $\tilde W_{n,+}$ (see
 \eqref{teW}). We start this section by obtaining the   system of
 equations obtained by applying
 the 
 Laplace transform to the  equations describing the evolution of the
 fluctuations of the Fourier Wigner functions.

\subsection{The dynamics of the  fluctuations of the Fourier-Wigner functions and their Laplace transforms} \label{sec:wig}

Recall \eqref{022210-25} and \eqref{020204-26} for the definition of the Wigner functions, and let (as in \eqref{teW}), for $t\ge 0$ and $(\eta,k)\in\bbZ\times \hat\T_n$,
\[
  \begin{split}
    \tilde W_{n,\pm}(t,\eta,k)&:=  W_{n,\pm}(t,\eta,k)-\bar W_{n,\pm}(t,\eta,k), \qquad \bar W_{n,\pm}(t,\eta,k):= \E_n\big[W_{n,\pm}(t,\eta,k)\big]\\
    \tilde U_{n,\pm}(t,\eta,k)&:=  U_{n,\pm}(t,\eta,k)-\bar U_{n,\pm}(t,\eta,k), \qquad \bar U_{n,\pm}(t,\eta,k):= \E_n\big[U_{n,\pm}(t,\eta,k)\big].
  \end{split}
\]
We also use the notation (as in \eqref{020204-26})
\begin{align*}
	\tilde  W_{n,0} := \frac12\big( \tilde  W_{n,+} + \tilde
	W_{n,-}\big)\quad\mbox{and} \quad \delta\tilde  W_{n} :=\frac12\big(\tilde  W_{n,+} - \tilde
	W_{n,-}\big).
\end{align*}
Recall the system of equations \eqref{012403-26}--\eqref{012403-26c}. Similar computations show that (recall also notation \eqref{eq:not}):
\begin{align}
  \label{010704-26z} 
   \dd \tilde W_{n,+}(t,\eta,k)      
   & =\Big[-in \delta_{n}\om(\eta,k)   \tilde  W_{n,+}(t,\eta,k)    +2\ga n^2 {\cal
     L}_n\big( \tilde  W_{n,+}- \tilde  U_{n,+}\big) (t,\eta,k)\Big] \dd t \notag\\
   &\quad 
     +   \underset{\ell'  \in \hat\T_n}{\hat{\sum}} \Big[
       \Big(-i U_{n,-}  + {\cal
     L}_n\big(W_{n,+} -U_{n,+} \big)\Big)(t-,\eta-n\ell',k)  
  \notag \\
  & \qquad + \Big( i U_{n,-}  + {\cal
     L}_n\big(W_{n,+} -U_{n,+} \big)\Big)  (t-,\eta-n\ell',k+\ell')
             \Big]\dd\thN  (\gamma n^2 t,\ell'),
\end{align}
\begin{align}
   \label{010704-26a}
       \dd\tilde  W_{n,-}(t,\eta,k)   & =\Big[in \delta_{n}\om(\eta,k)   \tilde W_{n,-}(t,\eta,k)    
           +2\ga n^2{\cal L}_n\big(   \tilde W_{n,-}- \tilde  U_{n,+}
    \big)(t,\eta,k)   \Big] \dd t \notag
  \\
   & \quad 
     + \underset{\ell'  \in \hat\T_n}{\hat{\sum}} \Big[
     \Big(i U_{n,-}  + {\cal
     L}_n\big(W_{n,-} -U_{n,+} \big)\Big) (t-,\eta-n\ell',k)  \notag
  \\
   & \qquad
     +  \Big(  -i U_{n,-}  + {\cal
     L}_n\big(W_{n,-} -U_{n,+} \big)\Big)  (t-,\eta-n\ell',k+\ell') 
                           \Big] \dd\thN  (\gamma n^2 t,\ell') ,
\end{align} 
\begin{align}
  \label{010704-26b}
 \dd  \tilde  U_{n,+}(t,\eta,k)  &   =\Big[2n^2 \bar\om_n(\eta,k)
       \tilde  U_{n,-}(t,\eta,k)  
          +2\ga n^2{\cal L}_n\big(  \tilde  U_{n,+}
             -  \tilde  W_{n,0}  \big)
     (t,\eta,k)     \Big]\dd t  \notag
  \\
   & \quad 
    + \underset{\ell'  \in \hat\T_n}{\hat{\sum}} \Big[
 {\cal L}_n \big( U_{n,+} -  W_{n,0} 
        \big)(t- ,\eta-n\ell',k)  \notag
      \\ & \qquad +  {\cal L}_n \big( U_{n,+} -  W_{n,0} 
        \big)(t-,\eta-n\ell',k+\ell')
         \Big]
      \dd\thN  (\gamma n^2 t,\ell')
\end{align}
and
\begin{align}
   \label{010704-26c}
 \dd\tilde   U_{n,-}(t,\eta,k)   &  = \Big[-2n^2 \bar\om_n(\eta,k)
    \tilde     U_{n,+}(t,\eta,k)  -2\ga n^2  \tilde  U_{n,-}(t,\eta,k)  \Big]\dd
                 t \notag \\
  &
   \quad  +i\underset{\ell'  \in \hat\T_n}{\hat{\sum}}\Big[
\big(  \delta W_{n}     +iU_{n,-} 
       \big)(t- ,\eta-n\ell',k) 
    \notag  \\ & \qquad  - \big(  \delta W_{n}     -iU_{n,-} 
       \big)(t-,\eta-n\ell',k+\ell')
          \Big]\dd\thN  (\gamma n^2 t,\ell').
\end{align}
 {We can rewrite \eqref{010704-26z}--\eqref{010704-26c} in a compact form:
for any $\mathrm{X} \in \{U,W\}$, $\mathrm{x}\in \{u,w\}$ and $\iota \in \{-,+\}$, we have
\begin{equation}
	\label{eq:compact}
\dd \tilde{\mathrm{X}}_{n,\iota}(t,\eta,k) = \mathfrak{D}_{n,\mathrm{x},\iota}(t,\eta,k)\dd t + \underset{\ell'  \in \hat\T_n}{\hat{\sum}} \mathfrak{r}_{n,\mathrm{x},\iota}(t-,\eta,k,\ell')\dd \tilde{\hat N}(\gamma n^2 t,\ell'),
\end{equation}
for suitable quantites $\mathfrak{D}_{n,{\rm x},\iota}$ and $\mathfrak{r}_{n,{\rm x},\iota}$ which can be easily deduced from  \eqref{010704-26z}--\eqref{010704-26c}.}

Suppose now that ${\rm Re}\,\la>0$. Recall the  definition of the Laplace transform in \eqref{eq:lapw}. In a similar manner, we define
\begin{equation}
  \label{lwu} \begin{split}
   \hat w_{n,\pm}(\la,\eta,k)&:= \int_0^{+\infty}e^{-\la t}\tilde
     W_{n,\pm}(t,\eta,k) \dd t,\\
     \hat u_{n,\pm}(\la,\eta,k)&:= \int_0^{+\infty}e^{-\la t}\tilde
      U_{n,\pm}(t,\eta,k) \dd t .
\end{split}\end{equation}
Thanks to Theorem \ref{thm011002-26} the Laplace transforms $\hat w_{n,\pm}(\la)$ and $\hat
u_{n,\pm}(\la)$ are defined $\bbP_n$ a.s.~for any $\la$ such that $\mathrm{Re}\,\lambda >0$. In accordance
to our previous convention we also let
\begin{equation}
  \label{lw0}\begin{split}
   \hat w_{n,0}(\la,\eta,k)&:= \int_0^{+\infty}e^{-\la t}\tilde
     W_{n,0}(t,\eta,k) \dd t,\\
      \delta \hat w_{n }(\la,\eta,k)&:= \int_0^{+\infty}e^{-\la t}
      \delta \tilde W_{n}(t,\eta,k) \dd t . 
\end{split}\end{equation}
Performing the Laplace transform on both sides of equations
\eqref{010704-26z}--\eqref{010704-26c} we conclude the following: for $\mathrm{Re}\,\la >0$, $(\eta,k)\in\bbZ\times\hat\T_n$,
\begin{align}
   \label{wn+}
     \la \hat w_{n,+}&(\la,\eta,k)  -   \tilde  W_{n,+}(0,\eta,k)     
     = -\big(2\ga n^2+in \delta_{n}\om(\eta,k)  \big)  \hat w_{n,+}
     (\la,\eta,k)  \\& +2\ga n^2\underset{\ell  \in \hat\T_n}{\hat{\sum}} \hat w_{n,+}
     (\la,\eta,\ell) 
     -2\ga n^2 {\cal
     L}_n    \hat u_{n,+}(\la,\eta,k)+ {\frak R}_{n,w,+}(\la,\eta,k), \notag \\
   \label{wn-}
       \la  \hat w_{n,-}&(\la,\eta,k)   -\tilde W_{n,-}(0,\eta,k)      =-\big(2\ga n^2 -in \delta_{n}\om(\eta,k)    \big)\hat w_{n,-}(\la,\eta,k)   \\& +2\ga n^2\underset{\ell  \in \hat\T_n}{\hat{\sum}}\hat w_{n,-}
     (\la,\eta,\ell) 
     -2\ga n^2{\cal L}_n  \hat u_{n,+}
     (\la,\eta,k)    + {\frak R}_{n,w,-}(\la,\eta,k), \notag \\
  \label{un+}
   \la  \hat u_{n,+}&(\la,\eta,k) -\tilde U_{n,+}(0,\eta,k)      =2n^2  \bar\om_n(\eta,k) \hat u_{n,-}(\la,\eta,k) 
         \\ & +\ga
     n^2{\cal L}_n\big(\hat u_{n,+}- 
              \hat w_{n,0} \big)
     (\la,\eta,k)    
     +  {\frak R}_{n,u,+}(\la,\eta,k), \vphantom{\hsum} \notag 
\\ 
   \label{un-}
 \la  \hat u_{n,-}&(\la,\eta,k)  -\tilde U_{n,-}(0,\eta,k)        = -2n^2
                 \bar\om_n(\eta,k) \hat u_{n,+}(\la,\eta,k) 
  \\ & -2\ga   n^2\hat u_{n,-}(\la,\eta,k)  +{\frak R}_{n,u,-}(\la,\eta,k),  \notag
\end{align}
where (recall \eqref{eq:compact}), for any $\mathrm{x}\in\{u,w\}$ and $\iota\in\{-,+\}$, 
\begin{equation}
	\label{Rn}     {\frak R}_{n,\mathrm{x},\iota}(\la,\eta,k) := \underset{\ell'  \in \hat\T_n}{\hat{\sum}} \int_0^{+\infty}e^{-\la t}{\frak r}_{n,\mathrm{x},\iota}(t,\eta ,k,\ell')
	\dd\thN  (\gamma n^2 t,\ell').
\end{equation}
The rest of the section proceeds as follows: in Section \ref{sec:prel} we give some preliminary estimates of $\lnm \hat w_{n,\pm}(\la) \rnm_\eta$, $\lnm \hat u_{n,\pm}(\la)\rnm_\eta$ which can be easily deduced from Theorem \ref{thm011002-26}. Then, in Section \ref{sec:steps} we give the two main intermediate steps, stated in Proposition \ref{prop:dirichlet} and Proposition \ref{prop:rest}, and explain how one can deduce Theorem \ref{thm:laplace}. In Section \ref{sec:proofdir} we prove Proposition \ref{prop:dirichlet} and in Section \ref{sec:proofrest} we prove Proposition \ref{prop:rest}.

\subsection{Preliminary estimates of the Laplace transforms following
  from Theorem \ref{thm011002-26}} \label{sec:prel}


 The first natural estimate that we can get from Theorem \ref{thm011002-26} is the following.
  \begin{proposition}
    \label{prop012302-26}
    There exists  $C >0$ such that, for any $\lambda$ s.t.~${\rm
    	Re}\,\lambda>0$,  $\eta \in \bbZ$, $n\ge 1$,
  \begin{multline}
    \label{012301-26}
  \big\lnm \hat w_{n,+}(\la) \big\rnm_\eta^2 +  \big\lnm \hat w_{n,-}(\la) \big\rnm_\eta^2 + \big\lnm \hat u_{n,+}(\la) \big\rnm_\eta^2 + \big\lnm \hat u_{n,-}(\la) \big\rnm_\eta^2  \le \frac{C(1+{\rm Re}\,\la)}{({\rm Re}\,\la)^2}.
\end{multline}
    \end{proposition}
    \proof
For ${\rm Re}\,\lambda >0$, $(\eta,k) \in\bbZ\times\hat\T_n$ we have
\begin{align*}
  \big|\hat w_{n,\pm}(\la,\eta,k)\big|&\le\int_0^{+\infty}e^{-{\rm Re}\,\la
     t} \big|\tilde
     W_{n,\pm}(t,\eta,k) \big|\dd  t
=
   \int_0^{+\infty}e^{-{\rm Re}\,\la
     t} \frac{\dd}{\dd t}\int_0^t \big|\tilde
    W_{n,\pm}(s,\eta,k) \big|\dd s\dd t \\
  &
    ={\rm Re}\,\la\int_0^{+\infty}e^{-{\rm Re}\,\la
     t} \dd t\int_0^t 
     \big| \tilde W_{n,\pm}(s,\eta,k) \big| \dd s.
\end{align*}
By the Jensen inequality
\begin{align*}
  &|\hat w_{n,\pm}(\la,\eta,k)|^2\le {\rm Re}\,\la\int_0^{+\infty}e^{-{\rm Re}\,\la
     t} \dd t\Big(\int_0^t 
     \big| \tilde W_{n,\pm}(s,\eta,k) \big| \dd s\Big)^2.
\end{align*}
  Using Theorem \ref{thm011002-26} and Cauchy-Schwarz inequality, we conclude therefore that 
\begin{align*}
  &\bbE_n\big[|\hat w_{n,\pm}(\la,\eta,k)|^2\big]
    \le C {\rm Re}\,\la \int_0^{+\infty}e^{-{\rm Re}\,\la
     t} t(t+1)  \dd t\le \frac{C}{({\rm Re}\,\la)^2}+\frac{C}{({\rm Re}\,\la)}
  \end{align*}
and \eqref{012301-26} follows since by definition (see \eqref{norms})
\[ \big\lnm \hat w_{n,\pm}(\lambda)\big\rnm_\eta^2 =  \hsum \bbE_n\big[|\hat w_{n,\pm}(\la,\eta,k)|^2\big].\]
 Estimates of $\bbE_n
     \big[|\hat u_{n,\pm}(\la,\eta,k) |^2\big]$ can be done analogously.
    \qed

 \medskip
 
Another direct consequence of Theorem \ref{thm011002-26} is the following control of  
the terms ${\frak R}_{n,{\rm x},\iota}(\la,\eta,k)$ defined in \eqref{Rn}:
\begin{proposition}
	\label{prop022302-26}
	There exists  $C>0$ such that   for any $\lambda$ s.t.~${\rm
		Re}\,\lambda>0$,   {$\eta \in \bbZ$ and  $n\ge 1$},
	\begin{multline}
		\label{032302-26}
		 \big\lnm {\frak
			R}_{n,w,+}(\la) \big\rnm_\eta^2 +
	     \big\lnm {\frak
	     	R}_{n,w,-}(\la) \big\rnm_\eta^2 + \big\lnm {\frak
	     	R}_{n,u,+}(\la) \big\rnm_\eta^2 + \big\lnm {\frak
	     	R}_{n,u,-}(\la) \big\rnm_\eta^2
	\\	\le  Cn^2\Big(1+\frac{1}{{\rm
				Re}\,\la}\Big). \vphantom{\int_0^1}
                          \end{multline}
                         
\end{proposition}
\proof   Using It\^o isometry we get, for ${\rm Re}\,\la >0$ and $\eta \in\bbZ$,
\begin{align*}
    \bbE_n\bigg[  \hsum &\Big|\underset{\ell'  \in \hat\T_n}{\hat{\sum}} \int_0^{+\infty}e^{-\la t}  W_{n,+}
  (t-,\eta-n\ell',k) \dd\thN  (\gamma n^2
    t,\ell')\Big|^2\bigg]\\
  &
    =\ga n^2\bbE_n   \bigg[\underset{k,\ell'  \in \hat\T_n}{\hat{\sum}} \int_0^{+\infty}e^{-2{\rm Re}\,\la t} \big| W_{n,+}
    (t,\eta-n\ell',k) \big|^2\dd t\bigg]\\
  &
    \le \ga n^2  \sup_{\eta'\in\bbZ}  \int_0^{+\infty}e^{-2{\rm Re}\,\la t}\bbE_n   \Big[ \big\| W_{n,+}
  (t,\eta',\cdot) \big\|^2_{{\mb L}^2(\hT_n)}\big]\dd t\\
  & {\le \ga n^2 \sup_{\eta' \in \bbZ}  \int_0^{+\infty}e^{-2{\rm Re}\,\la t} \big\lnm W_{n,+}(t)\big\rnm_{\eta'}^2\, \dd t}
\end{align*}
and similar estimates hold for $W_{n,-}$ and  $U_{n,\pm}$.
We conclude that the
left hand side of \eqref{032302-26} can be estimated, {after an
  integration by parts}, by
{\begin{align*}
           C n^2 &\sum_{\iota=\pm} \sup_{\eta \in \bbZ}\int_0^{+\infty}e^{-2{\rm Re}\,\la t} \Big\{  \big\lnm W_{n,\iota}(t)\big\rnm_\eta^2  +  \big\lnm U_{n,\iota}(t)\big\rnm_\eta^2\Big\}\dd t\\
         & =  C n^2 \sum_{\iota=\pm} \sup_{\eta \in \bbZ} \int_0^{+\infty} 2{\rm Re}\,\la \, e^{-2{\rm Re}\,\la t} \Big(\int_0^t\Big\{ \big\lnm W_{n,\iota}(s)\big\rnm_\eta^2  +  \big\lnm U_{n,\iota}(s)\big\rnm_\eta^2  \Big\} \dd s\Big)\dd t \\
         & \le Cn^2 \int_0^{+\infty} 2{\rm Re}\,\la e^{-2{\rm Re}\,\la t}(t+1)\dd t,
\end{align*}}
where the last inequality follows from Theorem \ref{thm011002-26}. Thus
\eqref{032302-26} follows.
\qed

Proposition \ref{prop012302-26} and \ref{prop022302-26} are not
  enough to conclude  Theorem \ref{thm:laplace}.  To conclude the
  result we need additional estimates obtained from the system of
  equations \eqref{wn+}--\eqref{un-} that shall be obtained in the
  following sections.

\subsection{Two intermediate steps and the proof of Theorem \ref{thm:laplace}} \label{sec:steps}
In this section we state the two main intermediate results towards Theorem \ref{thm:laplace}. 

\begin{proposition}
	\label{prop:dirichlet}
There exists
        $C >0$   such that, for any $ \eta\in\bbZ, \, n\ge 1$ and $\la$ s.t.~${\rm Re}\,\la>0$.
	\begin{equation} \label{eq:dir1}
		\E_n\Big[ \cD_n\big(\hat
                w_{n,+}(\lambda,\eta,\cdot)\big) \Big] \le  \frac{C}{n^2}\Big(1+\frac{n }{({\rm
                  Re}\,\la)^{1/2}}+\frac{|\la|}{{\rm Re}\,\la}\Big) \Big(1+\frac{1}{{\rm Re}\,\la}
                  \Big)     
                \end{equation}
                
\end{proposition}

Proposition \ref{prop:dirichlet} is proved in Section \ref{sec:proofdir}.

\begin{proposition}\label{prop:rest}
		For any $\la$ such that ${\rm Re}\,\la >0$, and any $\eta \in \bbZ$,
		\begin{equation}
			\label{eq:limsum} \lim_{n\to +\infty} \E_n \bigg[\bigg| \hsum \hat w_{n,+}(\la,\eta,k)\bigg| \bigg] = 0.
		\end{equation}
\end{proposition}

The proof of Proposition \ref{prop:rest} is quite involved and is
presented in Section \ref{sec:proofrest}. Both Propositions \ref{prop:dirichlet} and
  \ref{prop:rest} imply  Theorem \ref{thm:laplace}, since: 
\begin{align*}
	 \E_n &\Big[\big\|\hat w_{n,+}(\lambda,\eta,\cdot)\big\|_{{\mb
                L}^2(\T_n)}\Big] \\
              &
	 = \hsum \E_n\bigg[\Big|\hat w_{n,+}(\lambda,\eta,k)-\underset{\ell\in\hat\T_n}{\hat\sum} \hat w_{n,+}(\lambda,\eta,\ell)+\underset{\ell\in\hat\T_n}{\hat\sum} \hat w_{n,+}(\lambda,\eta,\ell)\Big|\bigg] \\
	& \le  	\Big\{\E_n\Big[ \cD_n\big(\hat w_{n,+}(\lambda,\eta,\cdot)\big) \Big] \Big\}^{1/2}+  \E_n \bigg[\bigg| \hsum \hat w_{n,+}(\la,\eta,k)\bigg|\bigg] ,
\end{align*}
where, to estimate the first term on the right hand side in the last
inequality, we have used the Jensen and then  the equality part of \eqref{030204-26}. Conclusion of Theorem
\ref{thm:laplace} then easily follows.

\subsection{Proof of Proposition \ref{prop:dirichlet}}

\label{sec:proofdir}

We first state two intermediate lemmas and conclude the proof of
Proposition \ref{prop:dirichlet} (recall \eqref{norms} for the definition of norms). 
\begin{lemma}
	\label{prop032302-26}
There exists
        $C >0$ such that   for
        any  $\la$ s.t.~${\rm Re}\,\la >0$, $\eta \in\bbZ, n\ge 1$,
	\begin{multline}
          \label{062302-26}
		\bbE_n \Big[ {\cal
			D}_n\Big( \big(  \hat w_{n,0} -  \hat u_{n,+}\big)(\la,\eta)\Big)\Big] + \big\lnm \hat u_{n,-}(\la)\big\rnm^2_\eta + \big\lnm \delta \hat w_n(\la)\big\rnm^2_\eta  
                      \\
                      \le   \frac{C}{n({\rm Re}\,\la)^{1/2}} \Big(1+\frac{1}{{\rm
           Re}\,\la} \Big) .
       \end{multline}
                            
\end{lemma}

 \begin{lemma}
	\label{prop042302-26}
There exists
        $C >0$ such that       for
        any  $\la$ s.t.~${\rm Re}\,\la >0$, $\eta \in\bbZ, n\ge 1$,
        \begin{align}
		\label{un-2}
	 \big\lnm \hat u_{n,+}(\la)\big\rnm_\eta^2 \le 
	   \frac{C}{n^2}\Big(1+\frac{n }{({\rm
                  Re}\,\la)^{1/2}}+\frac{|\la|}{{\rm Re}\,\la}\Big) \Big(1+\frac{1}{{\rm Re}\,\la}
                  \Big). 
	\end{align}
\end{lemma}

\subsubsection*{The end of the proof of Proposition
  \ref{prop:dirichlet}}
It is clear that \eqref{062302-26} and \eqref{un-2} imply Proposition \ref{prop:dirichlet} since from \eqref{eq:estimdir}
\begin{align*}&\cD_n\big(\hat w_{n,+}(\la,\eta)\big) \\
&\le 2\cD_n\Big(\big( \tfrac12(\hat w_{n,+}+\hat w_{n,-}) - \hat u_{n,+}\big)(\la,\eta)\Big) + 2 \cD_n\Big(\big(\tfrac12(\hat w_{n,+}-\hat w_{n,-})+ \hat u_{n,+}\big)(\la,\eta) \Big)\\
& \le 2 {\cal
	D}_n\Big( \big(  \hat w_{n,0} -  \hat u_{n,+}\big)(\la,\eta)\Big)+ 8 \big\| \delta\hat w_{n}(\la,\eta)  \big\|_{{\mb L}^2(\hT_n)}^2 + 8 \big\| \hat u_{n,+}(\la,\eta)\big\|_{\mL^2(\hT_n)}^2.
 \end{align*}
Taking expectation both sides, from \eqref{062302-26} and \eqref{un-2} we obtain \eqref{eq:dir1}. \qed

Now let us proceed with the proofs of the intermediate lemmas.

\subsubsection{Proofs of  Lemmas \ref{prop032302-26} and \ref{prop042302-26}}

Let us start with:

\proof[Proof of Lemma \ref{prop032302-26}]        Multiplying \eqref{wn+} and \eqref{wn-} by $\frac12 \hat
     w_{n,+}(\la,\eta)$ and $\frac12 \hat
     w_{n,-}(\la,\eta)$ sideways, respectively and afterwards
        adding   them   sideways  we get
        \begin{equation}
          \label{wnpm} \begin{split}
  \frac{\la}{2} \sum_{\iota=\pm}\big\| \hat
     w_{n,\iota}(\la,\eta)&\big\|_{\mL^2(\hT_n)}^2-\frac{1}{2} \sum_{\iota=\pm} \big\langle\tilde
     W_{n,\iota}(0,\eta ),  \hat w_{n,\iota}(\la,\eta)\big\rangle _{\mL^2(\hT_n)}\\
       &
            =\frac{in }{2}\hsum \delta_{n}\om(\eta,k)\Big[|\hat w_{n,-}(\la,\eta,k)|^2 -
              |\hat w_{n,+}(\la,\eta,k)|^2\Big]     \\
      &\quad 
        +2\ga n^2 \big\langle{\cal
            L}_n\big(  \hat w_{n,0} -  \hat u_{n,+}\big)(\la,\eta), \hat
        w_{n,0}(\la,\eta)  \big\rangle _{\mL^2(\hT_n)} \vphantom{\int_0}\\
      & \quad
        +2\ga n^2 \big\langle{\cal
            L}_n\big(  \delta\hat w_{n}\big)(\la,\eta) , \delta\hat
        w_{n}(\la,\eta) \big\rangle _{\mL^2(\hT_n)}
    \\ & \quad +   \frac12\sum_{\iota=\pm}\big\langle{\frak R}_{n,w,\iota}(\la,\eta) ,\hat w_{n,\iota}(\la,\eta) \big\rangle _{\mL^2(\hT_n)}. \end{split}
\end{equation}
 Multiplying scalarly \eqref{un+} and \eqref{un-} by $\hat
 u_{n,+}(t,\eta,k)$ and  $\hat
 u_{n,-}(t,\eta,k)$ sideways, respectively we get
\begin{equation}
  \label{un+1} \begin{split}
   \la \big\| \hat u_{n,+}(\la,\eta)\big\|_{\mL^2(\hT_n)}^2  &-\big\langle\tilde U_{n,+}(0,\eta), \hat u_{n,+}(\la,\eta)\big\rangle_{\mL^2(\hT_n)}    \\ &  =2n^2 \big\langle \bar\om_n(\eta) \hat u_{n,-}(\la,\eta) , \hat u_{n,+}(\la,\eta)\big\rangle_{\mL^2(\hT_n)} \vphantom{\int^1}
   \\
   & \vphantom{\int_0} \quad 
     +2\ga
     n^2\big\langle{\cal L}_n\big(   \hat u_{n,+}- 
              \hat w_{n,0} \big)(\la,\eta) , \hat u_{n,+}(\la,\eta)\big\rangle_{\mL^2(\hT_n)}   
\\ &  \quad + \big\langle {\frak R}_{n,u,+}(\la,\eta),  \hat
     u_{n,+}(\la,\eta)\big\rangle_{\mL^2(\hT_n)} ,
\end{split} \end{equation}
and
\begin{equation}
   \label{un-1} \begin{split}
 \la  \big\|\hat u_{n,-}(\la,\eta)\big\|^2_{\mL^2(\hT_n)}  &-\big\langle\tilde
                  U_{n,-}(0,\eta) , \hat
                  u_{n,-}(\la,\eta)\big\rangle_{\mL^2(\hT_n)}       
   \\ & = -2n^2
                 \big\langle\bar\om_n(\eta) \hat u_{n,+}(\la,\eta) , \hat u_{n,-}(\la,\eta)\big\rangle_{\mL^2(\hT_n)}  \vphantom{\int^1} 
  \\
  & \quad 
    -2\ga   n^2  \big\|\hat u_{n,-}(\la,\eta) \big\|^2_{\mL^2(\hT_n)} \\ & \quad  +  \big\langle{\frak R}_{n,u,-}(\la,\eta) , \hat u_{n,-}(\la,\eta)\big\rangle_{\mL^2(\hT_n)} . \vphantom{\int}  \end{split}
\end{equation}
Summing \eqref{wnpm}--\eqref{un-1} we obtain (recall the definition of the Dirichlet form $\cD_n$ in \eqref{eq:cD})
\begin{equation}
  \label{072302-26aa} \begin{split}
   &
  \frac{\la}{2} \sum_{\iota=\pm} \big\| \hat
     w_{n,\iota}(\la,\eta)\big\|_{\mL^2(\hT_n)}^2 +\la\sum_{\iota=\pm}\big\| \hat
     u_{n,\iota}(\la,\eta)\big\|_{\mL^2(\hT_n)}^2 +2\ga
     n^2  \big\|\hat u_{n,-}(\la,\eta) \big\|^2_{\mL^2(\hT_n)}   \\
   &
   \quad  +2\ga n^2 {\cal
            D}_n\big(   \hat w_{n,0} -  \hat u_{n,+}\big)(\la,\eta) +2\ga n^2 {\cal
            D}_n\big(   \delta\hat w_{n}(\la,\eta)\big)  \vphantom{\int_0}
    \\
     &
            = \frac{in }{2}\hsum \delta_{n}\om(\eta,k)\Big[|\hat w_{n,-}(\la,\eta,k)|^2 -
              |\hat w_{n,+}(\la,\eta,k)|^2\Big]   \\
   &
    \quad  +4in^2{\rm Im}\,\Big( \big\langle \bar\om_n(\eta) \hat u_{n,-}(\la,\eta) , \hat
     u_{n,+}(\la,\eta)\big\rangle_{\mL^2(\hT_n)}  \Big)  \\
       &
         \quad +\frac{1}{2} \sum_{\iota=\pm} \big\langle\tilde
     W_{n,\iota}(0,\eta ),  \hat w_{n,\iota}(\la,\eta)\big\rangle
     _{\mL^2(\hT_n)} + \sum_{\iota=\pm} \big\langle\tilde
     U_{n,\iota}(0,\eta ),  \hat u_{n,\iota}(\la,\eta)\big\rangle
         _{\mL^2(\hT_n)}  \\
        &
   \quad  +  \frac12\sum_{\iota=\pm}\big\langle{\frak R}_{n,w,\iota}(\la,\eta) ,\hat w_{n,\iota}(\la,\eta) \big\rangle _{\mL^2(\hT_n)}  + \sum_{\iota=\pm} \big\langle{\frak R}_{n,u,\iota}(\la,\eta) , \hat u_{n,\iota}(\la,\eta)\big\rangle_{\mL^2(\hT_n)} . \end{split}\end{equation}
     We first take the expectation w.r.t.~$\E_n$ and the real parts of both sides of the equality, and we use 
     the Cauchy-Schwarz inequality for the scalar product $\lsc \cdot
     \rsc$ in the last four terms.
     %
     Then, from Propositions \ref{prop030204-26}, \ref{prop012302-26} and \ref{prop022302-26},  we conclude
     the following:
      there exists $ C >0$ such that
       \begin{multline}
         \label{062302-26bis}
 2\ga n^2\bbE_n \Big[{\cal
            D}_n\big(   \big(\hat w_{n,0}-  \hat u_{n,+}\big)(\la,\eta)\big)\Big] \\  +2\ga n^2 \bbE_n \Big[ {\cal
            D}_n\big(   \delta\hat w_{n}(\la,\eta)\big)\Big]  +2\ga
     n^2 \bbE_n  \big[\|\hat u_{n,-}(\la,\eta) \|^2_{\mL^2(\hT_n)}\big]  
          \\ \le   \frac{C}{({\rm Re}\,\la)^{1/2}} \Big(1+\frac{1}{{\rm
           Re}\,\la} \Big)^{1/2}\bigg\{  1
         + n\Big(1+\frac{1}{{\rm Re}\,\la} \Big)^{1/2}\bigg\}. \end{multline}
         Since
       $\hat\sum_k\hat w_{n,+}(\la,\eta,k)=   \hat\sum_k\hat
           w_{n,-}(\la,\eta,k)$,
       we have
       \[
         {\cal
            D}_n\big(  \delta\hat w_{n}(\la,\eta) \big)   =\big\|
         \delta\hat w_{n}(\la,\eta) \big\|_{\mL^2(\hT_n)}^2.
       \]
     {Using the fact that $a+1\le 2a$ when $a \ge 1$ (with $a=n(1+1/{\rm Re}\,\la)^{1/2})$ we conclude that the right hand side of \eqref{062302-26bis} is bounded by the right hand side of \eqref{062302-26}, and we conclude the proof.}
      \qed

\medskip

\proof[Proof of Lemma \ref{prop042302-26}] Multiplying scalarly both sides of \eqref{un-} by $\hat
       u_{n,+}(\la,\eta,k) $ we get
       \begin{align*}
   &
 \la \big\langle \hat u_{n,-}(\la,\eta) , \hat
                  u_{n,+}(\la,\eta)\big\rangle_{\mL^2(\hT_n)}  -\big\langle
                  \tilde U_{n,-}(0,\eta) , \hat
                  u_{n,+}(\la,\eta)\big\rangle_{\mL^2(\hT_n)}         \\
         &
           = -2n^2
                \hsum  \bar\om_n(\eta,k)| \hat u_{n,+}(\la,\eta,k) |^2
  -2\ga   n^2 \big\langle \hat u_{n,-}(\la,\eta) , \hat u_{n,+}(\la,\eta)\big\rangle_{\mL^2(\hT_n)} \\ & \quad  +\big\langle {\frak R}_{n,u,-}(\la,\eta)  ,\hat u_{n,+}(\la,\eta)\big\rangle_{\mL^2(\hT_n)}. 
       \end{align*}
{Taking the expectation w.r.t.~$\bbE_n$ and using the fact that}  $\bar \omega_n(\eta,k) \ge \omega_0$, we obtain
 \begin{align*}
   2n^2\om_0 &{\big\lnm \hat u_{n,+}(\la)  \big\rnm_\eta^2}  \\
             & \le |\la|\; \big| \lsc  \hat u_{n,-}(\la) , \hat
                  u_{n,+}(\la)\rsc \big|  +\big|\lsc
                  \tilde U_{n,-}(0) , \hat
                  u_{n,+}(\la)\rsc \big|       \\
            & \quad  +2\ga n^2 \, \big| \lsc \hat u_{n,-}(\la) , \hat
           u_{n,+}(\la)\rsc \big| +\big|\lsc {\frak
           R}_{n,u,-}(\la)  ,\hat
           u_{n,+}(\la)\rsc\big|.
 \end{align*}
 {Now we use Young's inequality $|ab| \le  |a|^2/(2\eps) + \eps |b|^2/2$, for $\eps >0$, twice: first, inside the product $(2\ga n^2)|\hat u_{n,-}\hat u_{n,+}|$ chosing $\eps=\omega_0/(2\ga)$, and second, inside the product $|{\frak R}_{n,u,-} \hat u_{n,+}|$ with $\eps=n^2\omega_0$, and we get}
  \begin{align*}
2n^2\om_0 {\big\lnm \hat u_{n,+}(\la)  \big\rnm_\eta^2}  &
\le {|\la|\; \big| \lsc  \hat u_{n,-}(\la) , \hat
	u_{n,+}(\la)\rsc \big|}+\big|\lsc
\tilde U_{n,-}(0) , \hat
u_{n,+}(\la)\rsc \big|         \\ & \quad   + {\frac{2\ga^2 n^2}{\om_0} } \; \big\lnm \hat
           u_{n,-}(\la)\big\rnm_\eta^2   +     \frac{n^2\om_0}{2}\big\lnm \hat
            u_{n,+}(\la)\big\rnm_\eta^2 \\ & \quad  + {\frac{1}{2n^2\om_0}}\big\lnm  {\frak R}_{n,u,-}(\la) \big\rnm_\eta^2+
            \frac{n^2\om_0}{2}\big\lnm \hat
      u_{n,+}(\la)\big\rnm_\eta^2.
     \end{align*}
  As before we use the Cauchy-Schwarz inequality for the scalar product $\lsc \cdot \rsc$ in the first two terms, and we get \begin{align*}\big|\lsc
  		\tilde U_{n,-}(0) , \hat
  		u_{n,+}(\la)\rsc  \big|  & \le \big\lnm \tilde
                                           U_{n,-}(0)\big\rnm_\eta \;
                                           \big\lnm \hat
                                           u_{n,+}(\la)\big\rnm_\eta
                                           \le \frac{C(1+{\rm
                                           Re}\,\la)^{1/2}}{{\rm
                                           Re}\,\la} \\ \text{and}\quad 
 \big| \lsc  \hat u_{n,-}(\la) , \hat
 u_{n,+}(\la)\rsc \big|  & \le  \big\lnm \hat
  u_{n,-}(\la)\big\rnm_\eta \; \big\lnm \hat
                           u_{n,+}(\la)\big\rnm_\eta\le  \frac{C
                           (1+{\rm Re}\,\la)}{({\rm Re}\,\la)^2}
                           ,\end{align*} from Propositions
                         \ref{prop030204-26} and
                         \ref{prop012302-26}. Then, from the
                         inequality above,  grouping the terms
                         containing $ \lnm \hat
                         u_{n,+}(\la)\rnm_\eta^2$ together, and then
                         using Proposition \ref{prop022302-26}  and Lemma 
  \ref{prop032302-26}, we conclude that:
  	\begin{align}
  	\label{un-2bis}
 	n^2\om_0\big\lnm \hat u_{n,+}(\la)\big\rnm_\eta^2\le  &
 \frac{C |\la| (1+{\rm Re}\,\la)}{({\rm Re}\,\la)^2}+\frac{C(1+{\rm
                                           Re}\,\la)^{1/2}}{{\rm
    Re}\,\la} \\
          &
            +\frac{Cn}{ ({\rm Re}\,\la)^{1/2}} \Big(1+\frac{1}{{\rm
           Re}\,\la} \Big) +C \Big(1+\frac{1}{{\rm
				Re}\,\la}\Big).\notag
        \end{align}
      { Using the fact that $\sqrt{b}\le b$ when $b\ge 1$ (with $b=1+{\rm Re}\,\la$), we obtain \eqref{un-2}.}
   \qed
 
 \subsubsection{Auxiliary corollary}
 
 In this section we give an important corollary from last estimates:

\begin{corollary}
   \label{cor012502-26}
  There exists $C >0$ such that,  for any $\eta \in\bbZ$, $n\ge
  1$, and $\la$ s.t.~${\rm Re}\,\la>0$, 	\begin{align}
 		\label{012502-26}
 		|{\rm Im}\,\la\,|\sum_{\iota=\pm}\Big(\big\lnm \hat
 		w_{n,\iota}(\la)\big\rnm_\eta+\big\lnm \hat
 		u_{n,\iota}(\la)\big\rnm_\eta \Big)
 		\le  \frac{Cn}{({\rm Re}\,\la)^{1/4}}\Big(1+\frac{1}{{\rm
           Re}\,\la} \Big) .
                                        \end{align}
                                       
 \end{corollary}

\proof[Proof of Corollary \ref{cor012502-26}] We come back to the identity \eqref{072302-26aa} and now take the imaginary parts on both sides. We get
\begin{align}
  \label{072302-26}
  \frac{{\rm Im}\,\la}{2} \sum_{\iota=\pm}&\| \hat
     w_{n,\iota}(\la,\eta)\|^2_{\mL^2(\hT_n)} +{\rm Im}\,\la\sum_{\iota=\pm}\| \hat
     u_{n,\iota}(\la,\eta)\|^2_{\mL^2(\hT_n)} \notag
  \\
   &
     = \frac{n }{2}\hsum\delta_{n}\om(\eta,k)\Big[|\hat w_{n,-}(\la,\eta,k)|^2 -
              |\hat w_{n,+}(\la,\eta,k)|^2\Big]   \\
   &
     \quad +2n^2 {\rm Im}\,\Big( \langle \bar\om_n(\eta) \hat u_{n,-}(\la,\eta) , \hat
     u_{n,+}(\la,\eta)\rangle_{\mL^2(\hT_n)} \Big) .\notag
\end{align}
We have, by the identity $a^2-b^2=(a-b)(a+b)$ and Cauchy-Schwarz inequality,
\begin{multline*}
\bigg|\hsum\delta_{n}\om(\eta,k)\Big[|\hat w_{n,-}(\la,\eta,k)|^2 -
                 |\hat w_{n,+}(\la,\eta,k)|^2\Big]  \bigg| \\
    \le  \|\om'\|_\infty \big\|\delta\hat w_{n}(\la,\eta)\big\|_{\mL^2(\hT_n)}\sum_{\iota=\pm}\|\hat w_{n,\iota}(\la,\eta)\|_{\mL^2(\hT_n)}.
\end{multline*}
Hence, since $\bar\omega_n \le (\omega_0^2+4)^{1/2}$, and using again Cauchy-Schwarz inequality, we deduce
\begin{align}
  \label{072302-26a}
  |{\rm Im}\,\la|\Big(\sum_{\iota=\pm}\| \hat
     w_{n,\iota}(\la,\eta)&\|_{\mL^2(\hT_n)}^2 +\sum_{\iota=\pm}\| \hat
     u_{n,\iota}(\la,\eta)\|_{\mL^2(\hT_n)}^2\Big)  
  \\
   &
     \le Cn \big\|\delta\hat w_{n}(\la,\eta)\big\|_{\mL^2(\hT_n)}\sum_{\iota=\pm}\|\hat w_{n,\iota}(\la,\eta)\|_{\mL^2(\hT_n)}  \notag\\
   &
    \quad +Cn^2 \|\hat u_{n,-}(\la,\eta) \|_{\mL^2(\hT_n)} \, \|\hat
     u_{n,+}(\la,\eta) \|_{\mL^2(\hT_n)} \notag.
\end{align}
{Applying the expectation to both sides and using \eqref{062302-26}, we
obtain
\begin{multline}
  \label{072302-26z}
  |{\rm Im}\,\la| \Big(\sum_{\iota=\pm}\big\lnm \hat
 		w_{n,\iota}(\la)\big\rnm_\eta^2+\sum_{\iota=\pm}\big\lnm \hat
 		u_{n,\iota}(\la)\big\rnm_\eta^2 \Big)
  \\
     \le \frac{C }{({\rm Re}\,\la)^{1/4}} \Big(1+\frac{1}{{\rm
     Re}\,\la} \Big)^{1/2}(n^{1/2}+n) \Big(\sum_{\iota=\pm}\big\lnm \hat
 		w_{n,\iota}(\la)\big\rnm_\eta^2+\sum_{\iota=\pm}\big\lnm \hat
 		u_{n,\iota}(\la)\big\rnm_\eta^2 \Big)^{1/2}
\end{multline}
and \eqref{012502-26} follows.}
\qed

\section{Proof of Proposition \ref{prop:rest}}
\label{sec:proofrest}

  \subsection{Decomposition of the summation}
In order to prove \eqref{eq:limsum}, we first obtain an  expression for \[\hsum  \hat
w_{n,+}(\la,\eta,k)\] coming from the resolution of the Laplace-Wigner
system of equations \eqref{wn+}--\eqref{un-} that will allow us to
estimate its second moment in terms of the bounds obtained in Sections
\ref{sec:prel} and \ref{sec:steps}.

Summing up over $k$ of both sides of \eqref{wn+}  we obtain
\begin{multline}
  \label{041902-26}
  \frac{\bD_n(\la,\eta)}{n^2} \underset{\ell\in\hT_n}{\hat\sum}  \hat    w_{n,+}(\la,\eta,\ell)  =    \hsum  \frac{\tilde  W_{n,+}(0,\eta,k)}{\la +in
     \delta_{n}\om(\eta,k) +\ga n^2 }     
  \\
        -\hsum \frac{ n^2 }{\la +in \delta_{n}\om(\eta,k) +\ga n^2 } \,{\cal
     L}_n \hat u_{n,+}(\la,\eta,k) + \hsum \frac{{\frak R}_{n,w,+}(\la,\eta,k) }{\la +in \delta_{n}\om(\eta,k) +\ga n^2 },
\end{multline}
where 
\begin{equation} \label{eq:defD}
	\bD_n(\la,\eta):= n^2 \hsum \bigg( 1 - \frac{\gamma}{\frac\la{n^2}+\frac{i\delta_n\omega(\eta,k)}{n} + \ga} \bigg).
	\end{equation}
 Thanks to the fact that $\hat{\sum}_k{\cal L}_nf(k)=0$ we have
\[
\hsum \frac{ \la +in \delta_{n}\om(\eta,k) +\ga n^2 }{\la +in \delta_{n}\om(\eta,k) +\ga n^2 } {\cal
     L}_n \hat   u_{n,+}(\la,\eta, k) =\hsum  {\cal
     L}_n \hat u_{n,+}(\la,\eta,k)   =0.
  \]
  We also have, from the definition of ${\frak R}_{n,w,+}$, see \eqref{Rn},
  \begin{align} 
  	&\hsum \frac{ \la +in \delta_{n}\om(\eta,k) +\ga n^2
  	}{\la +in \delta_{n}\om(\eta,k) +\ga n^2 }{\frak
  		R}_{n,w,+}(\la,\eta,k) =\hsum {\frak
  		R}_{n,w,+}(\la,\eta,k) \label{eq:sum0}\\
  	& =   \underset{\ell'\in\hT_n}{\hat\sum} \int_0^{+\infty}e^{-\la t} \hsum
  	\Big(  i U_{n,-}     
  	(t-,\eta-n\ell',k+\ell' )
  	-i U_{n,-}    (t-,\eta-n\ell',k) \Big)
  	\dd\thN  (\gamma n^2 t,\ell') \notag 
  \\ &	=0. \vphantom{\int}   \notag 
  \end{align}
   Therefore, \eqref{041902-26} rewrites as
     \begin{align}
  \label{051902-26}
   \frac{\bD_n(\la,\eta)}{n^2}  \underset{\ell\in\hT_n}{\hat\sum}   \hat w_{n,+}(\la,\eta,\ell)  
    =  &  \hsum  \frac{\tilde  W_{n,+}(0,\eta,k)}{\la  +in
     \delta_{n}\om(\eta,k) +\ga n^2 }        \\
   & +\frac{1}{\ga} \hsum  \frac{  \la+in
     \delta_{n}\om(\eta,k)    }{\la +in \delta_{n}\om(\eta,k) +\ga n^2 } {\cal
        L}_n \hat u_{n,+}(\la,\eta,k) \notag \\ 
      & -\frac{1}{\ga n^2} \hat{\sum}_{k}\frac{(\la +in
      	\delta_{n}\om(\eta,k)  ){\frak R}_{n,w,+}(\la,\eta,k)
      }{\la +in \delta_{n}\om(\eta,k) +\ga n^2 }. \notag \end{align}
 As a result
\begin{equation}
  \label{061902-26}
   \underset{\ell\in\hT_n}{\hat\sum}   \hat    w_{n,+}(\la,\eta,\ell)  
    =   {\rm I}_n(\la,\eta)  +
     {\rm II}_n(\la,\eta) +   {\rm III}_n(\la,\eta), \end{equation}
     where \begin{align}
      {\rm I}_n(\la,\eta) & :=\bD_n^{-1}(\la,
     \eta)\hsum  \frac{\tilde  W_{n,+}(0,\eta,k)}{\frac{\la}{n^2}
     +\frac{i \delta_{n}\om(\eta,k)}{n} +\ga  }  , \label{eq:defI}\\
    {\rm II}_n(\la,\eta)&:=  -\frac{\bD_n^{-1}(\la,
     \eta)}{\ga} \hsum \frac{  \la+in
     \delta_{n}\om(\eta,k)    }{\frac{\la}{n^2}  +\frac{i \delta_{n}\om(\eta,k)}{n} +\ga  } {\cal
     L}_n \hat u_{n,+}(\la,\eta,k), \label{eq:defII} \\
     {\rm III}_n(\la,\eta)&:= - \frac{\bD_n^{-1}(\la,
     \eta)}{\ga} \hsum \frac{(\la +in
         \delta_{n}\om(\eta,k) ){\frak R}_{n,w,+}(\la,\eta,k) }{\la +in \delta_{n}\om(\eta,k) +\ga n^2 }. \label{eq:defIII}
\end{align}
We   estimate the   moments of  each term   ${\rm I}_n$ (the
  first moment) and ${\rm II}_n$, ${\rm III}_n$  (the
  second moments)  separately, in the following sections. Let us
state here the   statements that we are going to prove:

\begin{lemma} \label{lem:limI} For any $\la$ such that ${\rm Re}\,\la
  >0$ {and any $\eta\in\bbZ$}
	\begin{equation}
		\label{eq:limI} \lim_{n\to+\infty}  \E_n\Big[\big|{\rm I}_n(\la,\eta)\big| \Big] = 0. 
	\end{equation}
	
	\end{lemma}
	
	\begin{lemma} \label{lem:limII} For any $\la$ such that ${\rm Re}\,\la >0$ and any $\eta\in\bbZ$,
		\begin{equation}
			\label{eq:limII} \lim_{n\to+\infty}  \E_n\Big[\big|{\rm II}_n(\la,\eta)\big|^2 \Big] = 0. 
		\end{equation}
		
	\end{lemma}

	\begin{lemma} \label{lem:limIII} For any $\la$ such that ${\rm Re}\,\la >0$ and any $\eta\in\bbZ$,
		\begin{equation}
			\label{eq:limIII} \lim_{n\to+\infty}  \E_n\Big[\big|{\rm III}_n(\la,\eta)\big|^2 \Big] = 0. 
		\end{equation}
		
	\end{lemma}

	Lemmas \ref{lem:limI}, \ref{lem:limII} and \ref{lem:limIII} are proved in Sections \ref{sec:proofI}, \ref{sec:proofII} and \ref{sec:proofIII}, respectively. 

\subsection{Proof of Lemma \ref{lem:limI}} \label{sec:proofI}

Recall the definition of ${\rm I}_n(\lambda,\eta)$ in \eqref{eq:defI}. We first need to understand the limit of the coefficient $\bD_n(\la,\eta)$, and we claim:

\begin{lemma} \label{lem:D}
	For any $\la$ such that ${\rm Re}\,\la >0$, any $\eta \in\bbZ$, 
		\begin{equation}
		\label{140402-26}
		{\rm Re}\,\bD_n(\la,  \eta)\ge   \frac{\ga{\rm
				Re}\,  \la }{\Big[\ga+\frac{1}{n^2}{\rm
				Re}\lambda \Big]^2 +\sup_{k \in\hT_n} [\om'( k)]^2}
	\end{equation}
and
	\begin{equation}
		\label{150402-26}
		\lim_{n\to+\infty}\bD_n(\la,  \eta)=\frac{1}{\gamma}\big(\la+\bD(2\pi \eta)^2\big),
	\end{equation}
	with $\bD$ defined in \eqref{042303-26} as in Theorem \ref{thm022303-26}. 
\end{lemma}

 	\proof
We have
\begin{align*}
	\bD_n(\la,  \eta)=n^2\underset{k\in\hat\T_n}{\hat{\sum}}  \frac{\frac{\la}{n^2}  +\frac{i \delta_{n}\om(\eta,k)}{n}  
	}{\frac{\la}{n^2}  +\frac{i \delta_{n}\om(\eta,k)}{n} +\ga  } .
\end{align*}
Hence,
\begin{equation}
	\label{140402-26z}
	{\rm Re}\,\bD_n(\la,  \eta)\ge n^2 \underset{k\in\hat\T_n}{\hat{\sum}} \frac{{\rm
			Re}\,\big(\frac{\la}{n^2}\big)\Big[\ga+{\rm
			Re}\,\big(\frac{\la}{n^2}\big)\Big]+\Big(\frac{\delta_n\om(\eta,k)}{n}\Big)^2}{\Big[\ga+{\rm
			Re}\,\big(\frac{\la}{n^2}\big)\Big]^2 +\Big(\frac{\delta_n\om(\eta,k)}{n}\Big)^2}
\end{equation}
for any ${\rm Re}\,\la>0$, $\eta\in\bbZ$ and \eqref{140402-26} follows.
We can write
\begin{align*}
	&\bD_n(\la,  \eta)=\frac{\la}{\ga}
	+\bar \bD_n(\la,  \eta)
	+o_n(\la,  \eta),
\end{align*}
with
\begin{align*}
	&
	\bar \bD_n(\la,  \eta) :=n\underset{k\in\hat\T_n}{\hat{\sum}}  \frac{  i \delta_{n}\om(\eta,k)  
	}{\frac{\la}{n^2}  +\frac{i \delta_{n}\om(\eta,k)}{n} +\ga
	}\qquad\mbox{and}\qquad 
	\lim_{n\to+\infty}o_n(\la,  \eta)=0.
\end{align*}
Next, since $\hat\sum_{k\in\hT_n} \delta_n\omega(\eta,k) = 0$, we have 
\begin{align*}
	\bar \bD_n(\la,  \eta)&:=  n\underset{k\in\hat\T_n}{\hat{\sum}}  \bigg(\frac{  i \delta_{n}\om(\eta,k)  
	}{\frac{\la}{n^2}  +\frac{i \delta_{n}\om(\eta,k)}{n} +\ga  }-\frac{  i \delta_{n}\om(\eta,k)  
	}{\frac{\la}{n^2}   +\ga  }\bigg)\\
	&
	=-\underset{k\in\hat\T_n}{\hat{\sum}} \frac{  \big( i \delta_{n}\om(\eta,k)  \big)^2
	}{\Big[\frac{\la}{n^2}  +\frac{i \delta_{n}\om(\eta,k)}{n} +\ga
		\Big]\big[\frac{\la}{n^2}   +\ga \big]}
	=\Big(\frac{\eta}{\ga}\Big)^2\int_{\bbT}[\om'(k)]^2\dd k +o_n(\la,  \eta)
	,
\end{align*}
thus \eqref{150402-26} follows.
\qed


Using estimate \eqref{140402-26}  and the
convergence result \eqref{150402-26} we
conclude easily that for any ${\rm Re}\la>0$ and any
  $\eta\in\bbZ$, we have
\begin{align*}
     \lim_{n\to+\infty}\bbE_n\Big[\big|{\rm I}_n(\la,\eta) -\bar {\rm I}_n(\la,\eta) \big|\Big]=0,\end{align*}
     where \begin{align*}
     \bar {\rm I}_n(\la,\eta)  := (\la + \bD(2\pi\eta)^2)^{-1}\hsum  \tilde W_{n,+}(0,\eta,k).
\end{align*}
Invoking equality \eqref{eq:wigner_energy} and assumption \eqref{eq:1} we conclude that
%
%
%
%
%
%
%
%
\begin{align}
  \label{030704-26}
       \lim_{n\to+\infty}\bbE_n\Big[\big|\bar{\rm I}_n(\la,\eta)   \big|\Big]=0 .
\end{align}
This ends the proof of Lemma \ref{lem:limI}. \qed

\subsection{Proof of Lemma \ref{lem:limII}: Limit of ${\rm II}_n$} \label{sec:proofII} 

Recall the definition of $\mathrm{II}_n$ in \eqref{eq:defII}.  	
We start with the following.
\begin{lemma}
	\label{lm010303-26}
		For any $\la$ such that ${\rm Re}\,\la >0$ and $\eta \in\bbZ$, we have
	\begin{equation}
		\label{010303-26}
		\underset{k\in\hat\T_n}{\hat{\sum}} \delta_n\omega(\eta,k)\; {\cal
			L}_n \hat u_{n,+}(\la,\eta,k) =0.
	\end{equation}
\end{lemma}
\proof
Recall from \eqref{020204-26} that
\begin{align*}
	 U_{n,+}(t,\eta,k) :=\frac{1}{4n}  
	\Big[\widehat\psi\Big(n^2t,k+\frac{\eta}{n}\Big)
	\widehat\psi (n^2t,-k) +
	\widehat\psi^\star\Big(n^2t,-k-\frac{\eta}{n}\Big)
	\widehat\psi ^\star(n^2t,k)\Big].
\end{align*}
Hence
\begin{align}
  \label{tAn}
	\mathrm{A}_n & := \underset{k\in\hat\T_n}{\hat{\sum}} \Big(\om\Big(k+\frac{\eta}{n}\Big)-\om(k)\Big)  {\cal
		L}_n   U_{n,+}(t,\eta,k) \notag
	\\ & =-\underset{k\in\hat\T_n}{\hat{\sum}} \Big(\om\Big(k+\frac{\eta}{n}\Big)-\om(k)\Big)
	U_{n,+}(t,\eta,k)   \\
	&
	=-\frac{1}{4n}  \underset{k\in\hat\T_n}{\hat{\sum}} \Big(\om\Big(k+\frac{\eta}{n}\Big)-\om(k)\Big)
	\Big[\widehat\psi\Big(n^2t,k+\frac{\eta}{n}\Big)
	\widehat\psi (n^2t,-k) \notag\\ & \qquad \qquad \qquad \qquad \qquad \qquad \qquad +
	\widehat\psi^\star\Big(n^2t,-k-\frac{\eta}{n}\Big)
	\widehat\psi ^\star(n^2t,k)\Big].\notag
\end{align}
Substituting $k':=-k-\frac{\eta}{n}$ the right hand side can be
rewritten as 
\begin{align*}
\mathrm{A}_n	& =
	-\frac{1}{4n}  \underset{k\in\hat\T_n}{\hat{\sum}} \Big(\om(k)-\om\Big(k+\frac{\eta}{n}\Big)\Big)
	\Big[\widehat\psi\Big(n^2t,-k\Big)
	\widehat\psi \Big(n^2t, k+\frac{\eta}{n}\Big)  \\ & \qquad \qquad \qquad \qquad \qquad \qquad \qquad +
	\widehat\psi ^\star(n^2t,k) \widehat\psi^\star\Big(n^2t,
	-k-\frac{\eta}{n}\Big)\Big]
	=-\mathrm{A}_n.
\end{align*}
We conclude that $\mathrm{A}_n=0$. In the same way
$\tilde{\mathrm{A}}_n=0$, where $\tilde{\mathrm{A}}_n$ is defined  
replacing $U_{n,+}$ by $\tilde U_{n,+}$ in \eqref{tAn}. Taking then the Laplace transform  we conclude \eqref{010303-26}. \qed

Let us now decompose
\begin{align*}
	\mathrm{II}_n(\la,\eta)=A_n^{(1)}(\la,\eta)+A_n^{(2)}(\la,\eta), \end{align*} where 
	\begin{align*}
	A_n^{(1)}(\la,\eta)&= \frac{\bD_n^{-1}(\la,
		\eta)}{\ga} \underset{k\in\hat\T_n}{\hat{\sum}} \frac{  \la+in
		\delta_{n}\om(\eta,k)    }{\frac{\la}{n^2} +\ga  } {\cal
		L}_n \hat u_{n,+}(\la,\eta,k), \\
	A_n^{(2)}(\la,\eta)&= \frac{\bD_n^{-1}(\la,
		\eta)}{\ga} \underset{k\in\hat\T_n}{\hat{\sum}} \frac{  
		[\delta_{n}\om(\eta,k)]^2 -  \la\frac{i \delta_{n}\om(\eta,k)}{n} }{\big(\frac{\la}{n^2} +\frac{i \delta_{n}\om(\eta,k)}{n} +\ga\big) \big(\frac{\la}{n^2} +\ga \big) } {\cal
		L}_n \bar u_{n,+}(\la,\eta,k).
\end{align*}
Using \eqref{un-2}, {together with Cauchy-Schwarz inequality and the fact that  (recall \eqref{030204-26}) $\|\cL_n f \|_{\mL^2(\hT_n)} \le \|f \|_{\mL^2(\hT_n)}$}, we conclude that there exists $C>0$ such that 
\[
\sup_{\eta\in\bbZ}\bbE_n \Big[|A_n^{(2)}(\la,\eta)|^2\Big]\le\frac{C}{n^2}.\]
Thanks to Lemma \ref{lm010303-26} we obtain
\begin{align*}
	A_n^{(1)}(\la,\eta)
	= \frac{\bD_n^{-1}(\la,
		\eta)}{\ga}\underset{k\in\hat\T_n}{\hat{\sum}} \frac{  \la    }{\frac{\la}{n^2} +\ga  } {\cal
		L}_n \hat u_{n,+}(\la,\eta,k).
\end{align*}
Hence, using again  \eqref{un-2}, we also have
\[
\sup_{\eta\in\bbZ}\bbE_n \Big[|A_n^{(1)}(\la,\eta)|^2\Big]\le\frac{C}{n^2}.\]
Thus \eqref{eq:limII} follows, concluding the proof of Lemma \ref{lem:limII}.\qed

\subsection{Proof of Lemma \ref{lem:limIII}: Limit of ${\rm III}_n$}

\label{sec:proofIII}

We are left with the problem of estimating
$\bbE_n\big[|{\rm III}_n(\la,\eta)|^2\big]$ and this is the most involved part. We can write
 \begin{align*}
   {\rm III}_n(\la,\eta)=R_n^{(1)}(\la,\eta)+R_n^{(2)}(\la,\eta),\end{align*}
   where \begin{align*}
    R_n^{(1)}(\la,\eta)&:= -\frac{\bD_n^{-1}(\la,
     \eta)}{\ga} \hsum \frac{  \big(\la+in
     \delta_{n}\om(\eta,k)\big){\frak R}_{n,w,+}(\la,\eta,k)     }{ \la
     +\ga n^2  }, \\
     R_n^{(2)}(\la,\eta)&:= \frac{\bD_n^{-1}(\la,
     \eta)}{\ga} \hsum \frac{ \la i n\delta_{n}\om(\eta,k)  -n^2 
     [\delta_{n}\om(\eta,k)]^2 
     }{\big( \la +i n \delta_{n}\om(\eta,k) +\ga n^2\big) \big(\la
     +\ga n^2 \big) } {\frak R}_{n,w,+}(\la,\eta,k) .
 \end{align*}
 We start with $R_n^{(2)}$ which is easier: we have
 \begin{multline*}
    \bbE_n\Big[\big| R_n^{(2)}(\la,\eta)\big|^2\Big]=\frac{|\bD_n^{-2}(\la,
     \eta)|}{\ga^2} \underset{k,k'\in\hT_n}{\hat\sum} \frac{ \la i n\delta_{n}\om(\eta,k)  -n^2 
     [\delta_{n}\om(\eta,k)]^2 
     }{\big( \la +i n \delta_{n}\om(\eta,k) +\ga n^2\big) \big(\la
     +\ga n^2 \big) }\\
     \times \bigg(\frac{ \la i n\delta_{n}\om(k',\eta)  -n^2 
     [\delta_{n}\om(k',\eta)]^2 
     }{\big( \la +i n \delta_{n}\om(k',\eta) +\ga n^2\big) \big(\la
     +\ga n^2 \big) }\bigg)^\star\bbE_n\Big[ {\frak R}_{n,w,+}(\la,\eta,k) {\frak R}_{n,w,+}^\star(\la,\eta,k')\Big] .
 \end{multline*}
%
 Using the fact that
 \[   \Big| \frac{ \la i n\delta_{n}\om(\eta,k)  -n^2 
 	[\delta_{n}\om(\eta,k)]^2 
 }{\big( \la +i n \delta_{n}\om(\eta,k) +\ga n^2\big) \big(\la
 	+\ga n^2 \big) } \Big|\le\frac{C}{n^2}\] together with estimate \eqref{032302-26}
 we get 
\begin{align}
    \bbE_n\Big[\big| R_n^{(2)}(\la,\eta)\big|^2\Big]  \le  \frac{C}{n^2}. \label{eq:Rn2}
 \end{align} 
Therefore, to finish the proof of Lemma \ref{lem:limIII} we are left to proving
  \begin{equation}
    \label{101003-26}
    \lim_{n\to+\infty} \bbE_n\Big[\big| R_n^{(1)}(\la,\eta)\big|^2\Big]=0.
    \end{equation}
First of all, recall from \eqref{eq:sum0} that 
\begin{align*} 
	\hsum {\frak R}_{n,w,+}(\la,\eta,k)  =0.
\end{align*}
therefore 
we can write 
\begin{align}
  \label{050704-26}
     R_{n}^{(1)}(\la,\eta)  = -\frac{\bD_n^{-1}(\la,
     \eta)}{\ga} \hsum \frac{  in
     \delta_{n}\om(\eta,k) {\frak R}_{n,w,+}(\la,\eta,k)     }{ \la
     +\ga n^2  }. 
\end{align}
Then, as in Section \ref{sec:proofI} we replace $\bD_n$ with its limit: we let
\begin{equation}
  \label{060704-26}
  \bar R_{n}^{(1)}(\la,\eta)= - {(\la + \bD(2\pi\eta)^2)^{-1}} \hsum \frac{  in
     \delta_{n}\om(\eta,k){\frak R}_{n,w,+}(\la,\eta,k)     }{ \la
     +\ga n^2  }.
 \end{equation}
By virtue of \eqref{032302-26}, for  fixed $\eta$
   and ${\rm Re}\,\la>0$   there exists $C>0$ such that  $| R_{n}^{(1)}(\la,\eta)-\bar
   R_{n}^{(1)}(\la,\eta)|\le C$. Thus, by the dominated convergence Theorem
  we conclude  that
 \begin{align*}
   &
   \lim_{n\to+\infty}  \bbE_n\Big[\big| R_{n}^{(1)}(\la,\eta)-  
     \bar R_{n}^{(1)}(\la,\eta) \big|^2\Big] =0 .
 \end{align*}
 To show \eqref{101003-26} it suffices therefore to prove that
 \begin{align}
   \label{101003-26a}
   \lim_{n\to+\infty}  \bbE_n\Big[\big|  
     \bar R_{n}^{(1)}(\la,\eta) \big|^2\Big] =0 .
 \end{align}

\subsubsection{Intermediate steps and conclusion}

{Recall the definition of the stochastic terms ${\frak R}_{n,{\rm
      x},\iota}$ in \eqref{Rn} and the functions ${\frak r}_{n,{\rm
      x}, \iota}$ from \eqref{eq:compact}. We already proved an
  estimate for those terms in Proposition \ref{prop022302-26}. In the
present section we derive sharper estimates for these quantities.}
For $\iota\in\{-,+\}$ we let
\begin{equation}
  \label{030703-26a}
  \begin{split}
&
  {\frak W}_{n,\pm}^{w,\iota}:= \pm  i U_{n,-}  +{\cal L}_n\big( 
    W_{n,0} -U_{n,+} \big)  -\iota \delta W_{n} ,\\
    &
     {\frak U}_{n}^{u,+}:= {\cal L}_n\big(U_{n,+} - W_{n,0}  \big)
    \\
    &
     {\frak U}_{n,\pm}^{u,-}:= iU_{n,-}\pm  \delta W_{n}.
         \end{split}
       \end{equation}
Note that 
      \begin{align*}
        &
         {\frak r}_{n,w,+}(t,\eta ,k,\ell')=    {\frak W}^{w,+}_{n,-}
    (t,\eta-n\ell',k) +{\frak W}^{w,+}_{n,+} (t,\eta-n\ell',k+\ell')
        \\
        &
         {\frak r}_{n,w,-}(t,\eta ,k,\ell')=   {\frak W}^{w,-}_{n,-}
    (t-,\eta-n\ell',k) +{\frak W}^{w,-}_{n,+} (t-,\eta-n\ell',k+\ell'),
      \\
            &
         {\frak r}_{n,u,+}(t,\eta ,k,\ell')={\frak U}_{n}^{u}(t- ,\eta-n\ell',k)  
  +
              {\frak U}_{n}^{u} (t-,\eta-n\ell',k+\ell'),\\
                &
         {\frak r}_{n,u,-}(t,\eta ,k,\ell')=
 {\frak U}_{n,+}^{u,-} (t- ,\eta-n\ell',k) + {\frak U}_{n,-}^{u,-} (t-,\eta-n\ell',k+\ell').
      \end{align*}
Now, take the Laplace transforms: for any $\la$ with ${\rm Re}\,\la >0$, let
\begin{align*}
\hat {\frak r}_{n,{\rm x},\iota}(\la,\eta ,k,\ell'):=\int_0^{+\infty}e^{-\la
  t}{\frak r}_{n,{\rm x},\iota}(t,\eta ,k,\ell')\dd t ,\qquad {\rm x}\in\{w,u\}, \iota \in \{-,+\}.
\end{align*}
%
From \eqref{060704-26} we have
\begin{multline*}
	\bbE_n\Big[\big|  \bar R_{n}^{(1)}(\la,\eta) \big|^2\Big]
	=\frac{ \ga n^2|\la + \bD(2\pi\eta)^2|^{-2}}{4\ga^2} \underset{k,k'\in\hT_n}{\hat\sum} \frac{  n^2
		\delta_{n}\om(\eta,k) \delta_{n}\om(k',\eta)  }{ |\la
		+\ga n^2|^2  }\\
	\times \int_0^{+\infty} e^{-2 {\rm Re}\,\la t}\bbE_n\Big[{\frak
		r}_{n,w,+} (t,\eta ,k,0)    \big({\frak
		r}_{n,w,+} (t,\eta ,k',0)  \big)^\star\Big]\dd t.
\end{multline*}
Hence, 
\begin{align*}
	\bbE_n\Big[\big|  \bar R_{n}^{(1)}(\la,\eta) \big|^2\Big]
	& \le C \bbE_n\bigg[\int_0^{+\infty} \Big(e^{- {\rm Re}\,\la t}\big\|{\frak
		r}_{n,w,+} (t,\eta,\cdot,0)\big\|_{\mL^2(\hT_n)}\Big)^2\dd t  \bigg]\\
	&
	=\frac{C}{2\pi} \bbE_n\bigg[\int_{\R}\big\|\hat{\frak
		r}_{n,w,+} ({\rm Re}\,\la+i\xi,\eta,\cdot,0)\big\|_{\mL^2(\hT_n)}^2 \dd \xi  \bigg].
\end{align*}
The last equality follows   the Plancherel Theorem.
Therefore, \eqref{101003-26a}, thus also \eqref{101003-26}, is a
direct consequence of   the following lemma:


 \begin{lemma}
	\label{prop021003-26}
	For any $\la$ s.t.~${\rm Re}\,\la >0$,  there exists $C >0$ such that, for any $\eta \in\bbZ$, $n\ge 1$,
	\begin{equation}
		\label{081003-26}
	 \sum_{\iota=\pm} \sum_{\mathrm{x}\in\{u,w\}} \int_\R \dd \xi \Big(\bbE_n\Big[\|\hat {\frak r}_{n,\mathrm{x},\iota}(\mathrm{Re}\,\la+i\xi,\eta
	,\cdot,0)\|_{\mL^2(\hT_n)}^2\Big]\Big)  \le
		\frac{C}{n^2} .
	\end{equation}
\end{lemma}

\proof[Proof of Lemma \ref{prop021003-26}] From the definitions we immediately conclude the following: there exists a {(universal)} constant $C>0$ such that, for any $\la$ such that ${\rm Re}\,\la >0$, any $\eta \in\bbZ$, 
	\begin{align}
		\label{071003-26}
		&
		\sum_{\iota=\pm} \Big(\bbE_n\Big[\|\hat {\frak r}_{n,w,\iota}(\la,\eta
		,\cdot,0)\|_{\mL^2(\hT_n)}^2\Big]+\bbE_n\Big[\|\hat {\frak
			r}_{n,u,\iota}(\la,\eta , \cdot,0)\|_{\mL^2(\hT_n)}^2\Big]\Big) \\
		&
		\le C \bbE_n \Big[{\cal
			D}_n\Big(  \big( \hat w_{n,0} -  \hat u_{n,+}\big)(\la,\eta)\Big) +
		\|\hat u_{n,-}(\la,\eta) \|^2_{\mL^2(\hT_n)}  
		+ {\cal
			D}_n\big(  \delta \hat w_{n}(\la,\eta) \big)    \Big] . \notag
	\end{align}
Then, taking the real parts on both sides of \eqref{072302-26aa} and  
integrating over the imaginary part of variable $\la$, we
conclude that
\begin{equation}
  \label{051003-26}
  \begin{split}
   &
    2\ga n^2 \int_{\R}\bbE_n\bigg[{\cal
            D}_n\Big(\big(   \hat w_{n,0} -  \hat u_{n,+} \big) \big({\rm
          Re}\,\la+i\xi,\eta\big)\Big) +  \|\hat u_{n,-}\big({\rm
          Re}\,\la+i\xi,\eta\big)  \|^2_{\mL^2(\hT_n)}  
    \\ & \qquad \qquad \qquad + {\cal
            D}_n\big(  \delta\hat w_{n}     \big({\rm
          Re}\,\la+i\xi,\eta\big) \big)   \bigg]\dd\xi
    \\
     &
            \le   \frac{1}{2} \bigg|\int_{\R}\bbE_n\Big[\sum_{\iota=\pm} \langle\tilde
     W_{n,\iota}(0,\eta ),  \hat w_{n,\iota}\big({\rm
          Re}\,\la+i\xi,\eta\big) \rangle
     _{\mL^2(\hT_n)} \Big]\dd\xi\bigg| \\ & \quad + \sum_{\iota=\pm} \bigg|\int_{\R}\bbE_n\Big[\langle\tilde
     U_{n,\iota}(0,\eta ),  \hat u_{n,\iota}\big({\rm
          Re}\,\la+i\xi,\eta\big) \rangle
         _{\mL^2(\hT_n)}\Big]\dd\xi\bigg|\\
        &
     \quad +   \frac12\sum_{\iota=\pm} \bigg|\int_{\R}\bbE_n\Big[\big\langle{\frak R}_{n,w,\iota}\big({\rm
          Re}\,\la+i\xi,\eta\big)  ,\hat w_{n,\iota}\big({\rm
          Re}\,\la+i\xi,\eta\big)  \big\rangle _{\mL^2(\hT_n)}
        \Big]\dd\xi\bigg| \\
        &
   \quad     + \sum_{\iota=\pm}\bigg|\int_{\R}\bbE_n\Big[\langle{\frak R}_{n,u,\iota}\big({\rm
          Re}\,\la+i\xi,\eta\big) , \hat u_{n,\iota}\big({\rm
          Re}\,\la+i\xi,\eta\big) \rangle_{\mL^2(\hT_n)} \Big]\dd\xi\bigg|.
   \end{split}
 \end{equation}
   Thanks to Corollary \ref{cor012502-26} all the improper integrals
appearing on the left hand side, as well as the last two  expessions  on the right hand side of above expression
above are finite. The improper integrals appearing in the first two expressions on the right hand side can be
understood as the limits of integrals $\int_{-A}^A$, as
$A\to+\infty$.

 In what follows we prove that the last two terms in the left hand side of \eqref{051003-26} identically equal to zero, namely we claim the following. 
\begin{lemma}
  \label{cor011003-26}
 For any $c>0$, $\eta\in\bbZ$, any ${\rm x}\in\{u,w\}$,  we have
 \begin{equation}
   \label{041003-26}
 \int_{\R}    \bbE_n \Big[ \big\langle{\frak R}_{n,{\rm x},\pm}(c+i\xi,\eta) ,\hat
{\rm x}_{n,\pm}(c+i\xi,\eta) \big\rangle _{\mL^2(\hT_n)}\Big]
\dd\xi =0,
\end{equation}
\end{lemma}

  Before proving Lemma \ref{cor011003-26}, let us show how it is used in the estimate
of \eqref{081003-26}. Thanks to Corollary \ref{cor012502-26} we infer that  
\[
\xi\mapsto \bbE_n\Big[\sum_{\iota=\pm} \langle\tilde
     W_{n,\iota}(0,\eta ),  \hat w_{n,\iota}\big({\rm
          Re}\,\la+i\xi,\eta\big) \rangle
     _{\mL^2(\hT_n)} \Big]
\]
     belongs to $\mL^2(\R)$. The inversion formula
     yields that
  \[ t\mapsto \frac{1}{2\pi}\int_{\R}e^{i\xi t}\,\bbE_n\Big[\sum_{\iota=\pm} \langle\tilde
     W_{n,\iota}(0,\eta ),  \hat w_{n,\iota}\big({\rm
          Re}\,\la+i\xi,\eta\big) \rangle
     _{\mL^2(\hT_n)}  \Big]\dd\xi
     \]
     belongs to $\mL^2([0,+\infty))$ and equals a.e. 
 \[
t\mapsto  e^{-{\rm
          Re}\,\la t}\,\bbE_n\Big[\sum_{\iota=\pm} \langle\tilde
     W_{n,\iota}(0,\eta ), W_{n,\iota}(t,\eta) \rangle
     _{\mL^2(\hT_n)}  \Big] .
    \] 
  Since it also belongs to $C([0,+\infty))$ we conclude that in particular
     \begin{multline*}\int_{\R}\bbE_n\Big[\sum_{\iota=\pm} \langle\tilde
     W_{n,\iota}(0,\eta ),  \hat w_{n,\iota}\big({\rm
          Re}\,\la+i\xi,\eta\big) \rangle
                      _{\mL^2(\hT_n)} \Big]\dd\xi \\
         =\bbE_n\Big[\sum_{\iota=\pm} \langle\tilde
     W_{n,\iota}(0,\eta ),  W_{n,\iota}(0,\eta) \rangle
     _{\mL^2(\hT_n)} \Big]=\sum_{\iota=\pm} \big\lnm\tilde
     W_{n,\iota}(0)\big\rnm^2_{\eta} .
     \end{multline*}
        Similarly 
  \begin{align*}
&\int_{\R}\bbE_n\Big[\sum_{\iota=\pm} \langle\tilde
     U_{n,\iota}(0,\eta ),  \hat u_{n,\iota}\big({\rm
          Re}\,\la+i\xi,\eta\big) \rangle
                      _{\mL^2(\hT_n)} \Big]\dd\xi 
          = \sum_{\iota=\pm} \big\lnm \tilde
     U_{n,\iota}(0 )\big\rnm_\eta^2.
  \end{align*}
  As a result, from \eqref{051003-26}, \eqref{041003-26}, and recalling the initial bound (Proposition \ref{prop030204-26}), we obtain the estimate: for any ${\rm Re}\,\la>0$ there
  exists $C>0$ such that
  \begin{multline}
  \label{051003-26b}
    \int_{\R}\bbE_n\bigg[{\cal
    	D}_n\Big(\big(   \hat w_{n,0} -  \hat u_{n,+} \big) \big({\rm
    	Re}\,\la+i\xi,\eta\big)\Big) +  \big\|\hat u_{n,-}\big({\rm
    	Re}\,\la+i\xi,\eta\big)  \big\|^2_{\mL^2(\hT_n)}  
      \\
      + {\cal
    	D}_n\big(  \delta\hat w_{n}      \big({\rm
    	Re}\,\la+i\xi,\eta\big)\big)   \bigg]\dd\xi
            \le   \frac{C}{n^2},\qquad\mbox{for all  }\xi\in\bbR.
   \end{multline}
 Using \eqref{071003-26} we conclude the proof of \eqref{081003-26}. \qed

 This ends the proof of   \eqref{101003-26a}
finishing the proof of \eqref{101003-26}, concluding in
this way the  demonstration of Lemma \ref{lem:limIII}. The only thing yet
to be shown is  Lemma \ref{cor011003-26}.

\subsubsection{Proof of Lemma \ref{cor011003-26}}

 By the definition
   \begin{multline}
     \label{090704-26}
\int_{\R}    \bbE_n  \Big[\big\langle{\frak R}_{n,w,+}(c+i\xi,\eta) ,\hat
w_{n,+}(c+i\xi,\eta) \big\rangle _{\mL^2(\hT_n)}\Big]
\dd\xi \\
= \underset{k,\ell'\in\hat\T_n}{\hat{\sum}} \int_{\R} \dd\xi\; \bbE_n \Bigg[\bigg(\int_0^{+\infty}e^{-(c+i\xi) t}   {\frak r}_{n,w,+}
  (t-,\eta,k,\ell')   \dd\thN (\gamma n^2 t,\ell')
    \bigg) \\
    \times\bigg(\int_0^{+\infty}e^{-(c+i\xi) t}
  \dd t \int_0^t\dd W _{n,+}(s,\eta,k)\bigg)^\star\Bigg].
\end{multline}
Using \eqref{eq:compact} we can rewrite the right hand side of
\eqref{090704-26} in the form $D_{n,w}+S_{n,w}$, where
\begin{equation}
  \label{100704-26} 
    \begin{split}
 D_{n,w}&:= \underset{k,\ell'\in\hat\T_n}{\hat{\sum}}   \int_{\R} \dd\xi \; \bbE_n \Bigg[\bigg(\int_0^{+\infty}e^{-(c+i\xi) t}   {\frak r}_{n,w,+}
  (t-,\eta,k,\ell')   \dd\thN (\gamma n^2 t,\ell')
    \bigg) \\
    &
    \qquad \qquad \qquad \qquad  \times\bigg(\int_0^{+\infty}e^{-(c+i\xi) t}
    \dd t \int_0^t {\frak D}_{n,w,+}(s,\eta,k)\dd s\bigg)^\star \Bigg]\\
 S_{n,w}&:= \underset{k,\ell',\ell''\in\hat\T_n}{\hat{\sum}}   \int_{\R}\dd\xi \; \bbE_n \Bigg[ \bigg(\int_0^{+\infty}e^{-(c+i\xi) t}   {\frak r}_{n,w,+}
  (t-,\eta,k,\ell')   \dd\thN (\gamma n^2 t,\ell')
    \bigg) \\
    &
   \qquad \qquad \qquad \qquad  \times\bigg(\int_0^{+\infty}e^{-(c+i\xi) t}
    \dd t \int_0^t {\frak r}_{n,w,+}(s-,\eta,k,\ell'')  \dd\thN (\gamma n^2 s,\ell'')\bigg)^\star\Bigg].
\end{split}
\end{equation}
We are now going to prove that each term $D_{n,w}$ and $S_{n,w}$ identically equal to $0$.

\subsubsection*{Analysis of $D_{n,w}$}
 
Thanks to Theorem \ref{thm011002-26}, if  ${\rm Re}\,\la>0$, then
\[\lim_{t\to+\infty} e^{-\lambda t} \int_0^t
\bbE_n \Big[\big|   \mathfrak{D}_{n,\mathrm{x},\iota}(s,\eta,k) \big|^2\Big]\dd s = 0\] 
for each fixed $n$,     $\mathrm{x}\in \{u,w\}$ and $\iota \in
\{-,+\}$.  We rewrite
\begin{multline*}
 D_{n,w} 
 =\underset{k,\ell'\in\hat\T_n}{\hat{\sum}}   \lim_{\eps\to0+} \int_{\R} \dd\xi\; \bbE_n \Bigg[\bigg(\int_0^{+\infty}e^{-(c+i\xi) t}   {\frak r}_{n,w,+}
  (t-,\eta,k,\ell')   \dd\thN (\gamma n^2 t,\ell')
                 \Bigg)
   \\
    \times\bigg(\int_0^{+\infty}\frac{e^{-(c+i\xi) s}} {c+i\xi+\eps^2\xi^2 } {\frak D}_{n,w,+}(s,\eta,k)
      \dd s\bigg)^\star\Bigg].
  \end{multline*}
 Recalling that $\dd\thN(\gamma n^2t,\ell)=\dd \hat N(\gamma n^2 t,\ell)-\gamma n^3\delta_{\ell,0}\; \dd t$ we decompose accordingly \[D_{n,w}=\lim_{\eps \to 0+}\big( D_{n,w,\eps}^{({\rm jump})}-D_{n,w,\eps}^{({\rm comp})}\big),\]
  where
 \begin{multline*}
 D_{n,w,\eps}^{({\rm jump})} =\underset{k,\ell'\in\hat\T_n}{\hat{\sum}}    \int_{\R} \dd\xi\; \bbE_n \Bigg[\bigg(\int_0^{+\infty}e^{-(c+i\xi) t}   {\frak r}_{n,w,+}
 (t-,\eta,k,\ell')   \dd\hat N (\gamma n^2 t,\ell')
 \Bigg)
   \\
    \times\bigg(\int_0^{+\infty}\frac{e^{-(c+i\xi) s}} {c+i\xi+\eps^2\xi^2 } {\frak D}_{n,w,+}(s,\eta,k)
   \dd s\bigg)^\star\Bigg]
 \end{multline*}
and
\begin{multline*}
 D_{n,w,\eps}^{({\rm comp})}
  =n^{2}\ga\hsum  \int_{\R} \dd\xi\; \bbE_n \Bigg[\bigg(\int_0^{+\infty}e^{-(c+i\xi) t}   {\frak r}_{n,w,+}
  (t-,\eta,k,0)   \dd  t
                 \bigg)
   \\
    \times\bigg(\int_0^{+\infty}\frac{e^{-(c+i\xi) s}} {c+i\xi+\eps^2\xi^2 } {\frak D}_{n,w,+}(s,\eta,k)
      \dd s\bigg)^\star\Bigg].
\end{multline*}
Fix $\ell'$ and suppose that $0<\tau_1<\ldots$ are the jump times of $\big(\hat{N}
(\gamma n^2 t,\ell')\big)_{t\ge0}$.

Then, using Fubini's Theorem,
\begin{align*}
 D_{n,w,\eps}^{({\rm jump})}&=\underset{k,\ell'\in\hat\T_n}{\hat{\sum}}  \sum_{j\ge 1}\int_{\R} \dd\xi\; \bbE_n \bigg[ \Big(e^{-(c+i\xi) \tau_j}   {\frak r}_{n,w,+}
  (\tau_j-,\eta,k,\ell')   
                 \Big)
   \\
 & \qquad \qquad \qquad \qquad \qquad \times\bigg(\int_0^{+\infty}\frac{e^{-(c+i\xi) s}} {c+i\xi+\eps^2\xi^2 } {\frak D}_{n,w,+}(s,\eta,k)
      \dd s\bigg)^\star\bigg]
\\
&
  =\underset{k,\ell'\in\hat\T_n}{\hat{\sum}}   \sum_{j\ge 1}
     \bbE_n\bigg[ \int_0^{+\infty}\Big(\int_{\R}  \frac{e^{i\xi (s-\tau_j)}
     \dd\xi } {c-i\xi+\eps^2\xi^2 } \Big)  {\frak r}_{n,w,+}
  (\tau_j-,\eta,k,\ell')             
    e^{-c (s  +\tau_j)}{\frak D}_{n,w,+}^\star   (s,\eta,k)
     \dd s\bigg].
 \end{align*}
Let
$
f(x)= 1/(c+ix+\eps^2x^2).
$
We use the formula
\begin{align*}
&\widehat f_\eps(y)=\int_{-\infty}^{\infty} \frac{e^{-i y x}}{c+ix+\epsilon^2x^2}\,\dd x
            =\frac{2\pi}{\sqrt{1+4\epsilon^2 c}}
\begin{cases}
\exp\!\left(c_\eps \,y\right), & \quad y<0,\\[1ex]
\exp\!\left(-d_\eps\,y\right), &\quad  y>0,\\[1ex]
1,& \quad y=0,
\end{cases}
\end{align*}
where
\begin{align*}
  c_\eps = \dfrac{2c}{\sqrt{1+4\epsilon^2 c}+1} 
,\qquad 
d_\eps =\dfrac{\sqrt{1+4\epsilon^2 c}+1}{2\epsilon^2}.
\end{align*}
As a result, plugging this  formula into the expression for $D_{n,w,\eps}^{({\rm jump})}$ and then taking the limit as $\eps \to 0$, we get
\begin{align*}
 \lim_{\eps \to 0+}D_{n,w,\eps}^{({\rm
  jump})}&=\underset{k,\ell'\in\hat\T_n}{\hat{\sum}} \sum_{j\ge 1}
           2\pi \bbE_n \Bigg[   {\frak r}_{n,w,+}
  (\tau_j-,\eta,k,\ell')   
                      \bigg(\int^{ \tau_j} _0 e^{-c(s+\tau_j)
      }e^{-c(\tau_j-s)}  {\frak D}_{n,w,+}(s,\eta,k)
      \dd s\bigg)^\star\Bigg]
\\
&
  =\underset{k,\ell'\in\hat\T_n}{\hat{\sum}}   \sum_{j\ge 1} 2\pi \bbE_n \Bigg[e^{-2c\tau_j
      }   {\frak r}_{n,w,+}
  (\tau_j-,\eta,k,\ell')   
                     \bigg(\int_{0}^{\tau_j} {\frak D}_{n,w,+}(s,\eta,k)
     \dd s\bigg)^\star\Bigg]\\
  &
    =\underset{k,\ell'\in\hat\T_n}{\hat{\sum}} 2\pi \bbE_n \Bigg[\int_0^{+\infty} \dd \hat N(\ga n^2t,\ell')\;  e^{-2ct} {\frak r}_{n,w,+}
  (t-,\eta,k,\ell')   
                   \bigg(\int_{0}^{t} 
         {\frak D}_{n,w,+}(s,\eta,k)
     \dd s\bigg)^\star \Bigg]
\end{align*}
Similarly,
\begin{align*}
 \lim_{\eps \to 0+} D_{n,w,\eps}^{({\rm comp})}=\ga n^2\hsum 2\pi \bbE_n \Bigg[\int_0^{+\infty}\dd  t \; e^{-2ct 
      } {\frak r}_{n,w,+}
  (t-,\eta,k,0)   
                   \bigg(\int_{0}^{t}   {\frak D}_{n,w,+}(s,\eta,k)
     \dd s\bigg)^\star \Bigg].
 \end{align*}
 Hence, the summation gives
 \begin{align*}
 D_{n,w}=\underset{k,\ell'\in\hat\T_n}{\hat{\sum}} 2\pi \bbE_n \Bigg[\int_0^{+\infty} \dd \thN(\ga n^2t,\ell')\;  e^{-2ct} {\frak r}_{n,w,+}
 (t-,\eta,k,\ell')   
 \bigg(\int_{0}^{t} 
 {\frak D}_{n,w,+}(s,\eta,k)
 \dd s\bigg)^\star \Bigg].
 \end{align*}
 and a standard martingale argument using the dominated convergence Theorem allows us to prove that
 \begin{align}
\label{011003-26a}
 D_{n,w} =0.
                 \end{align}


\subsubsection{Analysis of $S_{n,w}$}

Recall the definition of $S_{n,w}$ in \eqref{100704-26}. We can write 
\begin{align*}   
 S_{n,w}       = S_{n,w}^{+,+}+S_{n,w}^{-,-}-S_{n,w}^{-,+}-S_{n,w}^{+,-},
\end{align*}
where the respective terms  $S_{n,w}^{\iota,\iota'}$, $\iota,\iota'\in\{-,+\}$ arise
from the expansion
$
\dd\thN (\gamma n^2 t,\ell)=\dd\hat{N} (\gamma
n^2 t,\ell)-\gamma n^3 t\delta_{0,\ell}.
$
The argument is very similar to the one we have used in the case of $D_{n,w}$. Let us start with $S_{n,w}^{+,+}$. 

We have
\begin{align*}  &
 S_{n,w}^{+,+}=\underset{k,\ell',\ell''\in\hat\T_n}{\hat{\sum}}  \lim_{\eps\to0+} \int_{\R}\dd\xi \bbE_n \Bigg[\bigg(\int_0^{+\infty}e^{-(c+i\xi) t}   {\frak r}_{n,w,+}
  (t-,\eta,k,\ell')   \dd\hat{N}(\gamma n^2 t,\ell')
    \bigg)\\
    &\qquad \qquad \qquad \qquad
    \times\bigg(\int_0^{+\infty} \frac{e^{-(c+i\xi) s}}{c+i\xi+\eps^2\xi^2}{\frak r}_{n,w,+}(s-,\eta,k,\ell'')  \dd\hat{N}(\gamma n^2 s,\ell'')
      \bigg)^\star\Bigg] 
\\ &
  =\underset{k,\ell',\ell''\in\hat\T_n}{\hat{\sum}}  \sum_{j,j' \ge 1}\lim_{\eps\to0+}\bbE_n
                  \Bigg[  \int_{\R}\frac{e^{i\xi( \tau_{j'}-\tau_j)} \dd\xi }{c-i\xi+\eps^2\xi^2} \\ & \qquad \qquad \qquad \qquad \times e^{-c(\tau_{j}+\tau_{j'})}{\frak r}_{n,w,+}
  (\tau_j-,\eta,k,\ell')    
    \Big\{  {\frak r}_{n,w,+}(\tau_{j'}-,\eta,k,\ell'')  \Big\}^\star \Bigg]\\
                &
                  =\underset{k,\ell',\ell''\in\hat\T_n}{\hat{\sum}} 2\pi\sum_{0<j\le j'} \bbE_n
                  \Bigg[ e^{-2c\tau_{j}} {\frak r}_{n,w,+}
  (\tau_j-,\eta,k,\ell')    
    \Big\{  {\frak
                  r}^{w,+}_{n}(\tau_{j'}-,\eta,k,\ell'')  \Big\}^\star
                  \Bigg] 
\\ &
  =2\pi \underset{k,\ell',\ell''\in\hat\T_n}{\hat{\sum}} \bbE_n \Bigg[\bigg(\int_0^{+\infty}e^{-2c t}  {\frak r}_{n,w,+}
  (t-,\eta,k,\ell')   \dd\hat{N}(\gamma n^2 t,\ell')
    \bigg) \\
    &
  \qquad \qquad \qquad \qquad   \times\bigg(\int_0^{t-}{\frak r}_{n,w,+}(s-,\eta,k,\ell'')  \dd\hat{N}(\gamma n^2 s,\ell'')
      \bigg)^\star\Bigg] .
\end{align*}
Dealing similarly with the remaining cases of $S_{n,w}^{\iota,\iota'}$,
$\iota,\iota'\in\{-,+\}$  we get finally
\begin{multline*}  
 S_{n,w}=2\pi\underset{k,\ell',\ell''\in\hat\T_n}{\hat{\sum}} \bbE_n \Bigg[\bigg(\int_0^{+\infty}e^{-2c t}   {\frak r}_{n,w,+}
  (t-,\eta,k,\ell')   \dd\thN(\gamma n^2 t,\ell')
    \bigg) \\
    \times\bigg(\int_0^{t-}{\frak
      r}_{n,w,+}(s-,\eta,k,\ell'')  \dd\thN(\gamma n^2 s,\ell'')
      \bigg)^\star \Bigg].
\end{multline*}
This in turn implies that 
 \begin{align}
\label{011003-26b}
 S_{n,w} =0.
                 \end{align}
               We have shown therefore that
               \begin{equation*}
  \int_{\R}    \bbE_n  \Big[\big\langle{\frak R}_{n,w,+}(c+i\xi,\eta) ,\hat
  w_{n,+}(c+i\xi,\eta) \big\rangle _{\mL^2(\hT_n)}\Big]
  \dd\xi=0.
\end{equation*}
Similar calculations can be done for ${\frak R}_{n,w,-}$ and ${\frak R}_{n,u,\pm}$.
This ends the proof of  Lemma \ref{cor011003-26}.\qed

\appendix

\section{Properties and examples of initial data}

\label{AppA}

 \subsection{Local Gibbs Measures} \label{app:lln}

~
 
We prove here that the initial condition $\mu_{n,T(\cdot)}$ defined in \eqref{eq:localG} {satisfies Assumption \ref{T}}. First, we prove the following version of the law of large numbers (recall \eqref{eq:1}).
 \begin{theorem}
	   \label{thm012303-26}
For any function $\varphi\in C(\bbT)$ and $T\in C(\T)$ which satisfies \eqref{Ts},
	   \begin{equation}
		     \label{012702-26}
		     \lim_{n\to+\infty}\int_{\Om_n}\Big|\frac{1}{n}\sum_{x\in\bT_n}e_x
		     \varphi\Big(\frac{x}{n}\Big)-\int_{\bbT}\varphi(u)T(u)\dd u\Big|\dd \mu_{n,T}=0.
		     \end{equation}
	   \end{theorem}
%
%

 \proof
 To prove Theorem \ref{thm012303-26} it suffices to show that, for any
 function $\varphi\in C(\bbT)$ we have
    \begin{equation}
      \label{012702-26bis}
      \lim_{n\to+\infty}\int_{\Om_n}\Big[\frac{1}{n}\sum_{x\in\bT_n}
      \left(e_x- T\left(\frac xn\right)\right)\varphi\Big(\frac{x}{n}\Big)\Big]^2\dd \mu_{n,T}=0.
      \end{equation}

    Let
    \begin{equation}
      \label{022503-26} 
e_x^{\rm kin}:=\frac12  p_x^2,\qquad e_x^{\rm pot}:=\frac12 \big(
        \om_0^2q_x^2+(\nabla^\star  q_x)^2\big), \end{equation}
        and note that \begin{equation*}
      \int_{\Om_n} e_x^{\rm kin}\dd \mu_{n,T}
      {= \frac 12 T\left(\frac xn\right)}.
\end{equation*}
{Let us denote { $\beta_x:=T^{-1}\left(\frac xn\right)$}.
  The random field $(p_x,q_x)_{x\in\bT_n}$ is   Gaussian with covariances \begin{equation}
		\label{cov}
		\begin{split}
                  &\int_{\Omega_n}p_xp_{x'}\dd\mu_n={\beta_x^{-1}}
                  \delta_x(x'),\qquad
			\int_{\Omega_n}q_xq_{x'}\dd\mu_n=G^{\beta}_{x,x'},\\
			&
			\int_{\Omega_n}p_xq_{x'}\dd\mu_n=0.
		\end{split}
	\end{equation}
	where  $\delta_x(x'):=\mathbf{1}_{[x=x']}$ and $G^{\beta}$
	is the inverse of the matrix  $\big[\big({L}_\beta \delta_x\big)_{x'}\big]_{x,x'\in\T_n}$, where
	\begin{equation}
		\label{L-beta}
		({L}_\beta f)_x:=\om_0^2 \; \beta_x f_x-\nabla\Big[\beta_\cdot \nabla^\star  f\Big]_x.
\end{equation}}
 Since $\big(e_x^{\rm kin}- {\frac12} T\left(\frac xn\right)\big)_{x\in\bT_n}$ are
    independent, centered  random variables with second moment equal to
      {$\frac 12 T\left(\frac xn\right)^2$},
    we can easily show that
    \begin{multline}
      \label{012702-26k}
        \int_{\Om_n}\Big[\frac{1}{n}\sum_{x\in\bT_n}
        \left(e_x^{\rm kin}- {\frac 12} T\left(\frac xn\right)\right)
        \varphi\Big(\frac{x}{n}\Big)\Big]^2\dd \mu_{n,T}\\
        = \frac{1}{n^2}\int_{\Om_n}\sum_{x\in\bT_n}
        \left(e_x^{\rm kin}- {\frac 12} T\left(\frac xn\right)\right)^2
        \varphi^2\Big(\frac{x}{n}\Big) \dd
      \mu_{n,T}\le
      \frac{2}{n}\Big(\frac{\max|\varphi| }{\beta_*}\Big)^2\mathop{\longrightarrow}_{n\to\infty} 0.
      \end{multline}
    The remaining part of the argument is devoted to the proof that
    \begin{equation}
      \label{012702-26p}
      \lim_{n\to+\infty}\int_{\Om_n}\Big[\frac{1}{n}\sum_{x\in\bT_n}
      \left(e_x^{\rm pot}-  {\frac 12} T\left(\frac xn\right)\right)\varphi\Big(\frac{x}{n}\Big)\Big]^2
      \dd \mu_{n,T}=0.
    \end{equation}
    Define the operator
    \begin{align*}
   {\mathfrak M}  f &=\sum_{x\in\bbT_n}\beta_x^{-1}\exp\Big\{\sum_{x'\in\bbT_n}\beta_{x'}e_{x'}
                     \Big\}\partial_{q_x}\Big(\exp\Big\{-\sum_{x'\in\bbT_n}\beta_{x'}e_{x'}
      \Big\}\partial_{q_x}f\Big)\\
      &
        =\sum_{x\in\bbT_n}\Big[\beta_x^{-1}\Big(\partial^2_{q_x}f-\nabla\big(\beta\nabla^\star
        q\big)_x\partial_{q_x}f\Big) - \om_0^2q_x\partial_{q_x}f
        \Big],\quad f\in C^2(\Om_n).
    \end{align*}
    It is easy to see that this is symmetric with respect to the
       scalar product in $\mL^2(\mu_{n,T})$.

      Using this   we conclude that
    \begin{equation}
      \label{Mq}
      \int_{\Om_n}{\frak M} \big(\tfrac12 q_x^2\big)   d\mu_{n,T}=0.
    \end{equation}
    Denote
    \begin{align*}
    {\cal H}^{\rm pot}_\beta&:=\sum_{x\in\bbT_n}\beta_xe_x^{\rm pot}\qquad
    \mbox{and}\qquad {\cal H}_\beta :=\sum_{x\in\bbT_n}\beta_xe_x\\
  \text{and also}\quad  \xi_x &:= \beta_x^{-1}q_x\partial_{q_x}{\cal H}^{\rm pot}_{\beta}=\omega_0^2 q_x^2  - q_x\beta_x^{-1} \nabla\left( \beta
    \nabla^\star q\right)_x.
    \end{align*}
   From \eqref{Mq} we get
    \begin{equation}
      \label{eq:6}
      \beta_x^{-1}   =\int_{\Om_n} \xi_x \dd
      \mu_{n,T}.
    \end{equation}
    Note that for $|x-y|>1$, integrating by parts, we obtain
    \begin{align*}
     \int_{\Om_n} \xi_x \xi_y \dd
     \mu_{n,T}
       & =-\frac{1}{Z_{n,T(\cdot)}}\int_{\Om_n}
        \xi_x \beta_y^{-1} q_y\partial_{q_y}\Big(\exp\big\{-{\cal H}_\beta\big\}\Big)
       \dd{\bf q}\dd{\bf p}\\
      &
        =\beta_y^{-1}\int_{\Om_n}
        \xi_x \dd
     \mu_{n,T}= \beta_y^{-1}\beta_x^{-1}.
     \end{align*}
  Thus   $(\xi_x,\xi_y)$ are uncorrelated, if $|x-y|>1$.
  Consequently it is easy to see that
  \begin{equation}
    \label{eq:10}
    \begin{split}
      &\lim_{n\to+\infty}\int_{\Om_n}\Big[\frac{1}{n}\sum_{x\in\bT_n}(\xi_x - \beta_x^{-1})
      \varphi\Big(\frac{x}{n}\Big)\Big]^2\dd \mu_{n,T} =0.
    \end{split}
  \end{equation}
Then, summing by parts and using the formula
$\nabla^\star(fg)_x=(\nabla^\star f)_xg_x+(\nabla^\star g)_xf_{x-1}$,
 we get
    \begin{equation}
      \label{eq:2}
      \begin{split}
       - \frac 1n \sum_{x\in\bT_n} q_x\beta_x^{-1} \nabla\left( \beta \nabla^\star q\right)_x
        \varphi\Big(\frac{x}{n}\Big)
       & = \frac 1n \sum_{x\in\bT_n} \nabla^\star (q\beta )_x
        \nabla^\star\left(\beta^{-1} q \varphi\Big(\frac{\cdot}{n}\Big)\right)_x
        \\
        &= \frac 1n \sum_{x\in\bT_n} (\nabla^\star q_x)^2
        \varphi\Big(\frac{x}{n}\Big) + \mathcal R_n,
      \end{split}
    \end{equation}
    where $\lim_{n\to+\infty}\int_{\Om_n}|\mathcal R_n|^2\dd\mu_{n,T} =0$, 
    thanks to the fact that $T(\cdot),\varphi\in C(\bT)$.
    This, in turn implies
    \begin{equation}
      \label{eq:11}
      \lim_{n\to+\infty}\int_{\Om_n}\Big[\frac{1}{n}\sum_{x\in\bT_n}
      (e_x^{\rm pot} -{\frac 12} \xi_x)\varphi\Big(\frac{x}{n}\Big)\dd \mu_{n,T}\Big]^2
      =0,
    \end{equation}
    that concludes the proof of the theorem.
\qed

\medskip



The next result deals with the estimate of the fourth moment \eqref{eq:3z}
  of the
Fourier transform of the wave function in the case of $\mu_{n,T}$.
From \eqref{cov} and \eqref{L-beta}, note that in particular
 \begin{align*}
   \sum_{x\in\bT_n}\int_{\Omega_n}e_x\dd\mu_{n,T}=\frac{1}{2}\sum_{x\in\bT_n}{\beta_x^{-1}}
   +\frac12{\rm Tr}\big(G^{\beta}(\om_0^2-\Delta)\big).
   \end{align*}
   {Recall assumption \eqref{Ts} on the function $T(\cdot)$ and let
     $$
     \beta_*=1/T^*
     \quad\mbox{and}\quad
     \beta^*=1/T_* .
   $$}
Then, for any $f:\T_n\to\bbR$
\begin{equation}
  \label{Gfi}
 \beta^*\langle (\om_0^2-\Delta)f,f\rangle_{\mL^2(\T_n)}\ge \langle  L_\beta
 f,f\rangle_{\mL^2(\T_n)}\ge  \beta_*\langle (\om_0^2-\Delta)f,f\rangle_{\mL^2(\T_n)}.
\end{equation}
Assumption \eqref{eq:3z} follows in particular from the following estimates.
 \begin{proposition} 
	\label{prop060204-26}
	We have
	\begin{equation}
		\label{080204-26}
		\sup_{k\in\hT_n}\int_{\Om_n}|\hat\psi(k)|^2\dd\mu_{n,T}\le
		n\Big(\frac1{T_*}+\frac{T^*(\om_0^2+4)}{\om_0^2}\Big),
	\end{equation}
	and, in consequence, 
	\begin{equation}
		\label{080204-26a}
		\sup_{k\in\hT_n}\int_{\Om_n}|\hat\psi(k)|^4\dd\mu_{n,T}\le
		3n^2\Big(\frac1{T_*}+\frac{T^*(\om_0^2+4)}{\om_0^2}\Big)^2.
	\end{equation}
\end{proposition}
\proof
Estimate \eqref{080204-26a} follows from \eqref{080204-26}.  Indeed, gaussianity of
$\hat\psi(k)$ implies that
\[
\int_{\Om_n}|\hat\psi(k)|^4\dd\mu_{n,T}\le 3
\Big(\int_{\Om_n}|\hat\psi(k)|^2\dd\mu_{n,T}\Big)^2.
\]
Therefore it suffices only to
prove \eqref{080204-26}. Note that
\begin{align*}
	&
	\int_{\Om_n}|\hat\psi(k)|^2\dd\mu_{n,T}=\sum_{x,x'\in\bT_n}e^{2\pi
		ik(x'-x)}\int_{\Om_n}\psi_x\psi_{x'}^\star \dd\mu_{n,T}\\
	&
	=\sum_{x,x'\in\bT_n}e^{2\pi
		ik(x'-x)}\int_{\Om_n}p_xp_{x'} \dd\mu_{n,T}+\sum_{x,x'\in\bT_n}e^{2\pi
		ik(x'-x)}\int_{\Om_n}(\tilde \om * q)_x (\tilde \om * q)_{x'} \dd\mu_{n,T}.
\end{align*}
Using the relations \eqref{cov} we obtain that the utmost right hand
side equals
\begin{multline*}
	\sum_{x\in\bT_n}\beta\Big(\frac{x}{n}\Big)+\sum_{x,x',z,z'\in\bT_n}
	e^{2\pi ik(x'-x)}\tilde \om_{x-z}\tilde \om_{x'-z'}G^\beta_{z,z'}   \\
	\le n \frac1{T_*} +\sum_{x,x',z,z'\in\bT_n}
	e^{-2\pi ik(x-z)}\tilde \om_{x-z}e^{2\pi
		ik(x'-z')}\tilde \om_{x'-z'}e^{-2\pi i kz}G^\beta_{z,z'}  e^{2\pi i kz'}
\end{multline*}
The second term on the right hand side equals 
\begin{align*}
	\om^2(k)\sum_{z,z'\in\bT_n}
	e^{i kz}G^\beta_{z,z'}  e^{-i kz'} =\om^2(k)\langle G^\beta {\frak
		e}(k),{\frak e}(k)\rangle,
\end{align*}
where ${\frak e}(k):=e^{i kz}$. Thanks to \eqref{Gfi} the right
hand side is estimated by
$nT^*(\om_0^2+4)/\om_0^2
$ and \eqref{080204-26} follows.
\qed
    
\subsection{Some properties of  Wigner functions}

$~$

Here we give some properties of the Wigner functions, in
  particular we explain the sense in which $W_{n,+}$ is ``close'' to the energy field. {In particular we give the details of the proof of \eqref{eq:wigner_energy}.}
	
Recall {\eqref{eq:omega}, \eqref{011307f}, \eqref{011307} for the definitions of} $\om(k)$, $\hat\psi(k)$, $\psi_x$. Recall also that the Wigner and  Fourier-Wigner functions are defined by 
\begin{align*}
   {\cal W}_{n,+}(x,k)&:=\sum_{\eta\in\bT_n}e^{2\pi i\eta
    x/n}\; W_{n,+}(\eta,k) , \qquad (x,k)\in\bT_n\times\hT_n,\\
  W_{n,+}(\eta,k)& :=\frac{1}{2n}  
 \widehat\psi\Big( k+\frac{\eta}{n}\Big)
 \widehat\psi^\star( k) ,\qquad  (\eta,k)\in\bbZ\times\hT_n.
\end{align*}
Note that
\begin{align*}
   {\cal W}_{n,+}(x,k)=\frac{1}{2n}\sum_{\eta=0}^{n-1}e^{2\pi i\eta
    x/n} \widehat\psi\Big( k+\frac{\eta}{n}\Big)
 \widehat\psi^\star( k) 
     = \frac{1}{2}\sum_{ y' \in\bT_n} \psi_x \psi_{y'}^\star e^{2\pi ik(
    y'-x)} .
\end{align*}
{Fix some Borel probability measure $\mu$ on $\Omega_n$ and} denote also, 
\begin{align*}
  &
   \bar {\cal W}_{n,+}(x,k)= \int_{\Om_n}{\cal W}_{n,+}(x,k)\dd\mu, \qquad (x,k)\in\bT_n\times\hT_n.
\end{align*}
Finally, {recalling the Parseval identity, see also \eqref{psi-e1}:}
\begin{align}
  \label{psi-e}
  \hsum |\hat\psi(k)|^2=\sum_{x\in\bT_n}|\psi_x|^2=2 \sum_{x\in\bT_n}e_x,
  \end{align}
  where $e_x=\frac12(p_x^2+\omega_0^2 q_x^2 + (\nabla^\star q_x)^2)$ is the energy given in \eqref{ex}.
The following estimate holds.
\begin{proposition}
  \label{prop012702-26}
  There exists a constant $C>0$ such that,  for any function $\varphi\in C^\infty(\bbT)$,
  \begin{multline}
    \label{022702-26}
   \int_{\Om_n}\Big|\frac{1}{n}\sum_{x\in\bT_n }
   \big(\frac12|\psi_x|^2 -e_x
   \big)\varphi\Big(\frac{x}{n}\Big)\Big|\dd\mu \\ \le
   \frac{C}{n^2}\Big(\|\varphi'\|_\infty+\sum_{\eta\in\bbZ}\big|\eta{\cal F}(\varphi)(\eta)\big|\Big) \sum_{x\in\bT_n }\int_{\Om_n} e_x^{\rm pot} \dd\mu.
 \end{multline}
 Here $e_x^{\rm pot}$ is defined in \eqref{022503-26}.
  \end{proposition}
  \proof
  We have
    $
|\psi_x|^2=\big(\tilde\om\star q_x\big)^2+p_x^2,
$
therefore, in order to show \eqref{022702-26} it suffices to prove
that 
 \begin{equation}
    \label{022702-26a}
   \int_{\Om_n}\Big|\frac{1}{n}\sum_{x\in\bT_n } \big(\frac12 \big(\tilde\om\star q_x\big)^2 -e_x^{\rm pot} \big)\varphi\Big(\frac{x}{n}\Big)\Big|\dd\mu\le \frac{C}{n^2}\|\varphi'\|_\infty \sum_{x\in\bT_n }\int_{\Om_n} e_x^{\rm pot} \dd\mu.
 \end{equation}
 First, we claim that
 \begin{equation}
    \label{032702-26}
    \int_{\Om_n}\Big|\frac{1}{n}\sum_{x\in\bT_n } \big(\frac12 (\alpha\star q)_xq_x -e_x^{\rm pot} \big)\varphi\Big(\frac{x}{n}\Big)\Big|\dd\mu\le \frac{C}{n^2}\|\varphi'\|_\infty \sum_{x\in\bT_n }\int_{\Om_n} e_x^{\rm pot}\dd\mu,
 \end{equation}
 where $\alpha_0=\om_0^2+2$, $\alpha_{\pm1}=-1$ and $\alpha_y=0$ for
 $|y|\ge 2$. Indeed,  we can write
 \begin{multline*} \frac{1}{n}\sum_{x\in\bT_n } \big(\frac12 (\alpha\star q)_xq_x
                  -e_x^{\rm pot} \big)\varphi\Big(\frac{x}{n}\Big)\\
     =\frac{1}{2n}\sum_{x\in\bT_n }
     \nabla^\star (q^2_x-  q_{x+1}q_x)
     \varphi\Big(\frac{x}{n}\Big) =\frac{1}{2n}\sum_{x\in\bT_n }
    (q^2_x-  q_{x+1}q_x)
      \nabla \varphi\Big(\frac{x}{n}\Big).
 \end{multline*}
 As a result the left hand side of \eqref{032702-26} is estimated by 
 \begin{align*}
         \frac{\|\varphi'\|_\infty}{2n^2}\sum_{x\in\bT_n } \int_{\Om_n}\sum_{x\in\bT_n } (q^2_x+ | q_{x+1}q_x|)\dd\mu\le \frac{C}{n^2}\|\varphi'\|_\infty \sum_{x\in\bT_n }\int_{\Om_n} e_x^{\rm pot}\dd\mu
\end{align*}
and \eqref{032702-26} follows.

To prove \eqref{022702-26a} it suffices therefore to show that 
\begin{multline}
    \label{022702-26b}
     \int_{\Om_n}\Big|\frac{1}{2n}\sum_{x\in\bT_n }
     \big(\big(\tilde\om\star q_x\big)^2 -(\alpha\star
     q)_xq_x\big)\varphi\Big(\frac{x}{n}\Big)\Big|\dd\mu \\ \le \frac{C}{n^2}\Big(\sum_{\eta\in\bbZ}\big|\eta{\cal F}(\varphi)(\eta)\big|\Big) \sum_{x\in\bT_n }\int_{\Om_n} e_x^{\rm pot}\dd\mu.
 \end{multline}
 Writing, by Fourier inversion,
\[
 \varphi\Big(\frac{x}{n}\Big)=\sum_{\eta\in\bbZ}{\cal
   F}(\varphi)(\eta)e^{2\pi i\eta x/n},
\]
 we have
 \begin{multline*}
\frac{1}{n}\sum_{x\in\bT_n } \Big(\frac12 \big(\tilde\om\star
                  q_x\big)^2 -(\alpha\star
                  q)_xq_x\Big)\varphi\Big(\frac{x}{n}\Big)\\
     =\sum_{\eta\in\bbZ}{\cal
   F}(\varphi)(\eta)\Big[\frac{1}{n}\sum_{x\in\bT_n } \Big(\frac12 \big(\tilde\om\star q_x\big)^2 -(\alpha\star q)_xq_x\Big) e^{2\pi i\eta x/n}\Big].
 \end{multline*}
 To prove \eqref{022702-26b} it suffices therefore to show that
 \begin{multline}
   \label{042702-26}
    \frac{1}{2n}\int_{\Om_n}\Big|\sum_{x\in\bT_n } \big(
    \big(\tilde\om\star q_x\big)^2 -(\alpha\star q)_xq_x\big)
   e^{2\pi i\eta x/n}\Big|\dd\mu
    \\
    \le\frac{C}{n^2}  |\eta | \sum_{x\in\bT_n }\int_{\Om_n} e_x^{\rm
      pot}\dd\mu,\qquad \eta\in\bbZ.
    \end{multline}
 Note that
 \begin{align*}
\frac{1}{2n}\sum_{x\in\bT_n } (\alpha\star q)_xq_x \; e^{2\pi i\eta x/n}&=\frac{1}{2n} \hsum \om^2\Big(k+\frac{\eta}{n}\Big)\hat q(k)\hat
     q\Big(-k-\frac{\eta}{n}\Big)
 \\
\frac{1}{2n}\sum_{x\in\bT_n } \big(\tilde\om\star q_x\big)^2   e^{2\pi i\eta x/n}
   &  =\frac{1}{2n} \hsum\om (k) \om \Big(k+\frac{\eta}{n}\Big)\hat q(k)\hat
     q\Big(-k-\frac{\eta}{n}\Big).
 \end{align*}
 As a result
 \begin{equation}
    \label{052702-26}
    \begin{split} 
    & \frac{1}{2n}\int_{\Om_n}\Big|\sum_{x\in\bT_n } \big(
      \big(\tilde\om\star q_x\big)^2 -(\alpha\star q)_xq_x\big)
      e^{2\pi i\eta x/n}\Big|\dd\mu\\
     & \le \frac{1}{2n} \hsum \int_{\Om_n}\Big|\om (k) -\om \Big(k+\frac{\eta}{n}\Big) \Big|\om \Big(k+\frac{\eta}{n}\Big) |\hat q(k)|\Big|\hat
     q\Big(-k-\frac{\eta}{n}\Big)\Big|\dd\mu \\ & \le\frac{C|\eta|}{n^2}\sum_{x\in\bT_n }\int_{\Om_n} e_x^{\rm pot}\dd\mu
    \end{split}
  \end{equation}
  and
  \eqref{042702-26} follows.
  \qed


Suppose that $\varphi\in C (\bbT)$. Then, the Wigner field is defined as
\begin{align*}
 {\cal W}_{n,+}[\varphi]:=\frac{1}{n}\sum_{x\in\bT_n }\hsum
   \varphi^\star\Big(\frac{x}{n}\Big){\cal W}_{n,+}(x,k)=
    \frac{1}{2n}\sum_{x\in\bT_n } |\psi_x|^2  \varphi^\star\Big(\frac{x}{n}\Big).
\end{align*}
  From Proposition \ref{prop012702-26} 
  we conclude the following.
\begin{corollary} 
  \label{cor032603-26}
  There exists a constant $C>0$ such that, for any real valued
  function $\varphi\in C^\infty(\bbT)$ and probability measure $\mu$
  on $\Om_n$, we have, for any $n\ge 1$,
  \begin{equation*}
  \int_{\Om_n}\Big|{\cal
     W}_{n,+}[\varphi]-\frac{1}{n}\sum_{x\in\bT_n }
    e_x
    \varphi\Big(\frac{x}{n}\Big)\Big|\dd\mu \le \Big(\frac{1}{n} \sum_{x\in\bT_n
  }\int_{\Om_n} e_x  \dd\mu\Big)
  \frac{C}{n}\Big(\|\varphi'\|_\infty+\sum_{\eta\in\bbZ}\big|\eta{\cal
    F}(\varphi)(\eta)\big|\Big).
\end{equation*}
\end{corollary}

	\subsection{The case of deterministic data}

        \label{xxx}

        In the present section we give an example of deterministic
        initial data, mentioned in  Remark \ref{rmk2.4}, that
        satisfy Assumption \ref{T}. To simplify the presentation we
        consider only the momenta components.
        
        Suppose that $T\in C(\bbT)$ is strictly positive. We show that
there exists a family of real sequences
$
(p_x^{(n)})_{x\in\bbT_n},
$
such that, for every $\varphi\in C(\bbT)$,
\begin{equation}
\label{eq:empirical-energy}
\lim_{n\to\infty}
\frac1n\sum_{x\in\bbT_n}
\big(p_x^{(n)}\big)^2
\varphi\left(\frac{x}{n}\right)
=
\int_{\bbT}\varphi(u)T(u)\,\dd u,
\end{equation}
and for some constant $C>0$
\begin{equation}
\label{eq:fourier-bound}
\sup_{k\in\widehat{\bbT}_n}
\big|\widehat p^{(n)}(k)\big|
\leq C\sqrt n,\qquad n=1,2,\ldots.
\end{equation}
Note that the above imply that conditions \eqref{eq:1} and
\eqref{eq:3z} are satisfied.

The construction is based on a discrepancy Theorem of Spencer in the
version, given in \cite[Theorem 7, p. 686]{spencer}.
\begin{theorem}
  \label{disc-thm}
  Suppose that   $V>0$. Let
  \[
L_j({\bf v}):=\sum_{m=1}^Na_{j,m}v_m,\quad
{\bf v}:=(v_1,\ldots,v_N)\in\bbR^N,\qquad j=1,\ldots, M,
\]
where $M\le N$ and $a_{j,m}\in[-V,V]$.
Then, there
exists a sequence $(\eps_1,\ldots, \eps_N)\in\{-1,1\}^N$ such that
\[
|L_j (\eps_1,\ldots, \eps_N)|\le C_{\rm U}V\sqrt{N},\qquad j=1,\ldots, M,
\]
where $C_{\rm U}>0$ is some universal constant.
\end{theorem}
We apply this theorem with $M=N=2n$. Set
\(
M_T:=\sqrt{\|T\|_\infty}.
\)
Define  the forms
$c(k,\cdot),s(k,\cdot)$ on $\bbR^{2n}$ by letting
\begin{equation}
\label{eq:vx}
\begin{split}
 & c(k,{\bf v})
=
\sum_{x=0}^{n-1} \sqrt{T(x/n)}
\cos(2\pi x k)v_x\\
&  s(k,{\bf v})
=
\sum_{x=0}^{n-1} \sqrt{T(x/n)}
\sin(2\pi x k) v_x, \qquad k \in\hT_n.
\end{split}
\end{equation}
By Spencer's Theorem for each $n$ there exists a vector of signs
\[
\varepsilon_x^{(n)} \in\{-1,1\},
\qquad x=0,\ldots,2n-1,
\]
such that
\begin{equation}
\label{eq:spencer-application}
\begin{split}
  & \bigg|
\sum_{x=0}^{n-1}
\varepsilon_x^{(n)} \sqrt{T(x/n)}
\cos(2\pi x  k)
 \bigg| 
\leq C_{\rm U}\sqrt{2n\|T\|_\infty} ,\\
&
\bigg|
\sum_{x=0}^{n-1}
\varepsilon_x^{(n)}  \sqrt{T(x/n)}
\sin(2\pi x k)
 \bigg| 
\leq C_{\rm U}\sqrt{2n\|T\|_\infty} 
\end{split}
\end{equation}
for all $k\in\hat\bT_n$. Now define
\begin{equation}
\label{eq:p-construction}
p_x^{(n)}
:=
\varepsilon_x^{(n)}
\sqrt{T(x/n)}.
\end{equation}
Then
 $
(p_x^{(n)})^2=T(\frac xn).
$
Consequently, \eqref{eq:empirical-energy} holds. 
 Estimate \eqref{eq:spencer-application} implies
that
\[
\big|
\hat p^{(n)}\left( k\right)
\big|
\leq
2C_{\rm U}\sqrt{n\|T\|_\infty} ,\qquad k\in\hat\bT_n,\,n=1,2,\ldots.
\]
We have therefore proved that sequences satisfying
\eqref{eq:empirical-energy}--\eqref{eq:fourier-bound} exist for
every strictly positive $T\in C(\bbT)$.


\section{Properties of the Laplace transform}

\label{secB}

 \subsection{Uniqueness of the Laplace transform}

Suppose that ${\frak W} \in {\frak L}'$. Given $\la\in\mathbb C$ such that ${\rm Re}\,\la>0$
we define its Laplace  transform $w(\la)\in {\cal M}(\bT)$ -- the
Banach space
of all real valued, signed measures of finite total bounded variation, equipped
with the total variation norm -- in the following
way: for a real valued $\varphi\in C(\bbT)$ we let
\begin{equation}
\label{062808}
  \int_{\bT}\varphi(u)w(\la,\dd u):= 
{\frak W}\big(F_{\la}(\varphi)\big),
\end{equation}
where $F_{\la}\in {\frak L}$ is given by the formula
$$
F_{\la}(t,u):=e^{-2\pi \la t}\varphi(u),\quad t\ge0.
$$
\begin{proposition}
\label{prop012808}
 Suppose that $\chi\in C_0^\infty(0,+\infty)$ and $\xi>0$. Let
 $\varphi\in C(\bbT)$ and
\begin{equation}
\label{082808}
F(t,u):=e^{-2\pi \xi t}\chi(t)  \varphi(u),\quad
(t,u)\in[0,+\infty)\times \bT.
\end{equation}
Then, 
\begin{equation}
\label{072808}
\int_{\bbR}\hat \chi(\eta) \dd\eta\int_{\bT}\varphi(u)w(\xi-i\eta,\dd u) =
{\frak W}\big(F\big).
\end{equation}
\end{proposition}
\proof
Using \eqref{062808} we can write that the left hand side of \eqref{072808} equals
\begin{equation}
\label{082808a}
\lim_{A\to+\infty}  
{\frak W}\big(G_A\big),
\end{equation}
where
\[
G_A(t,u):=
e^{-2\pi \xi t}\varphi(u)\int_{-A}^A\hat \chi(\eta) e^{2\pi i\eta t} \dd\eta.\]
It is easy to conclude that $\lim_{A\to+\infty}\|G_A-F\|_{{\frak L}}=0$,
and the assertion of the proposition follows.
\qed

As a corollary we conclude the following result.
\begin{theorem}
\label{thm012808}
Suppose that $w(\la)$, $v(\la)$ are the Laplace transforms of  ${\frak
  W},{\frak V} \in {\frak L}'$. Assume that
for any $M>0$ there exists $\la_0(M)$ such that
$
w(\xi) \equiv v(\xi) $, for $ \xi>\la_0.$
Then, ${\frak W}\equiv {\frak V}$. 
\end{theorem}
\proof
Suppose that $\varphi\in C(\bT)$. Since both functions
$$
\la\mapsto \int_{\bT}\varphi(u)w(\la,\dd u)\quad\mbox{and}\quad \la\mapsto \int_{\bT}\varphi(u)v(\la,\dd u),\quad {\rm Re}\,\la>0
$$
are holomorphic and agree on the half-line  ${\rm Re}\,\la>\la_0$,
${\rm Im}\,\la=0$ they agree on the entire half-space $ {\rm
  Re}\,\la>0$.
From Proposition \ref{prop012808} we conclude that functionals ${\frak
  W}$ and ${\frak V}$ agree on
functions $F$ of the form \eqref{082808}. 
A function 
\begin{equation}
\label{012311}
[0,+\infty)\times\bT \ni(t,u)\mapsto \mathbf{1}_{[a,b]}(t)\varphi(u),
\end{equation}
 where $0\le a<b$, can be 
approximated in the $\mL^1$ sense,
by functions of the form \eqref{082808}.
On the other hand, we can approximate  an arbitrary   function
$F\in C_c([0,+\infty)\times \bT)$  by functions of the form \eqref{012311}. This in turn allows us to
claim that such functions are dense in ${\frak L}$.
The above allows us to conclude the claim of the theorem.
\qed

\subsection{Proof of Theorem \ref{thm:intermediate}}
\label{AppC}

We prove here that Theorem \ref{thm:intermediate}  can be deduced
from Theorem \ref{thm:laplace}. In fact, in the present section we
consider a more general situation.  Let $({\cal X},{\cal A},\bbP)$ be  a
probability space, with $\bbE$ the respective expectation.
Suppose that $\big(\Xi_n(t)\big)_{t\ge0}$, $n=1,2,\ldots$, is a family
of  complex valued, measurable, stochastic
processes such that, there exist $C_W>0$ and $M>0$, for which
\begin{equation}
  \label{010109-26}
  \int_0^t \bbE\big[|\Xi_n(s)|\big]\dd s\le C_W(t+1)^M,\qquad t\ge0,\,n=1,2,\ldots.
\end{equation}
Thanks to \eqref{010109-26} for any $\la\in\mathbb C$ such that ${\rm
  Re}\,\la>0$, we can define the Laplace transforms of
the processes:
\begin{equation}
  \label{020109-26}
  \hat\Xi_n(\la):=\int_0^{+\infty}e^{-\la s}\Xi_n(s)\dd s.
\end{equation}
The improper integrals above are considered in the $\mL^1(\bbP)$ sense.

\begin{theorem}
  \label{thm010109-26}
  In addition to the assumptions made above suppose that
  \begin{equation}
  \label{030109-26}
  \lim_{n\to+\infty}\bbE\big[\big|\hat\Xi_n(\la)\big|\big]=0,\quad
  {\rm Re}\,\la>0.
\end{equation}
Then, for any  $\varphi \in C_c([0,+\infty))$ we have
\begin{equation}
  \label{040109-26}
  \lim_{n\to +\infty} \E \Bigg[ \bigg|  \int_0^{+\infty} \vphi(t) \Xi_n(t)\dd t \bigg| \Bigg] = 0.
\end{equation}
\end{theorem}
\proof
By a density argument we can assume that $\varphi \in
  C_c^\infty([0,+\infty))$ in \eqref{040109-26}.
Integrating by parts we infer that  the latter follows if we show
\begin{equation}
	\label{011103-26b}
	\lim_{n\to+\infty}  \bbE_n\Bigg[\bigg| \int_0^{+\infty}\varphi'(t)\Big( \int_0^t\Xi_n(s)\dd s\Big)\dd t \bigg|\Bigg]=0.
\end{equation}
Let $\psi:=\varphi' \in C_c^\infty([0,+\infty))$.

 Choose an arbitrary $\eps>0$.
Let $ G\in C^\infty_c(0,1)$ be given
by $G(e^{-t}):= \psi(t)$.   By virtue of the Weierstrass
approximation Theorem, for any $\eps>0$ there exists a polynomial $p$ such that
$$
\sup_{u\in[0,1]} | G'(u)-p(u)|<\frac{\eps }{100 (C_W+1)
  \gamma_M},\quad \mbox{where}\quad \gamma_M:=\int_0^{+\infty} (t+1)^Me^{-t}
	\dd t .
$$
Let
$
P(u)=\int_0^up(v)\dd v.
$
Thus,
\begin{equation}
	\label{021103-26}
	| G(u)-P(u)|\le \frac{\eps u}{100 (C_W+1) },\qquad u\in[0,1].
\end{equation}
As a result,  
\begin{align}
	\bbE_n\Bigg[\bigg|  \notag &\int_0^{+\infty}\psi(t) \Big( \int_0^t\Xi_n(s)\dd s\Big)\dd t
	\bigg|\Bigg]   =\bbE_n\Bigg[\bigg|\int_0^{+\infty} G(e^{-t}) \Big( \int_0^t\Xi_n(s)\dd s\Big)\dd t \bigg|\Bigg]\notag\\
	&
	\le \bbE_n\Bigg[\bigg| \int_0^{+\infty} \big[G(e^{-t})-P(e^{-t})\big] \Big( \int_0^t\Xi_n(s)\dd s\Big)\dd t \bigg|\Bigg] \label{eq:firstterm}\\
	&
	\quad +
	 \bbE_n\Bigg[\bigg| \int_0^{+\infty}  P(e^{-t})  \Big( \int_0^t\Xi_n(s)\dd s\Big)\dd t \bigg|\Bigg]. \label{eq:secondterm}
\end{align}
Since $P$ is a polynomial such that $P(0)=0$ we can write it as
$P(t)=\sum_{m=1}^Ma_mt^m$ for some integer $M\ge 1$ and $a_1,\ldots,a_m\in\bbR$.
Therefore, integrating by parts in $t$, we get
\begin{align*}
\bbE_n\Bigg[ \bigg|\int_0^{+\infty}  P(e^{-t})  \Big(
  \int_0^t\Xi_n(s)\dd s\Big)\dd t \bigg|\Bigg] =\bbE_n\Bigg[
  \bigg|\sum_{m=1}^M\frac{a_m}{m} \hat \Xi_n(m)\bigg|\Bigg] .
\end{align*}
Here we have used the fact that, thanks to \eqref{010109-26}, for each $n$ fixed and $\xi>0$
we have 
\[\lim_{t\to+\infty} e^{-\xi t}\bbE_n\Bigg[ \bigg| \int_0^t
  \Xi_n(s)\dd s \bigg|\Bigg]  = 0 .\]

We can use \eqref{030109-26} and conclude  that the
second term \eqref{eq:secondterm} vanishes, as $n\to+\infty$. Therefore we infer that
\begin{multline*}
\limsup_{n\to+\infty}\bbE_n\Bigg[\bigg| \int_0^{+\infty}\psi(t) \Big( \int_0^t\Xi_n(s)\dd s\Big)\dd t
	\bigg|\Bigg] \\
	\le  \limsup_{n\to+\infty} \bbE_n\Bigg[\bigg| \int_0^{+\infty} \big[G(e^{-t})-P(e^{-t})\big] \Big( \int_0^t\Xi_n(s)\dd s\Big)\dd t \bigg|\Bigg] .
\end{multline*}
Furthermore, by \eqref{021103-26} 
\begin{align*}
	\limsup_{n\to+\infty}\bbE_n\Bigg[\bigg|&\hsum\int_0^{+\infty}\psi(t) \Big( \int_0^t\Xi_n(s)\dd s\Big)\dd t
	\bigg|\Bigg] \\
	&
	\le  \frac{\eps }{100 (C_W+1)  }
	\limsup_{n\to+\infty} \bbE_n\Bigg[ \int_0^{+\infty} e^{-t}
	\int_0^t \Big|\Xi_n(s)\Big|\dd s\dd t \Bigg]  \\
	&
	\le  \frac{\eps C_W}{100 (C_W+1) }
	\int_0^{+\infty} (t+1)^Me^{-t}
	\dd t  <\eps  .
\end{align*}
We have used \eqref{010109-26} to claim the penultimate estimate.
Since $\eps>0$ has been chosen arbitrarily we conclude
\eqref{040109-26}. 
\qed

  \section{Proof of Theorem \ref{thm022303-26} } 
  
  \label{sec5}
  
 {We prove here Theorem \ref{thm022303-26} on the convergence of the mean energy profile. The proof  is a straightforward adaptation of the argument presented in \cite{KOS18} for the unpinned case.}
  
  Define the real Banach space ${\frak C}$ as the completion of the space
  $C_c^\infty\big(\T\times[0,+\infty))$ of smooth functions $\Phi(u,t)$ which are compactly supported w.r.t.~the time variable $t$, under the norm
  \begin{equation}
  	\label{fC}
  	\|\Phi\|_{{\frak C}}:=\int_0^{+\infty}\sup_{u\in\bT}|\Phi(u,t)| \; \dd t.
  \end{equation}
  Consider the linear functional   $\bar{\cal E}_n:{\frak C}\to\mathbb R$
  given by
  \begin{equation}
  	\label{bE}
  	\bar{\cal E}_n[\Phi]:=\frac{1}{n}\sum_{x\in\bT_n}\int_0^{+\infty}\bar
  	e_x^{(n)}(t)\Phi\Big(\frac{x}{n},t\Big)\dd t,\qquad \Phi\in C_c^\infty\big(\T\times[0,+\infty)),
      \end{equation}
      with $\bar
  	e_x^{(n)}(t)$ given in \eqref{bex1}.
 Thanks to \eqref{eq:energybounds}
  the functionals extend continuously to
  ${\frak C}$ and their dual norms satisfy
  \begin{equation}
  	\label{bEbound}
  	\limsup_{n\to+\infty}\|\bar{\cal E}_n\|_{{\frak C}'}\le \sup_{n\ge 1} \E_n\big[ \mathcal{E}_{\rm tot}^{(n)}(0) \big] < \infty
  	.
  \end{equation}
  Since ${\frak C}$ is separable the above implies that the sequence
  $\big(\bar{\cal E}_n\big)$ is $\star$-weakly sequentially compact in
  the dual space
  ${\frak C}'$.
  
  %
  {Recall from Section \ref{sec:wig} the notations for the} averaged Fourier-Wigner functions
  	\begin{equation}
  		\label{030202-26}
  		\begin{split}
  			&\overline{W}_{n,\iota}(t,\eta,k) :=\bbE_n \big[W_{n,\iota}(t,\eta,k) \big] ,\qquad \iota\in\{-,0,+\}\\
  			&\delta\bar W_{n}(t,\eta,k) :=\bbE_n \big[\delta W_{n}(t,\eta,k)\big] ,\\ 
  			&  \bar U_{n,\pm}(t,\eta,k) :=\bbE_n \big[U_{n,\pm}(t,\eta,k)  \big].
  		\end{split}
  	\end{equation}
  	Now let us define, for any $\Phi\in C_c^\infty\big(\T\times[0,+\infty))$,
  	\begin{equation}
  		\label{032503-26}
  		{\bar{\cal W}}_n[\Phi]:=\sum_{\eta\in{\bbZ}}\underset{k\in
  			\hT_n}{\hat{\sum}} \int_0^{+\infty}\bar W_{n,+}(t,\eta,k) {\cal F}(\Phi)^\star(\eta,t)\dd t .
  	\end{equation}
  	Here ${\cal F}(\Phi)(\eta,t)$ denote the Fourier transform
        with respect to the first variable of $\Phi$, see
        \eqref{FC}. Since $\Phi\in
          C_c^\infty\big(\T\times[0,+\infty))$,  \eqref{032503-26} is well defined. Note also that the functional takes real values for real
  	valued $\Phi$.
  	The   argument used in the foregoing shows also that 
  	$({\bar{\cal W}}_n)_{n\ge1}$
  	forms   $\star$-weakly sequentially compact in
  	${\frak C}'$ and 
  	Corollary \ref{cor032603-26} shows that
  	\begin{equation}
  		\label{bWE}
  		\lim_{n\to+\infty}\big({\bar{\cal W}}_n[\Phi]-\bar{\cal E}_n[\Phi]\big)=0,\qquad \Phi\in C_c^\infty\big(\T\times[0,+\infty)).
  	\end{equation}
  	Therefore, in order to prove  Theorem \ref{thm022303-26} it suffices
  	only to show that
  	\begin{equation}
  		\label{012503-26}
  		\lim_{n\to+\infty} {\bar{\cal W}}_n[\Phi] ={\cal T}[\Phi],\qquad \Phi\in
  		C_c^\infty\big(\T\times[0,+\infty)), \end{equation} where 
  	\begin{equation}
  		{\cal T}[\Phi]:=\int_0^{+\infty}\dd t\int_{\bT}\Phi(u,t)T(u,t)\dd u
  		\label{eq:defTphi}
  	\end{equation}
  	and $T(u,t)$ is solution to \eqref{032303-26} as in the statement of the theorem.

  	\subsection{Averaged Laplace-Fourier-Wigner functions. The limit identification}

 Similarly to \eqref{lwu}, define the Laplace transforms of the mean Fourier-Wigner functions
  	\begin{equation}
  		\label{010302-26}
  		\begin{split}
  			&\bar w_{n,\iota}(\la,\eta,k):=\int_0^{+\infty}e^{-\la t}\bar
  			W_{n,\iota}(t,\eta,k) \dd t ,\qquad \iota\in\{-,0,+\}\\
  			&\delta\bar w_{n}(\la,\eta,k):=\int_0^{+\infty}e^{-\la
  				t}\delta\bar W_{n}(t,\eta,k)\dd t ,\\ 
  			&  \bar u_{n,\pm}(\la,\eta,k):=\int_0^{+\infty}e^{-\la t}\bar
  			U_{n,\pm}(t,\eta,k) \dd t,
  		\end{split}
  	\end{equation}
  	for any complex $\la$ such that ${\rm Re}\,\la>0$.
  We  obtain {(recall \eqref{010704-26z}--\eqref{010704-26c})} that the average Fourier-Wigner functions satisfy
  	\begin{align*}
  		&
  		\frac{\dd}{\dd t}\bar  W_{n,+}(t,\eta,k)      
  		=-in \delta_{n}\om(\eta,k) \bar  W_{n,+}(t,\eta,k)    +2\ga n^2 {\cal
  			L}_n\big(\bar  W_{n,+} -\bar U_{n,+} \big) (t,\eta,k)  , \\
  		&
  		\frac{\dd}{\dd t}\bar  W_{n,-}(t,\eta,k)    =in \delta_{n}\om(\eta,k) \bar W_{n,-}(t,\eta,k)    
  		+2\ga n^2{\cal L}_n\big(\bar  W_{n,-} -\bar  U_{n,+}
  		\big)(t,\eta,k)  , \\
  		&
  		\frac{\dd}{\dd t}\bar  U_{n,+}(t,\eta,k)     =2n^2 \bar\om_n(\eta,k)
  		\bar  U_{n,-}(t,\eta,k)  
  		+2\ga n^2{\cal L}_n\big(\bar U_{n,+} 
  		-\bar W_{n,0}    \big)
  		(t,\eta,k) ,    \\
  		&
  		\frac{\dd}{\dd t}\bar  U_{n,-}(t,\eta,k)     = -2n^2 \bar\om_n(\eta,k)
  		\bar  U_{n,+}(t,\eta,k)  -2\ga n^2 \bar  U_{n,-}(t,\eta,k)  . 
  	\end{align*}
  	Taking the Laplace transform of the both sides of
  	equations we get
  	\begin{equation}
  		\label{020302-26} \begin{split}
  			\la\bar  w_{n,+}(\la,\eta,k)  -  \bar  W_{n,+}(0,\eta,k)    
  			= & -in \delta_{n}\om(\eta,k) \bar  w_{n,+}(\la,\eta,k)   \\ & \; +2\ga n^2 {\cal
  				L}_n\big(\bar  w_{n,+}-\bar u_{n,+}\big) (\la,\eta,k)  , \vphantom{\bigg)}\\
  			\la\bar  w_{n,-}(\la,\eta,k)   -  \bar  W_{n,-}(0,\eta,k)     = & \; in \delta_{n}\om(\eta,k) \bar w_{n,-}(\la,\eta,k)\\   
  			& +2\ga n^2{\cal L}_n\big(\bar  w_{n,-}-\bar  u_{n,+}
  			\big)(\la,\eta,k)  ,  \vphantom{\bigg)}\\
  			\la\bar  u_{n,+}(\la,\eta,k)  -\bar  U_{n,+}(0,\eta,k)     = & \: 2n^2 \bar\om_n(\eta,k)
  			\bar  u_{n,-}(\la,\eta,k)  
  			\\ & +2\ga n^2{\cal L}_n\big(\bar  u_{n,+}
  			- \bar  w_{n,0}\big) 
  			(\la,\eta,k) ,  \vphantom{\bigg)}   \\
  			\la\bar  u_{n,-}(\la,\eta,k)   -\bar  U_{n,-}(0,\eta,k)    = & \;  -2n^2 \bar\om_n(\eta,k)
  			\bar  u_{n,+}(\la,\eta,k)  -2\ga n^2\bar  u_{n,-}(\la,\eta,k)  
  			. \end{split}
  	\end{equation}
  	Thanks to Theorem \ref{thm012808} in order to prove \eqref{012503-26}
  	it suffices to show the following.
  	\begin{theorem}
  		\label{thm010402-26}
  		For any  smooth real-valued function $\varphi:\bbT\times
  		\bbT\to\mathbb C$ and $\la\in\mathbb C$ such that ${\rm Re}\,\la>0$ we have
  		\begin{equation*}
  			\lim_{n\to+\infty}\sum_{\eta\in\bbZ}\underset{k\in\bbT_n}{\hat\sum}\bar
  			w_{n,+}(\la,  \eta,k) {\cal F}(\varphi)^\star(\eta,k)=
  			\sum_{\eta\in\bbZ}\frac{{\cal F}(T)(\eta)}{ \la +
  				\mathbf{D} (2\pi \eta)^2 }\int_{\bbT} {\cal F}(\varphi)^\star(\eta,u)\dd u,
  		\end{equation*}
  		where ${\cal F}(T)$ is the Fourier transform of $T$ defined as in \eqref{FC}
  		and $	\mathbf{D}$ is as in the statement of Theorem \ref{thm022303-26} and is defined in \eqref{042303-26}.
  	\end{theorem}
  	The remaining part of the present section is devoted to the proof of
  	Theorem \ref{thm010402-26}.

  	\subsection{Estimates of the averaged Laplace-Fourier-Wigner functions}
  	
  	Taking the scalar products, in the $k$ variable, of each of the equation by $\bar  w_{n,+}$,
  	$\bar  w_{n,-}$, $\bar  u_{n,+}$ and $\bar  u_{n,-}$ respectively and
  	then summing the resulting equations sideways we end up with the
  	following identity (recall \eqref{072302-26aa}): for all $\eta\in\bbZ$, ${\rm Re}\,\la>0$,
  	\begin{equation} 
  		\label{010402-26} \begin{split}
  			&
  			\frac \la  2\sum_{\iota\in\{-,+\}}\|\bar
  			w_{n,+}(\la,\eta)\|^2_{\mb L^2(\hT_n)}+\la \sum_{\iota\in\{-,+\}}\|\bar
  			u_{n,\iota}(\la,\eta) \|^2_{\mb L^2(\hT_n)} +2\ga n^2    \|\delta\bar  w_{n}(\la,\eta)\|^2_{\mb L^2(\hT_n)}
  			\\
  			& + \ga n^2  \underset{k \in\hat \T_n}{\hat{\sum}} {\cal
  				D}_n\Big(\big(\bar  w_{n,0}-\bar u_{n,+}\big) (\la,\eta) \Big)
  			+2\ga n^2 \|\bar  u_{n,-}(\la,\eta)  \|^2_{\mb L^2(\hT_n)} 
  			\\
  			&
  			=  \frac12 \sum_{\iota=\pm}\big\langle\bar
  			W_{n,\iota}(0,\eta), \bar  w_{n,\iota} (\la,\eta)\big\rangle_{\mb L^2(\hT_n)}     +\sum_{\iota=\pm}\langle\bar  U_{n,\iota}(0,\eta), \bar  u_{n,\iota}(\la,\eta) \rangle_{\mb L^2(\hT_n)}     \\
  			&
  			\quad  +\frac{in}{2} \underset{k \in\hat \T_n}{\hat{\sum}} \delta_{n}\om(\eta,k)\big( |\bar
  			w_{n,-}(\la,\eta,k)|^2 -|\bar  w_{n,+}(\la,\eta,k)|^2 \big)
  			\\ & \quad
  			-4n^2 i\underset{k \in\hat \T_n}{\hat{\sum}}\bar\om_n(\eta,k)
  			{\rm Im}\, \big(\bar  u_{n,+}(\la,\eta,k)  u_{n,-}^\star(\la,\eta,k)\big),
  		\end{split}
  	\end{equation}
  	where {$\cD_n$ has been defined in \eqref{eq:cD}}.
  	Using the above and taking the real parts of both sides of
  	\eqref{010402-26} we obtain the following identity
  	\begin{equation}\label{020402-26}
  		\begin{split}
  			&
  			\frac {{\rm Re}\,\la}{  2}\sum_{\iota\in\{-,+\}}\|\bar
  			w_{n,+}(\la,\eta)\|^2_{\mb L^2(\hT_n)}+{\rm Re}\,\la \sum_{\iota\in\{-,+\}}\|\bar
  			u_{n,\iota}(\la,\eta) \|^2_{\mb L^2(\hT_n)} +2\ga n^2    \|\delta\bar  w_{n}(\la,\eta)\|^2_{\mb L^2(\hT_n)}
  			\\
  			&+\ga n^2  \underset{k\in\hat\T_n}{\hat{\sum}} {\cal
  				D}_n\Big(\big(\bar  w_{n,0}-\bar u_{n,+}\big) (\la,\eta) \Big)
  			+2\ga n^2 \|\bar  u_{n,-}(\la,\eta)  \|^2_{\mb L^2(\hT_n)}\\
  			&
  			= \underset{k\in\hat\T_n}{\hat{\sum}} {\rm
  				Re}\,\Big( \frac12 \sum_{\iota=\pm}\big\langle\bar
  			W_{n,\iota}(0,\eta), \bar  w_{n,\iota} (\la,\eta)\big\rangle_{\mb L^2(\hT_n)}     +\sum_{\iota=\pm}\big\langle\bar  U_{n,\iota}(0,\eta), \bar  u_{n,\iota}(\la,\eta) \big\rangle_{\mb L^2(\hT_n)}    \Big)  .
  	\end{split}\end{equation}

      {It is a consequence of Proposition \ref{prop030204-26} that
  		\begin{equation}
  			\label{040402-26}    \bar {\frak W}_2:= {\sup_{n\ge 1}\sup_{\eta\in\bbZ}} \bigg\{\sum_{\iota=\pm} \Big(\big\|\bar
  			W_{n,\iota}(\eta)\big\|_{\mb L^2(\hT_n)}
  			+\big\|\bar  U_{n,\iota} (\eta)\big\|_{\mb L^2(\hT_n)} \Big)\bigg\} <+\infty.
  		\end{equation}
  	From \eqref{020402-26} and \eqref{040402-26}  we conclude  in particular the
  	following.}
  	\begin{proposition}
  		\label{prop010402-26}
  	For any $\xi_0>0$ there exists a constant $C>0$ such that
  		for all ${\rm
  			Re}\,\la\ge \xi_0$ and $n\ge 1$,
  		\begin{multline}
  			\label{030402-26}
  			\sum_{\iota=\pm} {\sup_{\eta \in \bbZ}\Big\{\big\|\bar
  			w_{n,\iota}(\la,\eta)\big\|^2_{\mL^2(\hT_n)} 
  			+\big\|\bar  u_{n,\iota} (\la,\eta)\big\|^2_{\mL^2(\hT_n)}\Big\}} 
  			+  \ga n^2 \sup_{\eta\in \bbZ}{\cal D}\Big(\big(\bar  w_{n,0} -\bar u_{n,+} \big)
  			(\la,\eta )\Big) \\ +
  			2\ga n^2\;  {\sup_{\eta \in \bbZ}}\Big\{\big\|\delta\bar  w_{n}(\la,\eta)\big\|_{\mL^2(\hT_n)}^2 +  \big\|\bar  u_{n,-}(\la,\eta) \big\|_{\mL^2(\hT_n)}^2 \Big\} \le C.
  		\end{multline}
  		Furthermore, we also have
  		\begin{equation}
  			\label{060402-26a}
  		{\sup_{\eta \in \bbZ}}\;	\big\|\bar  u_{n,+}(\la,\eta) \big\|_{\mL^2(\hT_n)}^2
  			\le \frac{C}{n^2}.
  		\end{equation}
              \end{proposition}

  	\proof
  	In light of \eqref{020402-26} only estimate  \eqref{060402-26a}
  	requires a proof. To conclude it 
  	multiply   \eqref{020302-26} scalarly (in $k$) by $\bar
  	u_{n,+} (\la,\eta)$. We get
  	\begin{multline*}
  		(\la +2\ga n^2) \big\langle \bar  u_{n,-}(\la,\eta), \bar
  		u_{n,+} (\la,\eta)\big\rangle_{\mb L^2(\hT_n)} - \big\langle\bar  U_{n,-}(0,\eta) , \bar
  		u_{n,+} (\la,\eta)\big\rangle_{\mb L^2(\hT_n)}   \\
  		=2n^2 \underset{k\in\hat\T_n}{\hat{\sum}} \bar\om_n(\eta,k)
  		|\bar  u_{n,+}(\la,\eta,k)|^2.
  	\end{multline*}
  	Since    $\bar\om_n(\eta,k)\ge \om_0$ we have
  	\begin{multline}
  		\label{070402-26}
  		(|\la |+2\ga n^2)\big\|\bar  u_{n,-}(\la,\eta)\big\|_{\mb L^2(\hT_n)}\big\| \bar
  		u_{n,+}(\la,\eta) \big\|_{\mb L^2(\hT_n)} \\
  		+\big\|\bar  U_{n,-}(0,\eta) \big\|_{\mb L^2(\hT_n)}\big\| \bar
  		u_{n,+}(\la,\eta) \big\|_{\mb L^2(\hT_n)}   
  		\ge n^2\om_0  
  		\big\|\bar  u_{n,+}(\la,\eta) \big\|_{\mb L^2(\hT_n)} ^2 \vphantom{\bigg(}
  		.
  	\end{multline}
  	Hence
  	\begin{align*}
  		&
  		\Big(\frac{|\la |}{n^2}+\ga\Big) \big\|\bar  u_{n,-}(\la,\eta)\big\|_{\mb L^2(\hT_n)}
  		+\frac{1}{n^2}\big\|\bar  U_{n,-}(0,\eta) \big\|_{\mb L^2(\hT_n)} 
  		\ge  \om_0  
  		\big\|\bar  u_{n,+}(\la,\eta) \big\|_{\mb L^2(\hT_n)}     .
  	\end{align*}
  	Estimate \eqref{060402-26a} follows then from the already proved estimate \eqref{030402-26}.
  	\qed
  	
  	\subsection{The end of the proof of Theorem \ref{thm010402-26}} 
  	
  	From \eqref{020302-26} we obtain
  	\begin{align}
  		\label{080402-26c}
  		\bar  w_{n,+}(\la,\eta,k)  
  		=  &\;  \frac{\bar  W_{n,+}(\eta,k)}{\la +in \delta_{n}\om(\eta,k) +\ga n^2 }      +\frac{
  			\ga n^2  \hat{\sum}_{\ell}    \bar     w_{n,+}(\la,\eta,\ell)}{\la +in \delta_{n}\om(\eta,k) +\ga n^2 } \notag \\
  		&
  		-\frac{ n^2 }{\la +in \delta_{n}\om(\eta,k) +\ga n^2 } {\cal
  			L}_n \bar u_{n,+}(\la,\eta,k)   .
  	\end{align}
  	Using Proposition \ref{prop010402-26}  we obtain
  	the following.
  	\begin{proposition}
  		\label{prop042603-26}
  		We have
  		\begin{align}
  			\label{110402-26}
  			\bar  w_{n,+}(\la,\eta,k)  
  			=   \underset{\ell \in\hat\T_n}{\hat{\sum}}    \bar     w_{n,+}(\la,\eta,\ell)+ o_n(\la,\eta,k) ,
  		\end{align}
  		where $
  			\lim_{n\to+\infty} {\sup_{\eta\in\bbZ} \| o_n(\la,\eta,\cdot) \|_{\mL^1(\hT_n)}}=0.$
  	\end{proposition}
  	Summing up over $k$ of both sides of \eqref{080402-26c} we get
  	\begin{align}
  		\label{100402-26}
  		&
  		\underset{\ell \in\hat\T_n}{\hat{\sum}}      \bar     w_{n,+}(\la,\eta,\ell)  
  		=   \bD_n^{-1}(\la,
  		\eta) \underset{k\in\hat\T_n}{\hat{\sum}}   \frac{\bar  W_{n,+}(\eta,k)}{\frac{\la}{n^2}
  			+\frac{i \delta_{n}\om(\eta,k)}{n} +\ga  }   +
  		R_n(\la,\eta),\end{align}
  	where \begin{align*}
  		\bD_n(\la,
  		\eta)&:=n^2 \underset{k \in\hat\T_n}{\hat{\sum}} \Big(1-\frac{\ga 
  		}{\frac{\la}{n^2}  +\frac{i \delta_{n}\om(\eta,k)}{n} +\ga  }\Big),\\
  		R_n(\la,\eta)&:=\frac{\bD_n^{-1}(\la,
  			\eta)}{\ga}  \underset{k \in\hat\T_n}{\hat{\sum}}  \frac{  \la+in
  			\delta_{n}\om(\eta,k)    }{\frac{\la}{n^2} +\frac{i \delta_{n}\om(\eta,k)}{n} +\ga  } {\cal
  			L}_n \bar u_{n,+}(\la,\eta,k)  .\notag
  	\end{align*}
  	Thanks to Corollary \ref{cor032603-26} we have
  	\[
  	\lim_{n\to+\infty} \underset{k\in\hat\T_n}{\hat{\sum}}    \bar
  	W_{n,+}(\eta,k)=\lim_{n\to+\infty}\frac1n\sum_{x\in\bbT_n}e^{-2\pi i\eta\frac{x}{n}}\int_{\Om_n}e_x\dd\mu_n={\cal F}(T)(\eta).
  	\]
  	The conclusion of Theorem \ref{thm010402-26} follows then from
  	Proposition \ref{prop042603-26}, Lemma  \ref{lem:D} and Lemma
  	\ref{lm012603-26} proved below, as we now explain.

  	We start with the following result, similar to Lemma \ref{lm010303-26}.
  	\begin{lemma}
  		\label{lm010303-26bis}
  		For any $\eta\in\bbZ$, $\la$ s.t.~${\rm Re}\,\la>0$ we have
  		\begin{equation}
  			\label{010303-26bis}
  			\underset{k\in\hat\T_n}{\hat{\sum}} \Big[\om\Big(k+\frac{\eta}{n}\Big)-\om(k)\Big]  {\cal
  				L}_n \bar u_{n,+}(\la,\eta,k) =0.
  		\end{equation}
  	\end{lemma}
  	\proof
  The proof is identical to the one of Lemma \ref{lm010303-26}.   	\qed

  	\begin{lemma}
  		\label{lm012603-26}
  		For any  $\eta\in\bbZ$,  $\la$ s.t.~${\rm Re}\,\la>0$, 
  		\begin{equation}
  			\label{052603-26}
  			\lim_{n\to+\infty}R_n(\la,  \eta)=0.
  		\end{equation}
  	\end{lemma}
  	\proof
  	We write
  	\begin{align*}
  		&
  		R_n(\la,\eta)=R_n^{(1)}(\la,\eta)+R_n^{(2)}(\la,\eta),\quad\mbox{where}\\
  		&
  		R_n^{(1)}(\la,\eta)= \frac{\bD_n^{-1}(\la,
  			\eta)}{\ga} \underset{k\in\hat\T_n}{\hat{\sum}} \frac{  \la+in
  			\delta_{n}\om(\eta,k)    }{\frac{\la}{n^2} +\ga  } {\cal
  			L}_n \bar u_{n,+}(\la,\eta,k), \\
  		&
  		R_n^{(2)}(\la,\eta)= \frac{\bD_n^{-1}(\la,
  			\eta)}{\ga} \underset{k\in\hat\T_n}{\hat{\sum}} \frac{  
  			[\delta_{n}\om(\eta,k)]^2 -  \la\frac{i \delta_{n}\om(\eta,k)}{n} }{\big[\frac{\la}{n^2} +\frac{i \delta_{n}\om(\eta,k)}{n} +\ga\big] \big[\frac{\la}{n^2} +\ga \big] } {\cal
  			L}_n \bar u_{n,+}(\la,\eta,k).
  	\end{align*}
  	Using \eqref{060402-26a} we conclude that there exists $C>0$ such that
  	\[
  	\sup_{\eta\in\bbZ}|R_n^{(2)}(\la,\eta)|\le\frac{C}{n}.
  	\]
  	Thanks to Lemma \ref{lm010303-26bis} we obtain
  	\begin{align*}
  		R_n^{(1)}(\la,\eta)&= \frac{\bD_n^{-1}(\la,
  			\eta)}{\ga} \underset{k\in\hat\T_n}{\hat{\sum}} \frac{  \la+in
  			\delta_{n}\om(\eta,k)    }{\frac{\la}{n^2} +\ga  } {\cal
  			L}_n \bar u_{n,+}(\la,\eta,k)\\
  		&
  		= \frac{\bD_n^{-1}(\la,
  			\eta)}{\ga}\underset{k\in\hat\T_n}{\hat{\sum}} \frac{  \la    }{\frac{\la}{n^2} +\ga  } {\cal
  			L}_n \bar u_{n,+}(\la,\eta,k).
  	\end{align*}
  	Hence, also
  	\[
  	\sup_{\eta\in\bbZ}|R_n^{(1)}(\la,\eta)|\le\frac{C}{n}.
  	\]
  	Thus \eqref{052603-26} follows.\qed

\bibliographystyle{amsalpha}

\end{document}